\documentclass[11pt,final]{amsart}
\usepackage[utf8]{inputenc}
\usepackage[T1]{fontenc}
\usepackage[letterpaper,centering]{geometry}
\usepackage{lmodern}
\usepackage{eucal}
\usepackage{mathrsfs}
\usepackage{enumerate}
\usepackage{amsmath,amssymb,amsfonts}
\usepackage{xcolor}
\definecolor{darkblue}{rgb}{0,0,0.3}
\definecolor{darkgreen}{rgb}{0,0.4,0}
\usepackage[colorlinks=true,linkcolor=darkblue, citecolor=darkgreen, unicode,pdfborder={0 0 0}]{hyperref}
\usepackage[ps,arrow,matrix,tips,line, curve]{xy}
{\setbox0\hbox{$ $}}\fontdimen16\textfont2=\fontdimen17\textfont2
\entrymodifiers={+!!<0pt,\the\fontdimen22\textfont2>}
\SelectTips{cm}{11}
\usepackage{enumitem}
\setlist[enumerate]{label={\upshape(\arabic*)},topsep=.7ex, leftmargin=*}
\setlist[itemize]{leftmargin=*}

\usepackage{ifdraft}
\ifdraft{
\usepackage[notcite,notref]{showkeys}
\addtolength{\hoffset}{1.5cm}

}{}

\usepackage{etoolbox}

\theoremstyle{plain}
\newtheorem{step}{Step}
\newtheorem{thm}{Theorem}[section]
\newtheorem{conj}[thm]{Conjecture}
\newtheorem*{conjF}{Conjecture~$\cF$}
\newtheorem*{conjFplus}{Conjecture~$\cFplus$}
\newtheorem*{conjFconst}{Conjecture~$\cFconst$}

\newtheorem{lem}[thm]{Lemma}
\newtheorem{cor}[thm]{Corollary}
\newtheorem{prop}[thm]{Proposition}

\theoremstyle{definition}
\newtheorem{question}[thm]{Question}

\newtheorem{defn}[thm]{Definition}

\newtheorem{rmk}[thm]{Remark}
\newtheorem{rmks}[thm]{Remarks}
\newtheorem{example}[thm]{Example}

\numberwithin{equation}{section}

\input cyracc.def
\DeclareFontFamily{U}{russian}{}
\DeclareFontShape{U}{russian}{m}{n}
        { <5><6> wncyr5
        <7><8><9> wncyr7
        <10><10.95><12><14.4><17.28><20.74><24.88> wncyr10 }{}
\DeclareSymbolFont{Russian}{U}{russian}{m}{n}
\DeclareSymbolFontAlphabet{\mathcyr}{Russian}
\makeatletter
\let\@math@cyr\mathcyr
\renewcommand{\mathcyr}[1]{\@math@cyr{\cyracc #1}}
\makeatother
\newcommand{\Sha}{{\mathcyr{Sh}}}

\newcommand{\Frob}{\mathrm{Fr}}
\newcommand{\ab}{{\mathrm{ab}}}
\newcommand{\rank}{{\mathrm{rank}}}

\newcommand{\isoto}{\myxrightarrow{\,\sim\,}}
\newcommand{\isotoleft}{\myxleftarrow{\,\sim\,}}
\makeatletter
\def\myrightarrow{{\setbox\z@\hbox{$\rightarrow$}\dimen0\ht\z@\multiply\dimen0 6\divide\dimen0 10\ht\z@\dimen0\box\z@}}
\def\myrightarrowfill@{\arrowfill@\relbar\relbar\myrightarrow}
\def\myleftarrow{{\setbox\z@\hbox{$\leftarrow$}\dimen0\ht\z@\multiply\dimen0 6\divide\dimen0 10\ht\z@\dimen0\box\z@}}
\def\myleftarrowfill@{\arrowfill@\myleftarrow\relbar\relbar}
\newcommand{\myxrightarrow}[2][]{\ext@arrow 0359\myrightarrowfill@{#1}{#2}}
\newcommand{\myxleftarrow}[2][]{\ext@arrow 3095\myleftarrowfill@{#1}{#2}}
\makeatother
\newcommand{\mtilde}{{\mathchoice
    {\widetilde{m}}
    {\widetilde{m}}
    {\rlap{$\scriptscriptstyle{m}$}\vphantom{\raise0pt\hbox{$m$}}\smash{\lower2.5pt\hbox{$\scriptscriptstyle\widetilde{\phantom{\scriptscriptstyle{m}}}$}}}
    {\rlap{$\scriptscriptstyle{m}$}\vphantom{\raise.2pt\hbox{$m$}}\smash{\lower2.05pt\hbox{$\scriptscriptstyle\widetilde{\phantom{\scriptscriptstyle{m}}}$}}}}}
\newcommand{\ctilde}{{\widetilde{c}\mkern1.1mu}}

\newcommand{\Mtilde}{{\mathchoice
    {\rlap{$M$}\mkern1mu\smash[b]{\lower.5pt\hbox{$\widetilde{\phantom{M}}$}}\mkern-1mu}
    {\rlap{$M$}\mkern1mu\smash[b]{\lower.5pt\hbox{$\widetilde{\phantom{M}}$}}\mkern-1mu}
    {\rlap{$\scriptstyle{M}$}\mkern1mu\smash[b]{\lower.5pt\hbox{$\widetilde{\phantom{\scriptstyle{M}}}$}}\mkern-1mu}
    {\widetilde{M}}}}
\newcommand{\etabar}{{\bar\eta}}

\newcommand{\et}{{\text{ét}}}

\newcommand{\sE}{{\mathscr E}}
\newcommand{\sF}{{\mathscr F}}
\newcommand{\sO}{{\mathscr O}}
\newcommand{\sOint}{{\mathcal O}}

\newcommand{\sP}{{\mathscr P}}

\newcommand{\sU}{{\mathscr U}}

\newcommand{\sW}{{\mathscr W}}
\newcommand{\sX}{{\mathscr X}}
\newcommand{\sY}{{\mathscr Y}}
\newcommand{\A}{{\mathbf A}}

\ifdef{\C}%
{\renewcommand{\C}{{\mathbf C}}}%
{\newcommand{\C}{{\mathbf C}}}

\newcommand{\cF}{\mathrm F}

\newcommand{\cFplus}{\mathrm F_+}
\newcommand{\cFconst}{\mathrm F_\const}

\newcommand{\F}{{\mathbf F}}

\renewcommand{\P}{{\mathbf P}}
\newcommand{\Q}{{\mathbf Q}}

\newcommand{\Z}{{\mathbf Z}}

\newcommand{\nr}{\mathrm{nr}}
\newcommand{\const}{\mathrm{const}}

\newcommand{\Gm}{\mathbf{G}_\mathrm{m}}

\newcommand{\Gal}{\mathrm{Gal}}

\newcommand{\Pic}{\mathrm{Pic}}
\newcommand{\Br}{\mathrm{Br}}
\newcommand{\Spec}{\mathrm{Spec}}
\newcommand{\Div}{\mathrm{Div}}
\renewcommand{\phi}{\varphi}
\renewcommand{\emptyset}{\varnothing}

\newcommand{\Hom}{{\mathrm{Hom}}}

\newcommand{\mmu}{\boldsymbol{\mu}}

\newcommand{\chapeau}{{\rlap{\smash{\hbox{\lower4pt\hbox{\hskip1pt$\widehat{\phantom{u}}$}}}}}}
\newcommand{\Picplushat}{\Pic_+^{{\smash{\hbox{\lower4pt\hbox{\hskip0.4pt$\widehat{\phantom{u}}$}}}}}}
\newcommand{\PicplusAhat}{\Pic_{+,\A}^{{\smash{\hbox{\lower4pt\hbox{\hskip.4pt$\widehat{\phantom{u}}$}}}}}}

\newcommand{\Pichat}{\Pic^{{\smash{\hbox{\lower4pt\hbox{\hskip0.4pt$\widehat{\phantom{u}}$}}}}}}

\DeclareMathOperator{\Imm}{Im}
\DeclareMathOperator{\Ker}{Ker}
\DeclareMathOperator{\inv}{inv}
\DeclareMathOperator{\Cores}{Cores}
\newcommand{\Fv}[1]{\F_{\mkern-2mu#1}}
\newcommand{\cyc}{\mathrm{cyc}}

\newcommand{\wtl}{\widetilde}

\ifdef{\G}%
{\renewcommand{\G}{{\mathcal{G}}}}%
{\newcommand{\G}{{\mathcal{G}}}}

\newcommand{\Ext}{\mathrm{Ext}}

 \makeatletter
 \renewcommand{\tocsection}[3]{%
   \indentlabel{\@ifnotempty{#2}{\bfseries\ignorespaces#1 #2\quad}}\bfseries#3}

 \renewcommand{\tocsubsection}[3]{%
   \indentlabel{\@ifnotempty{#2}{\hspace{1.6em}\ignorespaces#1 #2\quad}}#3}
 \makeatother

\makeatletter\let\@wraptoccontribs\wraptoccontribs\makeatother

\date{September 7th, 2021; revised on June 1st, 2022}
\title[Rational points on fibrations with few non-split fibres]{Rational points on fibrations\\with few non-split fibres}

\author{Yonatan Harpaz}
\address{Institut Galil\'ee, Universit\'e Sorbonne Paris Nord, 99~avenue Jean-Baptiste Cl\'ement, 93430 Villetaneuse, France}
\email{harpaz@math.univ-paris13.fr}

\author{Dasheng Wei}
\address{Hua Loo-Keng Key Laboratory of Mathematics,
Academy of Mathematics and System Science, CAS, Beijing 100190,
P.\ R.\ China \& School of mathematical Sciences, University of  CAS, Beijing
100049, P.\ R.\ China}
\email{dshwei@amss.ac.cn}

\author{Olivier Wittenberg}
\address{Institut Galil\'ee, Universit\'e Sorbonne Paris Nord, 99~avenue Jean-Baptiste Cl\'ement, 93430 Villetaneuse, France}
\email{wittenberg@math.univ-paris13.fr}

\begin{document}

\begin{abstract}
We revisit the abstract framework underlying the fibration
method for producing rational points on the total space of fibrations over
the projective line.  By fine-tuning its dependence on external arithmetic
conjectures, we render the method unconditional when the degree of the
non-split locus is $\leq 2$, as well as in various instances where it
is~$3$. We are also able to obtain improved results in the regime that is
conditionally accessible under Schinzel's hypothesis, by incorporating into
it, for the first time, a technique due to Harari for controlling the
Brauer--Manin obstruction in families.
\end{abstract}

\maketitle

\section{Introduction}
\label{sec:intro}

In 1970, Manin~\cite{maninicm} showed that an obstruction based on
Brauer groups of schemes, now referred to as the Brauer--Manin
obstruction, can often explain failures of the Hasse principle and weak approximation
for the rational points of an algebraic variety~$X$ defined over a number field~$k$.
A conjecture of Colliot-Thélène predicts that the Brauer--Manin obstruction
explains all such failures when~$X$ is smooth, proper and rationally connected---by which we mean that for any
algebraically closed field extension~$K$ of~$k$, two general $K$\nobreakdash-points of~$X$
are joined by a rational curve defined over~$K$ (see~\cite[Chapter~IV]{kollarbook}).
In terms of the diagonal embedding of the set of rational points~$X(k)$ in the space
of adelic points~$X(\A_k)$, this conjecture is stated as follows:

\newcommand{\citeconjct}{\cite{ctbudapest}}
\begin{conj}[\citeconjct]
\label{conj:ct}
Let $X$ be a smooth, proper and rationally connected variety over a number field $k$.
The set $X(k)$ is dense in the Brauer--Manin set $X(\A_k)^{\Br(X)}$.
\end{conj}

Though wide open, Conjecture~\ref{conj:ct} has been established for many
special families of rationally connected varieties.
The reader will find  in~\cite[\textsection3]{wittenbergslc}
an almost up-to-date survey of known
methods and results.

A common structure that can often be fruitfully exploited to study rational points on~$X$ is that of a
fibration $f:X \to \P^1_k$.  (We use the term ``fibration'' in a loose sense,
to refer to a morphism whose generic fibre is geometrically irreducible.)
By the Graber--Harris--Starr theorem
\cite{ghs}, if the generic fibre of~$f$ is rationally connected, then $X$ is rationally connected as well. In the context of Conjecture~\ref{conj:ct}, this naturally leads one to ask:

\begin{question}\label{q:fibration-intro}
Let~$X$ be a smooth, proper, irreducible variety over a number field~$k$
and $f:X \to \P^1_k$ be a dominant morphism whose geometric generic fibre is rationally connected.
Assume that $X_c(k)$ is dense in $X_c(\A_k)^{\Br(X_c)}$
for all but finitely many $c \in \P^1(k)$,
where $X_c = f^{-1}(c)$.
Does it follow that~$X(k)$ is dense in $X(\A_k)^{\Br(X)}$?
\end{question}

Question~\ref{q:fibration-intro} has been extensively studied. One approach
consists in applying the theory of descent developed by
Colliot-Thélène and Sansuc~\cite{ctsandescent2} to reduce the problem to certain
torsors associated with the vertical Brauer group of $X$ relative to~$\P^1_k$.
This approach, systematically formalised by Skorobogatov \cite{skorodescent}
and by Colliot-Thélène and Skorobogatov \cite{ctskodescent},
has been applied successfully in many special
cases, including Châtelet surfaces~\cite{cssI}, other types of conic and
quadric bundles~\cite{bms}, and various normic (or more generally toric)
bundles \cite{heathbrownskorobogatov,cthasko,derenthalsmeetswei,skodescenttoric,browningmatthiesen}.
See also \cite{sd-conics-with-six,browningheathbrown} for instances of implicit uses of this approach.

Another approach is the fibration method, the first instance of which was Hasse's proof of
the local-global principle for the isotropy of quadratic forms over number fields.
Here, the argument is less concerned with
the particular geometry of $X$, but on the other hand
 it is highly sensitive to the
 \emph{rank} of~$f$,
defined as the degree of the finite locus in~$\P^1_k$ consisting of the closed points~$m$
such that the fibre~$X_m$ is not split over~$k(m)$,
and to the possible finite extensions $L/k(m)$ that split~$X_m$, i.e.\ that are
such that the $L$\nobreakdash-variety
$X_m \otimes_{k(m)} L$ is split.
(We recall that a variety is said to be \emph{split}
if it contains a geometrically integral open subset, see \cite[Definition~0.1]{skorodescent}.)

When the rank of~$f$ is~$\leq 1$,
a version of the fibration method yielding a positive answer to Question~\ref{q:fibration-intro} was established by Skorobogatov~\cite{skofibration} under the additional assumption that the smooth fibres of~$f$ satisfy the Hasse principle and weak approximation. Under the same assumption on the smooth fibres, this result was extended by Colliot-Thélène and Skorobogatov~\cite{skorodescent, ctskodescent} to the case where the rank of~$f$ is~$2$, and to the case where the rank is~$3$ and every non-split fibre $X_m$ is split by a quadratic extension of~$k(m)$.
When the rank of~$f$ is~$\leq 1$, the assumption that the fibres satisfy the Hasse principle and weak approximation was removed in Harari's thesis~\cite{harariduke}, using a delicate argument to control the Brauer--Manin
obstruction in the smooth fibres.

In situations more general than the above, one can still make the fibration method work conditionally if one assumes Schinzel's hypothesis. This was first observed by Colliot-Thélène and Sansuc~\cite{ctsansucschinzel} and
later extended in \cite{serrecollege,sdpencils,ctsd94,ctsksd98}, culminating 
in~\cite[Theorem 1.1(e)]{ctsksd98}
in a positive answer to Question~\ref{q:fibration-intro} under the following assumptions:
\begin{enumerate}
\item
Schinzel's hypothesis holds.
\item\label{item:abelian}
Each non-split fibre $X_m$ is split by an abelian extension of $k(m)$.
\item\label{item:smooth}
The smooth fibres of~$f$ satisfy the Hasse principle and weak approximation.
\end{enumerate}
We recall that Schinzel's hypothesis is a vast generalization of the twin prime conjecture, which says that
a finite collection of irreducible polynomials in $\Z[t]$ infinitely often  take prime values simultaneously,
unless there is a local obstruction to this being so.
Schinzel's hypothesis is required
in \emph{loc.\ cit.},
when $k=\Q$,
for the polynomials that encode
the non-split fibres of~$f$ over the points of $\A^1_k \subset \P^1_k$.
(When $k\neq \Q$, see \cite[Proposition~4.1]{ctsd94}.)
The only known case of Schinzel's hypothesis is the case of a single polynomial of degree~$1$ (i.e.\ Dirichlet's theorem on primes in arithmetic progression), which makes the last theorem unconditional when the locus of non-split
fibres consists of either one or two rational points of~$\P^1_k$.
However, as we described above, these small rank cases are also known under less restrictive conditions.

In a more recent development, the first- and third-named authors, in
collaboration with Skorobogatov~\cite{hsw}, showed that the fibration
method can also be set up using a variant of Schinzel's hypothesis applied
to homogeneous polynomials in two variables (still keeping conditions~\ref{item:abelian}
and~\ref{item:smooth} above).  This homogeneous variant, referred to as
Schinzel's hypothesis $(\mathrm{HH}_1)$, has the advantage of being known
in a slightly wider set of circumstances. Most notably,
by the seminal work of
Green--Tao--Ziegler~\cite{gt1,gt2,gtz3}, it is known, over~$\Q$, when all
the homogeneous polynomials are linear
(see \cite[Proposition~1.2]{hsw}).
Another known case, due to Heath-Brown and Moroz~\cite{hbm},
is that
of a single homogeneous polynomial of degree~$3$, over~$\Q$.

The use of the homogeneous variant of Schinzel's hypothesis
has led to more unconditional answers to Question~\ref{q:fibration-intro} (see \cite[Theorem~3.3]{hsw});
however, these were still subject to
conditions~\ref{item:abelian} and especially~\ref{item:smooth} above.
Later on,
in \cite[\textsection9]{hwfibration}, the first- and third-named authors suggested a way of
bypassing condition~\ref{item:abelian} by replacing Schinzel's hypothesis
$(\mathrm{HH}_1)$ with another hypothesis, namely~\cite[Conjecture~9.1]{hwfibration}, and showed that
the latter implies a positive
answer to Question~\ref{q:fibration-intro} in complete generality, even in the absence
of condition~\ref{item:smooth}. This
yields an unconditional positive answer to Question~\ref{q:fibration-intro} when $k=\Q$ and all the non-split
fibres of~$f$ lie above rational points of~$\P^1_\Q$, as the corresponding special
case of~\cite[Conjecture 9.1]{hwfibration} was established by
Matthiesen~\cite{matthiesen}, building on the work of Browning
and Matthiesen~\cite{browningmatthiesen}.

Conjecture 9.1 of~\cite{hwfibration} mentioned above depends, among others, on a collection of
closed points $m_1,\dots,m_n$ of $\P^1_k$ and
a collection of field extensions $L_1/k(m_1), \dots, L_n/k(m_n)$.
Unfortunately, given a fibration $f:X \to \P^1_k$ with
rationally connected geometric generic fibre,
the precise choice of $m_1,\dots,m_n$ and $L_1,\dots,L_n$
needed to obtain a positive answer to
Question~\ref{q:fibration-intro} using the results of~\cite{hwfibration}
is not completely
straightforward: while one has to  include in this list at least the closed
points over which the fibre of~$f$ is non-split, additional points might
be required in order to make it possible to describe the Brauer groups of the smooth fibres
uniformly in terms of Brauer classes defined on the complement, in~$X$, of
the union of the fibres $X_{m_1},\dots,X_{m_n}$.
Similarly, each $L_i$
needs to be not only large enough to split the fibre $X_{m_i}$, it
must also split the ``constant'' part of the residues of these Brauer
classes.  All in all, this state of affairs is inconvenient since for a given fibration $f:X\to \P^1_k$,
the precise choice of $m_1,\dots,m_n$ and $L_1,\dots,L_n$ to which
Conjecture 9.1 needs to be applied is implicit and hard to determine
in practice. A more serious consequence is that the
applicability of~\cite[\textsection9]{hwfibration} finds itself hindered in the regime where unconditional cases
of Question~\ref{q:fibration-intro} are potentially within reach. For
example:
\begin{enumerate}[label={\upshape(\roman*)}]
\setlength\itemsep{.7ex}
\item Conjecture 9.1 of \cite{hwfibration} is known in various
instances where $r = \sum_{i=1}^n \deg(m_i)$ is small, e.g.\ when $r \leq 2$,
or when $r \leq 3$ and each $L_i/k(m_i)$ is quadratic
(see~\cite[Theorem 9.11]{hwfibration} and Remark~\ref{rem:puncturedaffinecone}
below). However, due to the possible need for additional closed points, one
cannot deduce from this an unconditional answer to
Question~\ref{q:fibration-intro} for fibrations of rank~$2$, or for
fibrations of rank~$3$ whose non-split fibres are split by quadratic
extensions. One can still deduce a positive answer under the additional
assumption that the smooth fibres of~$f$ satisfy the Hasse principle and
weak approximation, thus (only) recovering \cite[Theorem~A and Theorem~B]{ctskodescent}.
\item
When all the field extensions $L_i/k(m_i)$ are abelian (or, more generally,
almost abelian in the sense of~\cite[Definition 9.4]{hwfibration}),
Conjecture 9.1 of \cite{hwfibration} is implied by Schinzel's hypothesis
$(\mathrm{HH}_1)$ for the defining homogeneous polynomials of the points
$m_1,\dots,m_n$. However, 
due to the possible need to increase the fields~$L_i$,
one cannot deduce that Schinzel's
hypothesis $(\mathrm{HH}_1)$ implies a positive answer to
Question~\ref{q:fibration-intro} when the fibres of~$f$ are split by, say, abelian extensions.
One can still deduce a positive answer under the additional
assumption that the smooth fibres of~$f$ satisfy the Hasse principle and weak
approximation, a result already implicit in~\cite{hsw}.
By Heath-Brown and
Moroz~\cite{hbm}, this yields
an unconditional positive answer to
Question~\ref{q:fibration-intro},
when $k=\Q$,
in the case of a single non-split fibre over
a point of degree~$3$, but only when conditions~\ref{item:abelian}
and~\ref{item:smooth} above both hold.
\end{enumerate}

Our goal in this article is to address these issues by putting forth an improved form
of~\cite[Conjecture 9.1]{hwfibration}. To this end, we rely on
a construction of auxiliary varieties that was introduced and exploited in~\cite{hwfibration} in
connection with Conjecture~9.1 of \emph{op.\ cit.}
This construction,
 antecedents of which
had previously come up
in the context of descent theory
(see \cite[\textsection3.3]{skorodescent}, \cite[p.~391]{ctskodescent},
\cite[\textsection4.4]{skobook}),
involves associating with a choice of
points $m_1,\dots,m_n$ and of finite extensions $L_1/k(m_1),\dots, L_n/k(m_n)$ a certain family of
non-proper varieties, typically denoted $W$, equipped with a
morphism $W \to \P^1_k$ whose singular fibres lie over the closed points
$m_1,\dots,m_n$ and are respectively split by the field extensions $L_1/k(m_1),\dots, L_n/k(m_n)$. It
is shown in~\cite[Proposition 9.9]{hwfibration} that Conjecture 9.1 of
\emph{op.\ cit.}\ holds for the parameters $m_1,\dots,m_n$ and $L_1/k(m_1),\dots,L_n/k(m_n)$
if and only if 
for every variety $W$ in the family associated with these parameters,
the subset $\bigcup_{c \in \P^1(k)} W_c(\A_k)$ of
$W(\A_k)$ is a dense subset.
In the present article, we study a strengthened version of this last property, which, on the one hand, becomes equivalent to the original version when stated for all possible choices of $m_1,\dots,m_n$ and $L_1/k(m_1),\dots, L_n/k(m_n)$, and, on the other hand, is sufficiently strong to allow the incorporation of non-vertical Brauer classes into the fibration method.
This leads to an equivalent, alternative approach to the framework of~\cite[\S 9]{hwfibration}
that has the practical advantage of unlocking many of the remaining unconditional cases which
escaped \emph{op.\ cit.} In particular, we obtain the following new cases as applications:

\begin{thm}[see Theorem~\ref{th:applicationrank2} and Theorem~\ref{thm:applicationrank3}]
\label{thm:applications-intro}
Question~\ref{q:fibration-intro} has a positive answer in each of the following cases:
\begin{enumerate}[label={\upshape(\roman*)}]
\item the rank of $f$ is at most $2$;
\item the rank of $f$ is $3$ and every fibre~$X_m$ is split by a quadratic extension of~$k(m)$;
\item the rank of $f$ is $3$, one fibre~$X_m$ lies above a rational point of~$\P^1_k$
and every remaining fibre~$X_m$  is split by a quadratic extension of~$k(m)$.
\end{enumerate}
\end{thm}

In addition, using this improved framework, we are able to combine for the first time
the arguments of Harari~\cite{hararifleches} for dealing with Brauer--Manin obstructions in the fibres
with the use of Schinzel's hypothesis 
in the fibration
method---a problem that had been open since the 1990's. This leads to a version of
the theorem of Colliot-Thélène, Skorobogatov and Swinnerton-Dyer~\cite[Theorem~1.1(e)]{ctsksd98}
in which the smooth fibres are not assumed any more to satisfy the Hasse principle or weak approximation
(though at the expense of assuming that the splitting fields are cyclic, or, alternatively, almost abelian but not abelian):

\begin{thm}[see Theorem~\ref{thm:schinzelcyclic}]\label{thm:schinzel-intro}
Question~\ref{q:fibration-intro} has a positive answer when the following two conditions are both satisfied:
\begin{enumerate}
\item Schinzel's hypothesis~$(\mathrm{HH}_1)$ holds for the set of irreducible homogeneous two-variable polynomials
vanishing on the closed points $m \in \P^1_k$ such that the fibre~$X_m$ is non-split;
\item\label{item:schinzel-intro-almostabelian} every fibre $X_m$
is split by an extension of $k(m)$ that is either cyclic or almost abelian but non-abelian
(e.g.\ a cubic extension).
\end{enumerate}
\end{thm}

When combined with the work of Heath-Brown and
Moroz~\cite{hbm},
 Theorem~\ref{thm:schinzel-intro}  yields
an unconditional positive answer to
Question~\ref{q:fibration-intro},
when $k=\Q$,
in the case of a single non-split fibre over
a point of degree~$3$, subject to condition~\ref{item:schinzel-intro-almostabelian}
of Theorem~\ref{thm:schinzel-intro}
but without assuming anything on the smooth fibres of~$f$ beyond their rational connectedness
(see Theorem~\ref{thm:applicationrank3}).

The article is organised as follows. We begin in \S\ref{sec:unramified brauer classes}
with a discussion of the unramified Brauer groups of norm~$1$ tori and their
torsors, over arbitrary fields of characteristic~$0$.
In~\S\ref{sec:W},
we discuss the variety~$W$ which plays a key r\^ole in this
work. After verifying that the Brauer group of $W$ is reduced to constant classes
(Proposition~\ref{prop:brauerw}),
we discuss
 some canonical ramified Brauer classes on~$W$ 
(see~\textsection\ref{subsec:canonical ramified})
and use them to formulate a conjecture on the arithmetic of~$W$, which we call
Conjecture~$\cFplus$
(see~\textsection\ref{sec:statementFplus}). It is a strengthened version of the property that
the subset
$\bigcup_{c \in \P^1(k)} W_c(\A_k)$ is dense in~$W(\A_k)$, which we dub Conjecture~$\cF$.
In~\S\ref{sec:fibration}, we establish the main technical result
of the article (Theorem~\ref{th:fibration}),
according to which Conjecture $\cFplus$ implies a positive answer to
Question~\ref{q:fibration-intro}.
The argument is  similar to the
proof of~\cite[Theorem~9.17]{hwfibration}, with the difference that the
variety $W$ for which one needs to assume Conjecture $\cFplus$ is under
better control. In \S\ref{sec:comparing}, we analyse the relationship
between Conjectures $\cF$ and $\cFplus$ in greater detail. The comparison goes through an
auxiliary intermediate statement, which we call Conjecture~$\cFconst$. The
results of this section, and in particular those of
\S\ref{sec:from-const-to-plus}, are the main input needed for
Theorem~\ref{thm:schinzel-intro}. In \S\ref{sec:knowncases}, we collect
all of the cases of Conjecture~$\cFplus$ that we are able to prove, either by strong
approximation arguments, from Schinzel's hypothesis or from additive combinatorics.
Finally, in~\S\ref{sec:newapplications}, we deduce
Theorem~\ref{thm:applications-intro} and Theorem~\ref{thm:schinzel-intro} and
discuss some examples.

\bigskip
\emph{Acknowledgements.}

We thank the anonymous referee for their helpful comments.
The second-named author is supported by National Key R\&D Program of China and National Natural Science
Foundation of China (Grant Nos.\ 11622111 and 11631009).

\bigskip
\emph{Notation and conventions.}
In this article, an \emph{extension} of a field will always be meant to be a field extension.
Let~$X$ be a variety, that is, a scheme of finite type, over a field~$k$.
We say that~$X$ is
 \emph{split} if it contains
a geometrically integral open subset.
We say that~$X$ is
\emph{rationally connected}
if for any  algebraically closed
field extension~$K$ of~$k$,
the variety $X \otimes_k K$ over~$K$
is rationally connected in the sense of Campana,
Koll\'ar, Miyaoka and Mori.
The Brauer group $\Br(X)=H^2_{\et}(X,\Gm)$
contains the \emph{algebraic Brauer group} $\Br_1(X)=\Ker\big(\Br(X) \to \Br(X \otimes_k K)\big)$,
where~$K$ denotes any algebraically closed field extension of~$k$
(the kernel does not depend on the choice of~$K$),
and the subgroup $\Br_0(X) = \Imm\big(\Br(k)\to\Br(X)\big)$ of constant classes.
When~$X$ is smooth and connected and~$k$ has characteristic~$0$, we denote by $\Br_{\nr}(X) \subseteq \Br(X)$
the \emph{unramified Brauer group}, which coincides with the Brauer group of any smooth and proper
variety birationally equivalent to~$X$.
\emph{Hilbert subsets} of an irreducible variety are meant in the sense of
Lang \cite[Chapter~9, \textsection5]{langfunddioph}, and thus do not only consist
of rational points; for example
the generic point belongs to every Hilbert subset
 (see also
\cite[Notation and conventions]{hwfibration}).
If~$k$ is a number field and~$v$ is a place of~$k$, we denote by~$k_v$ the completion of~$k$
at~$v$ and, if~$v$ is finite, by $\Fv{v}$ the residue field of~$v$.
Given a nonzero $k$\nobreakdash-algebra~$L$ and a variety~$Y$ over~$L$,
we denote by $R_{L/k}Y$ the Weil restriction of~$Y$ from~$L$ down to~$k$.
If~$M$ is an abelian group and~$L$ is an \'etale $k$\nobreakdash-algebra,
say $L=\prod L_i$ where the~$L_i$ are fields, we write $H^n(L,M)$ for
the group $\prod H^n(L_i,M)$, where $H^n(L_i,M)$ denotes the $n$\nobreakdash-th Galois cohomology
group of~$L_i$ with values in~$M$.
If~$G$ is a group and~$M$ is a $G$\nobreakdash-module,
we denote by $C^n(G,M)$ the group of $n$\nobreakdash-cochains of~$G$ with values in~$M$
and by $Z^n(G,M)$ the subgroup of $n$\nobreakdash-cocycles.
Finally, if~$k$ is a number field with algebraic closure~$\bar k$,
we say that a finite field extension $L/k$ is \emph{almost abelian} 
(in the sense of
 \cite[Definition~9.4]{hwfibration})
if it is abelian or if
there exist a prime number~$p$ and a bijection
$\Spec(L\otimes_k \bar k) \simeq \F_p$
that makes $\Gal(\bar k/k)$ act on~$\F_p$ by affine transformations.

\section{Unramified Brauer groups of torsors under norm tori}
\label{sec:unramified brauer classes}

We discuss three complementary
aspects of unramified Brauer groups of norm tori and of their torsors
over arbitrary fields of characteristic~$0$,
in three
 logically independent sections
\textsection\ref{subsec:criterion unramifiedness},
\textsection\ref{subsec:some norm tori},
\textsection\ref{subsec:central extensions}.
The first one provides a criterion
for classes of the shape $\Cores_{L/k}(z,\chi)$ to be unramified
(Proposition~\ref{prop:chinr}).
This will play a key  r\^ole
in the proof of Theorem~\ref{thm:schinzel-intro}.
The second one is devoted to norm tori~$T$ such that $\Br_{\nr}(T)=\Br_0(T)$,
a condition which also turns up in the proof
of Theorem~\ref{thm:schinzel-intro}
(see Corollary~\ref{cor:comparisonFconstFplus}).
The third one reinterprets the computation of the quotient group $\Br_{\nr}(T)/\Br_0(T)$
for norm tori~$T$
in terms of central extensions of a finite group~$G$ by~$\Q/\Z$
(Proposition~\ref{prop:bijection h2extgh}).
This last result
is not used
 in the rest of the article
but it may be of interest in a broader perspective.
We illustrate it
by analysing an example of Kunyavski\u{\i} of a class of three-dimensional tori
satisfying
 $\Br_{\nr}(T)\neq\Br_0(T)$
(Example~\ref{ex:abelian-alternating})
and by giving an independent proof of Bartels' theorem according to which $\Br_{\nr}(T)=\Br_0(T)$
if~$T$ is the norm torus associated with a degree~$n$ field extension whose Galois closure has Galois group~$D_n$
(Example~\ref{ex:bartels}).

\subsection{Unramifiedness of corestrictions}
\label{subsec:criterion unramifiedness}

Let $k$ be a field of characteristic $0$ with algebraic closure $\bar k$.
Let~$T$ be an algebraic torus over~$k$.
Let~$Z$ be a torsor, over~$k$, under~$T$.

The problem of calculating the unramified Brauer group of $Z$ is
classically addressed by noting, on the one hand,
that the group $\Br_\nr(Z \otimes_k \bar k)$ vanishes,
so that $\Br_\nr(Z) \subseteq \Br_1(Z)$, and on the other hand,
that the group $\Pic(Z \otimes_k \bar k)$ vanishes,
so that the Hochschild--Serre
spectral sequence determines an isomorphism $\Br_1(Z) = H^2(k,\bar k[Z]^*)$ and hence
an injection
\begin{align}
\label{eq:injectionbr1br0z}
\Br_1(Z)/\Br_0(Z) \hookrightarrow H^2(k,\bar k[Z]^*/\bar k^*) = H^2(k,\widehat{T})\rlap,
\end{align}
which is an isomorphism when
either $Z(k)\neq\emptyset$
or $H^3(k,\bar k^*) = 0$, e.g.\ when $k$ is a
number field.
Let us denote by $\Sha^2_\omega(k,\widehat T)$
the subgroup of $H^2(k,\widehat T)$
consisting of those classes which vanish in $H^2(k',\widehat T)$
for any field extension $k'/k$ such that the absolute Galois group
of~$k'$ is procyclic.
It is known that an element of $\Br_1(Z)$ belongs to
$\Br_{\nr}(Z)$ if and only if its image in $H^2(k,\widehat{T})$ belongs to
$\Sha^2_{\omega}(k,\widehat{T})$ (see \cite[Théorème~8.1]{bodeha}).
This results in an injection
\begin{align}
\label{eq:classical-formula-part1}
\Br_{\nr}(Z)/\Br_0(Z) \hookrightarrow \Sha^2_{\omega}(k,\widehat{T})\rlap{,}
\end{align}
which is an isomorphism when $Z(k)\neq\emptyset$ or $H^3(k,\bar k^*) = 0$,
thus
turning, in these cases, the determination of $\Br_{\nr}(Z)/\Br_0(Z)$ into a finite
and
computable endeavour: indeed,
the action of $\Gal(\bar k/k)$ on $\widehat{T}$
factors through a finite quotient, say~$G$,
and we can identify
\begin{align}
\label{eq:classical-formula-part2}
\Sha^2_{\omega}(k,\widehat{T}) \isotoleft \Sha^2_{\cyc}(G,\widehat{T})\rlap,
\end{align}
where $\Sha^2_{\cyc}(G,\widehat{T})$
 denotes the subgroup of $H^2(G,\widehat{T})$
consisting of those elements whose image in $H^2(H,\widehat{T})$ vanishes for
every cyclic subgroup~$H$ of~$G$.

Our goal in this section is to extract from the above an explicit
description of certain classes in $\Br_{\nr}(Z)$ in the case where~$T$
is the norm 1 torus associated with a given nonzero étale algebra $L$
over~$k$ and~$Z$ is the corresponding norm $c$ torsor for
some~$c \in k^*$.

We fix~$L$ and~$c$ until the end of~\textsection\ref{sec:unramified brauer classes}.
The torsor~$Z$ (resp.\ the torus~$T$) can be explicitly described
as the closed
subvariety
of $R_{L/k}(\A^1_L)$ defined by the equation $N_{L/k}(z)=c$
(resp.\ by the equation $N_{L/k}(z)=1$), where~$z$ denotes a point
of $R_{L/k}(\A^1_L)$.
We shall again denote by~$z$
the invertible function on $Z \otimes_k L$ obtained by restricting
the tautological regular function on $R_{L/k}(\A^1_L) \otimes_k L$.
For $\chi \in H^1(L,\Q/\Z)$, let us write $(z,\chi) \in \Br(Z\otimes_k L)$ for
the class obtained by identifying $H^1(L,\Q/\Z)$ with $H^2(L,\Z)$ and
considering the image of the pair $(z,\chi)$ by the cup product map
\begin{align*}
H^0(Z \otimes_k L,\Gm) \times H^2(L,\Z) \to \Br(Z \otimes_k L)\rlap.
\end{align*}
We note that $(z,\chi) \in \Ker\big(\Br(Z\otimes_k L) \to \Br(Z \otimes_k L \otimes_k \bar k)\big)$
since~$\chi$ vanishes when pulled back to $L\otimes_k \bar k$.
Hence its push-forward
$\Cores_{L/k}(z,\chi) \in \Br(Z)$  along the
projection $Z \otimes_k L \to Z$
belongs to the subgroup $\Br_1(Z) \subseteq \Br(Z)$.

The following proposition gives criteria for this class
$\Cores_{L/k}(z,\chi) \in \Br_1(Z)$ to be constant or 
unramified.
It builds on the ideas contained in \cite[Theorem~5]{weiunramified}, in which~$L$ is assumed
to be a Galois extension of~$k$.

\begin{prop}
\label{prop:chinr}
Let $\chi \in  H^1(L,\Q/\Z)$.
\begin{enumerate}
\item
$\Cores_{L/k}(z,\chi) \in \Br_0(Z)$ if and only if $\chi \in \Imm\mkern-2mu\big(H^1(k,\Q/\Z)\to H^1(L,\Q/\Z)\big)$.
\item
$\Cores_{L/k}(z,\chi) \in \Br_{\nr}(Z)$ if and only if for any field extension~$k'/k$ such
that the absolute Galois group of~$k'$ is procyclic and for any two
$k$\nobreakdash-algebra homomorphisms $\iota_1,\iota_2:L\to k'$,
the two images of~$\chi$ in $H^1(k',\Q/\Z)$ via~$\iota_1$ and via~$\iota_2$
are equal.
\end{enumerate}
\end{prop}

\begin{proof}
Letting~$\bar k$ denote a fixed algebraic closure of $k$
and $\Sigma$ the set of $k$\nobreakdash-algebra homomorphisms $L \to \bar k$,
the character group $\widehat T$ of the torus~$T$
sits in the commutative diagram of Galois modules
\begin{align}
\begin{aligned}
\label{diag:contains alpha}
\xymatrix@R=3ex{
0 \ar[r] & \Z \ar[r]\ar[d] & \Z[\Sigma] \ar[r]\ar[d]^(.45){\alpha} &  \widehat T \ar[r]\ar[d]^(.4)\wr & 0\\
0 \ar[r] & \bar k^* \ar[r] & \bar k[Z]^* \ar[r] & \bar k[Z]^*/\bar k^* \ar[r] & 0
}
\end{aligned}
\end{align}
whose rows are exact, whose leftmost vertical arrow sends~$1$ to~$c$,
and
whose middle vertical arrow sends $\sigma \in \Sigma$ to the image
of $z$ by the morphism $L[Z]^* \to \bar k[Z]^*$ induced by~$\sigma$.
The rightmost vertical arrow is an isomorphism.
Passing to cohomology and taking into account the canonical isomorphism
$\Br_1(Z)=H^2(k,\bar k[Z]^*)$
given by the Hochschild--Serre spectral sequence,
we obtain the commutative diagram with exact rows
\begin{align}
\label{diag:normtorus}
\begin{aligned}
\xymatrix@R=3ex{
H^1(k,\Q/\Z) \ar[r] \ar[d] & H^1(L,\Q/\Z) \ar[r]^\beta \ar[d]^(.45){\alpha'} & H^2(k,\widehat T)
\ar[d]^(.4)\wr \\
\Br(k) \ar[r] & \Br_1(Z) \ar[r] & H^2(k,\bar k[Z]^*/\bar k^*)\rlap{.}
}
\end{aligned}
\end{align}

\begin{lem}
\label{lem:whoisbeta}
The map $\alpha'$ sends $\chi \in H^1(L,\Q/\Z)$
to $\Cores_{L/k}(z,\chi) \in \Br_1(Z)$.
\end{lem}

\begin{proof}
We denote by $f:Z\to \Spec(k)$ and $f':Z \otimes_kL \to \Spec(L)$ the structure morphisms
and by $\nu:Z \otimes_kL \to Z$ the natural projection.
Let $\gamma:\Z \to \Gm$ be the morphism of \'etale sheaves on $Z \otimes_kL$
defined by $\gamma(1)=z \in H^0(Z \otimes_k L,\Gm)$.
Applying the functors $H^2(Z,-)$ and $H^2(k,f_*-)$, related by a natural transformation
$H^2(k,f_*-)\to H^2(Z,-)$,
to the morphisms of étale sheaves on~$Z$
\begin{align}
\label{eq:mapofsheaves}
\xymatrix{
\nu_*\Z \ar[r]^{\nu_*\gamma} & \nu_*\Gm \ar[r]^{N_{L/k}} & \Gm\rlap{,}
}
\end{align}
we obtain a commutative diagram
\begin{align}
\label{diag:whoisbeta}
\begin{aligned}
\xymatrix@C=2.4em@R=3ex{
\ar@{=}[d] H^2(k,f_*\nu_*\Z) \ar[r] & H^2(k,f_*\nu_*\Gm) \ar[r] & H^2(k,f_*\Gm) \ar@{=}[d] \\
*!<3.35em,0ex>\entrybox{H^1(L,\Q/\Z)=H^2(L,\Z)}
 \ar[d]_{f'^*} \ar[rr] &&
H^2(k,\bar k[Z]^*)
\ar[d] \\
 H^2(Z \otimes_k L,\Z) \ar[r]^(.52){\xi \mkern3mu\mapsto\mkern3mu z \smile \xi} &
\Br(Z \otimes_k L) \ar[r]^(.56){\Cores_{L/k}} & \Br(Z)\rlap,
}
\end{aligned}
\end{align}
The map $(f_*\nu_*\Z)(\bar k) \to (f_*\Gm)(\bar k)$
obtained from the composition of the morphisms in~\eqref{eq:mapofsheaves}
coincides,
via the canonical isomorphism $(f_*\nu_*\Z)(\bar k)=\Z[\Sigma]$,
with the map~$\alpha$ appearing in~\eqref{diag:contains alpha}.
Thus,
the composition of the middle horizontal maps
of~\eqref{diag:whoisbeta}
with the lower right-hand side vertical map
of~\eqref{diag:whoisbeta} coincides with~$\alpha'$.
The lemma then follows from the commutativity of this diagram.
\end{proof}

Lemma~\ref{lem:whoisbeta}
and a diagram chase in~\eqref{diag:normtorus} together
imply Proposition~\ref{prop:chinr}~(1).

Let us now prove Proposition~\ref{prop:chinr}~(2).
As discussed above,
an element of $\Br_1(Z)$ belongs to the subgroup $\Br_{\nr}(Z)$
if and only if its image in $H^2(k,\widehat T)$
by the map $\Br_1(Z) \to H^2(k,\widehat T)$
extracted from the diagram~\eqref{diag:normtorus}
belongs to $\Sha^2_\omega(k,\widehat T)$, see \cite[Théorème~8.1]{bodeha}.
In view of~\eqref{diag:normtorus} and of Lemma~\ref{lem:whoisbeta},
the proof will therefore be complete once we
check that
the following two conditions are equivalent:
\begin{enumerate}
\setlength\itemsep{.7ex}
\item[(a)]
for any
field extension $k'/k$ such that the absolute Galois group
of~$k'$ is procyclic,
the images of~$\chi$ in $H^1(k',\Q/\Z)$ via any
two $k$\nobreakdash-algebra morphisms $L\to k'$
are equal;
\item[(b)]
for any
field extension $k'/k$ such that the absolute Galois group
of~$k'$ is procyclic,
 the image of $\chi$ in $H^1(L \otimes_k k',\Q/\Z)$
comes from $H^1(k',\Q/\Z)$.
\end{enumerate}
Any $k$\nobreakdash-algebra morphism $L\to k'$ induces
a $k'$\nobreakdash-algebra morphism $L \otimes_k k' \to k'$ and hence a retraction of the natural
map $H^1(k',\Q/\Z)\to H^1(L\otimes_k k',\Q/\Z)$.
Therefore (b)$\Rightarrow$(a).
  For the converse,
we need the following lemma.

\begin{lem}
\label{lem:elementary}
Let $N \geq 1$.
Let $n_1,\dots,n_N$ be nonzero integers and let
$(\lambda_i)_{1\leq i \leq N} \in (\Q/\Z)^N$
satisfy the equality
\begin{align}
\frac{n}{n_i}\lambda_i=\frac{n}{n_j}\lambda_j
\end{align}
for all $i,j \in \{1,\dots,N\}$
and for all integers~$n$ divisible by both~$n_i$ and~$n_j$.
Then there exists $\lambda \in \Q/\Z$ such that $\lambda_i=n_i\lambda$
for all $i \in \{1,\dots,N\}$.
\end{lem}

\begin{proof}
By induction on~$N$, we may assume that $N=2$.
Then, as $\Q/\Z$ is divisible, we may assume that~$n_1$ and~$n_2$ are coprime.
Write $1=a_1 n_1 + a_2 n_2$ for $a_1,a_2\in \Z$.
Applying the hypothesis to $n=n_1n_2$ shows that
 $\lambda=a_1 \lambda_1 + a_2\lambda_2$
satisfies the desired property.
\end{proof}

Going back to the implication (a)$\Rightarrow$(b),
let us write $L\otimes_k k' = E_1\times\dots \times E_N$
where $E_i$ is a field extension of~$k'$ of degree~$n_i$
and let us choose an algebraic closure $\bar k'$ of~$k'$, a topological generator
$\tau \in \Gal(\bar k'/k')$ and, for each $i \in \{1,\dots,N\}$, a
$k'$\nobreakdash-linear
isomorphism between~$E_i$ and the subfield of~$\bar k'$ fixed by $\tau^{n_i}$.
Using these choices to identify
$H^1(L\otimes_k k',\Q/\Z)$ with~$(\Q/\Z)^N$, the implication
(a)$\Rightarrow$(b) now results from Lemma~\ref{lem:elementary}.
\end{proof}

When~$L$ is a field,
Proposition~\ref{prop:chinr}~(2) can be reformulated in Galois theoretical terms,
viewing~$\chi$ as a group homomorphism $\Gal(\bar k/L)\to \Q/\Z$.

\begin{cor}
\label{cor:chinr}
Suppose that $L$ is a field.
Fix an algebraic closure~$\bar k$ of~$k$ and a $k$\nobreakdash-linear embedding $L \hookrightarrow \bar k$.
Let $\chi \in H^1(L,\Q/\Z)$.
The following conditions are equivalent:
\begin{enumerate}
\item $\Cores_{L/k}(z,\chi) \in \Br_{\nr}(Z)$;
\item 
for any $\tau_1,\tau_2 \in \Gal(\bar k/L)$ that are conjugate in $\Gal(\bar k/k)$,
the equality
 $\chi(\tau_1)=\chi(\tau_2)$ holds,
where we regard~$\chi$ as a group homomorphism $\Gal(\bar k/L)\to \Q/\Z$.
\end{enumerate}
\end{cor}

\begin{proof}
Assume that $\Cores_{L/k}(z,\chi) \in \Br_{\nr}(Z)$
and let
 $\tau_1,\tau_2 \in \Gal(\bar k/L)$
and
$\sigma \in \Gal(\bar k/k)$ satisfy $\sigma^{-1}\tau_1\sigma = \tau_2$.
Let $\iota_1:L\to k'$ be the inclusion
 of~$L$
in the subfield~$k'$ of~$\bar k$ fixed by~$\tau_1$,
and $\iota_2:L \to k'$ the map  induced by~$\sigma$.
Proposition~\ref{prop:chinr}~(2) now ensures
that
$\chi(\tau_1)=\chi(\tau_2)$.

To prove the converse implication, it suffices to note that
if two embeddings~$\iota_1,\iota_2$ of the extension~$L/k$ into
a field extension~$k'/k$ with procyclic absolute Galois group
are given, then, after choosing an algebraic closure~$\bar k'$ of~$k'$,
a topological generator $\tau \in \Gal(\bar k'/k')$ and, for
each $i \in \{1,2\}$, an embedding
$\bar\iota_i:\bar k\to \bar k'$
that extends~$\iota_i$, the images~$\tau_1$ and~$\tau_2$ of~$\tau$
by the restriction maps $\bar\iota_i^*:\Gal(\bar k'/k')\to \Gal(\bar k/L)$
are conjugate in $\Gal(\bar k/k)$.
\end{proof}

\begin{rmk}
If $k$ is a number field, then $H^2(k,\Q/\Z)=0$
(see \cite[Corollary~18.17]{harbook}), so that the map~$\beta$ appearing in~\eqref{diag:normtorus} is onto.
As a consequence, all elements of $\Br_1(Z)$ can be written
as $\delta + \Cores_{L/k}(z,\chi)$ for some $\chi \in H^1(L,\Q/\Z)$ and $\delta \in \Br_0(Z)$
(see~\eqref{diag:normtorus}).
Thus, in this case, Proposition~\ref{prop:chinr} fully describes
the unramified Brauer group of~$Z$.
\end{rmk}

\subsection{Some norm tori with \texorpdfstring{$\Br_{\nr}(T)=\Br_0(T)$}{Brₙᵣ(T)=Br₀(T)}}
\label{subsec:some norm tori}

For later use in~\textsection\ref{subsec:from schinzel},
the next proposition collects a few cases in which norm tori satisfy
$\Br_{\nr}(T)=\Br_0(T)$.

\begin{prop}
\label{prop:criterion-for-brauer}
Let~$L/k$ be a finite extension of fields of characteristic~$0$.
After choosing a bijection $\Spec(L\otimes_k \bar k)=\{1,\dots,n\}$,
we view
 the Galois group~$G$ of a Galois closure of~$L/k$
as a
transitive subgroup of the symmetric group~$S_n$.
The torus $T=R^1_{L/k}\Gm$ over~$k$
satisfies $\Br_{\nr}(T)=\Br_0(T)$ under any of the following assumptions:
\begin{enumerate}[label={\upshape(\roman*)}]
\item
for every prime divisor $p$ of $n$, the $p$-Sylow subgroups of~$G$
are cyclic;
\item $G$ is cyclic;
\item $n$ is prime;
\item $L/k$ is Galois and $H^3(G,\Z)=0$;
\item
 $\Ker\big(H^{\ab} \to G^{\ab}\big)$ has order prime to~$n$ (for example $H^{\ab}=0$)
and $n$ is squarefree;
\item $G=S_n$ is the symmetric group;
\item $G=D_n$ is the dihedral group of order~$2n$;
\item $G=A_n$ is the alternating group and $n\geq 5$.
\end{enumerate}
\end{prop}

Cases~(vi), (vii), (viii)
of
Proposition~\ref{prop:criterion-for-brauer}
are
due to Kunyavski\u{\i}--Voskresenski{\u\i}~\cite{voskresenskii-kunyavskii},
Bartels~\cite{bartels} and Macedo~\cite{macedo}, respectively.  For the
convenience of the reader, we reproduce below the main points of the
arguments, and provide a simple alternative proof of Macedo's
theorem in the exceptional case $G=A_6$, which in \emph{op.\ cit.}\ relied on
the use of a computer.  The proof we give here was inspired by an argument
of Drakokhrust and Platonov~\cite{drakplatonov}.

\begin{proof}[Proof of Proposition~\ref{prop:criterion-for-brauer}]
Let $H \subseteq G$ denote the subgroup whose fixed field is~$L$,
so that
$\widehat T = \Z^{G/H}\mkern-2mu/\mkern2mu\Z$
and $H^2(G,\widehat T)$ fits into an exact sequence
\begin{align}
\label{se:h2ghatt}
\xymatrix{
H^2(G,\Z) \ar[r] & H^2(H,\Z) \ar[r] & H^2(G,\widehat T) \ar[r] & H^3(G,\Z) \ar[r] & H^3(H,\Z)\rlap,
}
\end{align}
in view of Shapiro's lemma.
We need to prove that
$\Sha^2_{\cyc}(G, \widehat T)=0$ (see~\eqref{eq:classical-formula-part1} and~\eqref{eq:classical-formula-part2}).

\begin{lem}
\label{lem:sha2cycpvanishes}
Let~$p$ be a prime number.
The $p$\nobreakdash-torsion subgroup of $\Sha^2_{\cyc}(G,\widehat T)$ vanishes
under any of the following assumptions:
\begin{enumerate}[label={\upshape(\alph*)}]
\item $p$ does not divide~$n$;
\item the $p$-Sylow subgroups of~$G$ are cyclic;
\item
the group $\Ker\big(H^{\ab} \to G^{\ab}\big)$ has order prime to~$p$,
and,
denoting by~$H_p$
a $p$\nobreakdash-Sylow subgroup of~$H$, by~$G_p$  a $p$\nobreakdash-Sylow subgroup of~$G$ containing~$H_p$,
and by~$L'$ and~$k'$ their respective fixed fields,  the torus
 $T'=R^1_{L'/k'}\Gm$ over~$k'$
satisfies $\Br_{\nr}(T')=\Br_0(T')$.
\end{enumerate}
\end{lem}

\begin{proof}
The $H$\nobreakdash-module $\widehat T$ is isomorphic to
$\Z^{(G/H) \setminus \{H\}}$,
hence it is a permutation $H$\nobreakdash-module.
We deduce, by Shapiro's lemma,
that
 $\Sha^2_{\cyc}(H,\widehat T)=0$.
It follows
 that $(G:H)=n$ kills  $\Sha^2_{\cyc}(G,\widehat T)$
(see \cite[Corollary~1.5.7]{nsw}).
This takes care of~(a).
By the same token,
the index $(G:S)$ kills $\Sha^2_{\cyc}(G,\widehat T)$
for any cyclic subgroup~$S$ of~$G$,
since $\Sha^2_{\cyc}(S,\widehat T)$ obviously vanishes.
This takes care of~(b).  Now assume that~(c) holds.  We let $M\{p\}$ denote the $p$\nobreakdash-primary
torsion subgroup of an abelian group~$M$ and consider the commutative diagram
with exact rows
\begin{align*}
\xymatrix@R=3ex{
H^2(G,\Z)\{p\} \ar[d] \ar[r]^\alpha & H^2(H,\Z)\{p\} \ar[d] \ar[r] & H^2(G,\widehat T)\{p\} \ar[d]^(.44)\beta \ar[r] & H^3(G,\Z)\{p\} \ar[d]^(.44)\gamma \\
H^2(G_p,\Z)\{p\}  \ar[r] & H^2(H_p,\Z)\{p\}  \ar[r] & H^2(G_p,\widehat T')\{p\}  \ar[r] & H^3(G_p,\Z)\{p\}
}
\end{align*}
induced, thanks to Shapiro's lemma, by the natural morphism of exact sequences
\begin{align*}
\xymatrix@R=3ex{
0 \ar[r] & \Z \ar[r]\ar@{=}[d] & \Z^{G/H} \ar[r]\ar[d] & \widehat T \ar[d]\ar[r] & 0 \\
0 \ar[r] & \Z \ar[r] & \Z^{G_p/H_p} \ar[r] & \widehat T' \ar[r] & 0\rlap.
}
\end{align*}
The first part of~(c) implies that~$\alpha$ is surjective,
since $H^2(G,\Z)\{p\}=\Hom(G^\ab,\Q_p/\Z_p)$ and
 $H^2(H,\Z)\{p\}=\Hom(H^\ab,\Q_p/\Z_p)$.
In addition, a corestriction argument as above shows that~$\gamma$ is injective.
We conclude that $\beta$ is injective.
On the other hand,
the image of $\Sha^2_{\cyc}(G,\widehat T)\{p\}$ by~$\beta$ is contained in
$\Sha^2_{\cyc}(G_p,\widehat T')$, which vanishes by the second part of~(c),
in view of~\eqref{eq:classical-formula-part1} and~\eqref{eq:classical-formula-part2}.
Hence $\Sha^2_{\cyc}(G,\widehat T)\{p\}=0$.
\end{proof}

Proposition~\ref{prop:criterion-for-brauer} in case~(i)
results from
Lemma~\ref{lem:sha2cycpvanishes}~(a)(b).
Cases~(ii) and~(iii) are subcases of~(i), as the $p$\nobreakdash-Sylow subgroups of~$S_p$
are cyclic.
In case~(iv), even  $H^2(G,\widehat T)$ vanishes.
Proposition~\ref{prop:criterion-for-brauer} in case~(v)
results from
Lemma~\ref{lem:sha2cycpvanishes}~(a)(c)
and from case~(iii) of
Proposition~\ref{prop:criterion-for-brauer},
once one notes that if~$n$ is squarefree, then for any prime~$p$ dividing~$n$,
the extension $L'/k'$ appearing in Lemma~\ref{lem:sha2cycpvanishes}~(c) has prime degree
(namely, degree~$p$).

We now turn to cases~(vi)--(viii).
When $G=S_n$ and $n \geq 6$ (so that $H=S_{n-1}$)
or $G=A_n$ and $n \geq 8$ (so that $H=A_{n-1}$),
the restriction map $H^2(G,\Z) \to H^2(H,\Z)$ is onto,
being
Pontrjagin
dual to the natural map $H^{\ab}\to G^{\ab}$
(see \cite[p.~52, l.~7]{nsw}),
which is injective
 if $G=S_n$ and $n \geq 3$
(since
 $H^{\ab}=G^{\ab}=\Z/2\Z$ is generated by a transposition)
or
if $G=A_n$ and $n\geq 6$
(since
 $H^{\ab}=0$);
on the other hand, the restriction map $H^3(G,\Z) \to H^3(H,\Z)$ is injective,
by \cite[Lemma~2.2]{macedo}
and \cite[Proposition~3.1.2]{nsw}
when $G=A_n$ and $n \geq 8$,
and by similar arguments
starting from \cite[Proposition~2.4]{macedo}
when $G=S_n$ and $n \geq 6$.
Thus $H^2(G,\widehat T)=0$ in these cases.
When $G=D_n$ (so that $H=\Z/2\Z$), Lemma~\ref{lem:sha2cycpvanishes}~(a)(b)(c) reduces one to the case where~$n$
is a power of~$2$; in this case, a direct computation using the semidirect product structure of~$D_n$
shows that the restriction map $H^2(G,\widehat T) \to H^2(H,\widehat T)$ is injective
 (see \cite[Lemma~3]{bartels}
and apply \cite[Proposition~3.1.2]{nsw}),
hence $\Sha^2_{\cyc}(G,\widehat T)=0$.
When $n=4$ and $G=S_4$,
one has $\Sha^2_{\cyc}(G,\widehat T)=0$
by Lemma~\ref{lem:sha2cycpvanishes}~(a)(c)
and by Proposition~\ref{prop:criterion-for-brauer}~(vii) for $n=4$.
When $n=6$ and $G=A_6$,
case~(v) of Proposition~\ref{prop:criterion-for-brauer} is applicable.
Finally, when $n \in \{2,3,5,7\}$, case~(iii) of Proposition~\ref{prop:criterion-for-brauer} is applicable.
\end{proof}

\begin{rmks}
\label{rk:brnrtrivial}
(i)
When the extension~$L/k$ is
cyclic or has prime degree,
the torus~$T$ is even a direct factor of a rational variety
(when~$n$ is prime, see
\cite[Proposition~9.1, Proposition~9.5]{ctsanflasque};
when~$L/k$ is cyclic,
the torus~$T$ is itself rational, see \cite[Chapter~2, \textsection4.8]{voskbook}
and the proof
of \cite[Proposition~9.1]{ctsanflasque}).
We note that the case where~$n$ is prime
 includes the case where $L/k$ is an
 almost abelian extension in the sense of \cite[Definition~9.4]{hwfibration}
but is not abelian.

(ii)
The equality $\Br_{\nr}(T) = \Br_0(T)$ fails
whenever the extension $L/k$ is abelian but is not cyclic,
since in this
case $\Sha^2_{\cyc}(G,\widehat T)=H^3(G,\Z) \neq 0$.
Another situation in which this equality fails can be found
in Example~\ref{ex:abelian-alternating} below.
\end{rmks}

\subsection{Central extensions by \texorpdfstring{$\Q/\Z$}{𝐐/𝐙}}
\label{subsec:central extensions}

Keeping the set-up of \textsection\ref{subsec:criterion unramifiedness},
we now explain how the elements
of the group $\Sha^2_{\cyc}(G,\widehat T)$,
which appears in~\eqref{eq:classical-formula-part2},
can be interpreted in terms of central extensions of~$G$ by~$\Q/\Z$.
This turns the practical computation
of $\Br_{\nr}(T)/\Br_0(T)$ into a very concrete question on extensions and their trivialisations.
We illustrate this point of view
in Example~\ref{ex:abelian-alternating}, where we provide
a simple explanation for an example due to Kunyavski\u{\i} of a class of norm tori
such that
 $\Br_{\nr}(T)\neq\Br_0(T)$,
and in Example~\ref{ex:bartels}, where we give an independent proof of
Proposition~\ref{prop:criterion-for-brauer}~(vii)
(Bartels' theorem).

From now on, we assume that~$L$ is a field
and we fix an algebraic closure~$\bar k$ of~$k$ and a $k$\nobreakdash-linear
embedding $L \hookrightarrow \bar k$.  Let~$\wtl L$ be the Galois closure of~$L/k$
in~$\bar k$.
Set $G=\Gal(\wtl L/k)$ and $H=\Gal(\wtl L/L)$.
To interpret $H^2(G,\widehat T)$, we consider central
extensions
\begin{align}
\label{eq:central extension}
\xymatrix{
1 \ar[r] & \Q/\Z \ar[r]^(.55){\iota} & \widetilde G \ar[r]^{\smash[b]{\rho}} & G \ar[r] & 1
}
\end{align}
of~$G$ by~$\Q/\Z$
equipped with a splitting of their pull-back
\begin{align}
\label{eq:pulled extension}
\xymatrix{
1 \ar[r] & \Q/\Z \ar[r] & \rho^{-1}(H) \ar[r] & H \ar[r] & 1\rlap.
}
\end{align}
along the inclusion $H \hookrightarrow G$.
Such data can be encoded by a triple $(\widetilde G,\rho,r)$ formed by a group~$\widetilde G$,
a surjective homomorphism $\rho:\widetilde G\to G$ and a homomorphism $r:\rho^{-1}(H) \to \Q/\Z$
whose restriction to the kernel of~$\rho$ is an isomorphism.
A morphism from a triple
 $(\widetilde G_1,\rho_1,r_1)$
to a triple $(\widetilde G_2,\rho_2,r_2)$ 
is by definition a homomorphism
$\gamma:\widetilde G_1 \to \widetilde G_2$ such that
$\rho_2 \circ \gamma = \rho_1$ and
$r_2 \circ \Big( \gamma|_{\rho_1^{-1}(H)}\Big) = r_1$.
We denote by $\Ext(G,H;\Q/\Z)$ the set of isomorphism classes
of such triples $(\widetilde G,\rho,r)$
and by $\Ext_{\nr}(G,H;\Q/\Z)$ the subset consisting
of the isomorphism classes of unramified triples, in the following sense.

\begin{defn}
\label{def:unramified triple}
A triple 
 $(\widetilde G,\rho,r)$ is \emph{unramified} if
the equality $r(\tilde h_1)=r(\tilde h_2)$
holds for all $\tilde h_1,\tilde h_2 \in \rho^{-1}(H)$ that are conjugate in~$\widetilde G$.
\end{defn}

The next proposition justifies this terminology, in view of~\eqref{eq:classical-formula-part1}
and~\eqref{eq:classical-formula-part2}.

\begin{prop}
\label{prop:bijection h2extgh}
There is a canonical bijection $H^2(G,\widehat T)=\Ext(G,H;\Q/\Z)$.
It induces, by restriction, a bijection $\Sha^2_{\cyc}(G,\widehat T)=\Ext_{\nr}(G,H;\Q/\Z)$.
\end{prop}

\begin{proof}
Let $M = \widehat T \otimes_\Z \Q/\Z$.
As $\widehat T$ is a free $\Z$\nobreakdash-module,
the sequence of $G$\nobreakdash-modules
\begin{align}
\label{eq:firstesproofbij}
\xymatrix{
0 \ar[r] & \widehat T \ar[r] & \widehat T \otimes_\Z \Q \ar[r] & M \ar[r] & 0
}
\end{align}
is exact.
For any subgroup $C \subseteq G$,
the abelian group $H^i(C,\widehat T\otimes_\Z\Q)$ vanishes
for all  $i>0$ as it is
 killed by the order of~$C$ while being
a $\Q$\nobreakdash-vector space.
The sequence~\eqref{eq:firstesproofbij} therefore
induces
isomorphisms $H^1(G,M)= H^2(G,\widehat T)$
and
 $\Sha^1_{\cyc}(G,M)= \Sha^2_{\cyc}(G,\widehat T)$.
In addition, the exact sequence of $G$\nobreakdash-modules
\begin{align}
\label{se:qzqzm}
\xymatrix{
0 \ar[r] & \Q/\Z \ar[r]^(.38)\Delta & (\Q/\Z)^{G/H} \ar[r] & M \ar[r] & 0\rlap,
}
\end{align}
where~$\Delta$ is the diagonal map,
yields
an identification of the group of cochains
\begin{align*}
C^1(G,M)=C^1\big(G,(\Q/\Z)^{G/H}\big)/\big(\Delta\circ C^1(G,\Q/\Z)\big)
\end{align*}
and hence
an identification of the cohomology group
\begin{align}
\label{eq:canonical identification h1gm}
H^1(G,M)=\frac{\big\{(\alpha,\beta) \in Z^2(G,\Q/\Z) \times C^1\big(G,(\Q/\Z)^{G/H}\big);\mkern2mu
 d\beta=\Delta\circ\alpha \big\}}{\big\{(d\gamma,d\delta+\Delta\circ\gamma);\mkern2mu (\gamma,\delta) \in C^1(G,\Q/\Z) \times (\Q/\Z)^{G/H}\big\}}\rlap.
\end{align}
To be precise,
in the notation of~\eqref{eq:canonical identification h1gm},
the group 
$C^1(G,M)$ consists
 of the $\beta$'s modulo the $\Delta\circ \gamma$'s;
the cocycles are defined by the condition that $d\beta=\Delta\circ\alpha$ for some~$\alpha$;
the coboundaries are represented by the $d\delta$'s.

Given a central extension~\eqref{eq:central extension}, let us consider
its push-forward by $\Delta$, i.e.\ the extension
\begin{align}
\label{eq:pushed extension}
\xymatrix{
1 \ar[r] & (\Q/\Z)^{G/H} \ar[r] &
\displaystyle \frac{(\Q/\Z)^{G/H} \rtimes \widetilde G}{((-\Delta)\times \iota)(\Q/\Z)} \ar[r] & G \ar[r] & 1\rlap,
}
\end{align}
where~$\widetilde G$ acts on $(\Q/\Z)^{G/H}$ through~$\rho$.
The extension~\eqref{eq:pushed extension} is also the induction from~$H$ to~$G$
of~\eqref{eq:pulled extension}
in the sense of \cite[\textsection2.4]{stixtrading}, since~\eqref{eq:pulled extension}
is the extension obtained from~\eqref{eq:pushed extension} by applying
the functor~$\mathrm{sh}^2$ of \emph{loc.\ cit.}\ (see \textsection2.3.4 and \textsection2.4.2 of \emph{op.\ cit.}).
By Corollary~15
of \emph{op.\ cit.},
splittings of~\eqref{eq:pulled extension}
are therefore in one-to-one correspondence with splittings
of~\eqref{eq:pushed extension}
 up to conjugation by~$(\Q/\Z)^{G/H}$.

Recall that $H^2(G,\Q/\Z)$
is in canonical bijection with the set of isomorphism classes of central extensions
of~$G$ by~$\Q/\Z$; if $s:G \to \widetilde G$
is
a set-theoretic section of~$\rho$,
the class of~\eqref{eq:central extension} is given by the cocycle $\alpha \in Z^2(G,\Q/\Z)$
defined by $\iota(\alpha(\sigma,\tau))= s(\sigma)s(\tau)s(\sigma\tau)^{-1}$
(see \cite[Chapter~IV, Theorem~3.12]{brown}).
Once~$s$ is fixed,
the datum of a splitting of~\eqref{eq:pushed extension}
is equivalent to that of a map $\beta:G \to (\Q/\Z)^{G/H}$ such that
$(-\beta) \times s:G \to (\Q/\Z)^{G/H}\rtimes \widetilde G$
is a homomorphism modulo $((-\Delta)\times \iota)(\Q/\Z)$, or in other words,
such that $d\beta=\Delta\circ\alpha$.
For $(\gamma,\delta) \in C^1(G,\Q/\Z) \times (\Q/\Z)^{G/H}$,
 conjugating $(-\beta)\times s$ by~$\delta$
and replacing~$s$ with $(\iota\circ \gamma)s$
amounts to adding
$(d\gamma,d\delta+\Delta\circ\gamma)$
to the pair $(\alpha,\beta)$.

Altogether,
we have now obtained a 
canonical map $\Ext(G,H;\Q/\Z) \to H^1(G,M)$.
By unwinding its construction and applying  \cite[Chapter~IV, Theorem~3.12]{brown},
one readily checks that it is bijective,
thus proving the first part of Proposition~\ref{prop:bijection h2extgh}.

To verify the second part, let us fix a central extension~\eqref{eq:central extension},
a set-theoretic section~$s$ of~$\rho$ and
a splitting of~\eqref{eq:pushed extension}, corresponding on the one hand to a triple $(\widetilde G,\rho,r)$
and on the other hand to  $\alpha \in Z^2(G,\Q/\Z)$
and $\beta \in C^1\big(G,(\Q/\Z)^{G/H}\big)$ as above, with $d\beta=\Delta\circ\alpha$.
Denoting by $\beta(g)(g')$, for $g,g' \in G$, the image of $g'H$ by $\beta(g):G/H\to \Q/\Z$,
we shall now
conclude the proof
by establishing the equivalence of the following conditions:
\begin{enumerate}[label={\upshape(\roman*)}]
\item the triple  $(\widetilde G,\rho,r)$ is unramified;
\item $\beta(h)(g') = \beta(h)(1)$ for all $h \in H$ and  $g' \in G$ such
that $g'^{-1}hg'\in H$;
\item $\beta(g)(g')=\beta(g)(g'')$ for all $g,g',g'' \in G$
such that $g'^{-1}gg' \in H$ and $g''^{-1}gg'' \in H$;
\item the image of $\beta$ in $H^1(G,M)$ belongs to the subgroup $\Sha^1_{\cyc}(G,M)$.
\end{enumerate}

The splitting of~\eqref{eq:pulled extension} induced by the given splitting of~\eqref{eq:pushed extension}
is $h \mapsto \iota(-\beta(h)(1))s(h)$ (see \cite[\textsection2.4.2]{stixtrading}); hence
$r(s(h))=\beta(h)(1)$ for every $h\in H$.
Writing elements of $\rho^{-1}(H)$ as $\iota(u)s(h)$ with $u \in \Q/\Z$ and $h \in H$,
we deduce that~(i) holds if and only if
\begin{align}
\label{eq:ibetas}
\iota\big(\beta(h_2)(1)-\beta(h_1)(1)\big) = s(g)s(h_2)s(g)^{-1}s(h_1)^{-1}
\end{align}
for all $h_1,h_2 \in H$ and $g \in G$ such that $h_1=gh_2g^{-1}$.
As the right-hand side of~\eqref{eq:ibetas} is equal to $\iota(\alpha(g,h_2)-\alpha(h_1,g))$
and as $d\beta=\Delta\circ \alpha$, we can rewrite~\eqref{eq:ibetas}
as
$\beta(h_2)(1) = \beta(h_2)(g^{-1})$.
This proves the equivalence (i)$\Leftrightarrow$(ii).
The implications (iv)$\Rightarrow$(iii)$\Rightarrow$(ii) are immediate.
To see that (ii)$\Rightarrow$(iii), we remark
that as $d\beta=\Delta\circ\alpha$,
the image $\bar\beta \in C^1(G,M)$ of~$\beta$ is a cocycle,
hence $\bar\beta(g)=\bar\beta(g')+g'(\bar\beta(g'^{-1}gg')) - g(\bar\beta(g'))$
for all $g,g' \in G$.
Evaluating at~$g'$ and~$g''$
and noting that $\beta(g')(g')=\beta(g')(g^{-1}g')$ and $\beta(g')(g'')=\beta(g')(g^{-1}g'')$
if~$g'$ and~$g''$ satisfy the hypotheses of~(iii), one deduces that (ii)$\Rightarrow$(iii).

It remains to prove that (iii)$\Rightarrow$(iv).
We assume that~(iii) holds, fix a cyclic subgroup $C \subseteq G$
and verify that the image of~$\bar\beta$ in $H^1(C,M)$ vanishes.
To this end, we choose a generator~$g$ of~$C$
and
representatives $g_1,\dots,g_N \in G$ of the orbits of~$C$ on~$G/H$.
Let $n_1,\dots,n_N$ be the lengths of these orbits.
Condition~(iii) implies the following:
\begin{enumerate}[label={\upshape(\roman*)}]
\item[(iii')]
 $\beta(g^n)(g_i)=\beta(g^n)(g_j)$
for all $n,i,j$
such that~$n_i$ and~$n_j$ both divide~$n$.
\end{enumerate}
On the other hand, the exact sequence~\eqref{se:qzqzm} induces an exact sequence
\begin{align}
\label{eq:h1ch1ch1ch2c}
\xymatrix{
H^1(C,\Q/\Z) \ar[r] & H^1\big(C,(\Q/\Z)^{G/H}\big) \ar[r] & H^1(C,M) \ar[r] & H^2(C,\Q/\Z)\rlap.
}
\end{align}
As $H^2(C,\Q/\Z)=0$
(see \cite[Proposition~1.7.1]{nsw} and note that~$C$ is cyclic and~$\Q/\Z$ is divisible),
there exists $\beta' \in Z^1\big(C,(\Q/\Z)^{G/H}\big)$
whose image in $H^1(C,M)$ equals the image of~$\bar\beta$.
Condition~(iii') still holds with~$\beta$ replaced with~$\beta'$.
Set $\lambda_i=\beta'(g^{n_i})(g_i)$.
For any integer~$n$ divisible by~$n_i$,
we have
 $\beta'(g^n)=(1+g^{n_i}+g^{2n_i}+\cdots+g^{n-n_i})(\beta'(g^{n_i}))$
since~$\beta'$ is a cocycle,
and hence $\beta'(g^n)(g_i)=\frac{n}{n_i}\lambda_i$.
Thus, 
thanks to~(iii'),
we deduce 
from Lemma~\ref{lem:elementary}
the existence of $\lambda \in \Q/\Z$ such that $\lambda_i=n_i\lambda$ for all~$i$.

Let~$m$ be the order of~$g$.
Decomposing $(\Q/\Z)^{G/H}$ as a direct sum indexed by the orbits of~$C$
on~$G/H$
and using Shapiro's lemma on each summand,
we rewrite~\eqref{eq:h1ch1ch1ch2c} as
\begin{align*}
\xymatrix{
\displaystyle\bigg(\frac 1 m \Z\bigg)/\mkern2mu\Z \ar[r] & \displaystyle\bigoplus_{i=1}^N\Bigg(\bigg(\frac{n_i}m\Z\bigg)/\mkern2mu\Z\Bigg)
\ar[r] & H^1(C,M) \ar[r] & 0\rlap,
}
\end{align*}
where the first map is $x \mapsto (n_ix)_{1 \leq i \leq N}$.  As the class of~$\beta'$ in the second group
is $(\lambda_i)_{1 \leq i \leq N}$, we see that~$\beta'$ comes from the left, hence $\bar\beta$ vanishes
in~$H^1(C,M)$
and~(iv) is proved.
\end{proof}

\begin{example}[Kunyavski\u{\i}~\cite{kunyavskii}]
\label{ex:abelian-alternating}
Let~$k$ be a field of characteristic~$0$ and~$L$ be a quartic field extension of~$k$.
If a  Galois
 closure of~$L/k$
has Galois group $G=A_4$,
then the torus $T=R^1_{L/k}\Gm$ over~$k$
satisfies $\Br_{\nr}(T)/\Br_0(T)=\Z/2\Z$.

To verify this claim, we first note that
 $H^2(G,\widehat T)=\Z/2\Z$, as one deduces from~\eqref{se:h2ghatt}.
Proposition~\ref{prop:bijection h2extgh} now
reduces us to showing that
$\Ext_{\nr}(A_4,A_3;\Q/\Z)=\Ext(A_4,A_3;\Q/\Z)$.
To this end, we fix a central extension~\eqref{eq:central extension}
and a retraction $r:\rho^{-1}(H)\to \Q/\Z$
of~\eqref{eq:pulled extension}, with $G=A_4$ and $H=A_3$,
and let  $\tilde h_1,\tilde h_2 \in \rho^{-1}(H)$ and $\tilde g \in \widetilde G$
satisfy $\tilde h_1=\tilde g \tilde h_2 \tilde g^{-1}$.
Set $h_i=\rho(\tilde h_i)$ and $g=\rho(\tilde g)$. Now $h_1,h_2\in A_3$ and $g\in A_4$ are such that
$h_1=gh_2g^{-1}$, and this implies that $h_1=h_2=1$ or $g \in A_3$.
In both cases, we deduce that~$\tilde h_1$
and~$\tilde h_2$ are conjugate in~$\rho^{-1}(H)$, and hence that $r(\tilde h_1)=r(\tilde h_2)$, as desired.
\end{example}

\begin{example}
\label{ex:bartels}
Taking up the notation of Proposition~\ref{prop:criterion-for-brauer}, let us assume that $G=D_n$
and explain how
 the point of view of central extensions
leads to a short proof of Bartels' theorem that $\Br_{\nr}(T)=\Br_0(T)$.
According to Proposition~\ref{prop:bijection h2extgh}, we have to check
that any unramified triple $(\widetilde G,\rho,r)$ is isomorphic to the trivial triple
$(\Q/\Z \times G, \mathrm{pr}_2, \mathrm{pr}_1)$.
Here $H=\Z/2\Z$ and the inclusion $H \hookrightarrow G$ admits a retraction, so that
the restriction map $H^1(G,\Q/\Z) \to H^1(H,\Q/\Z)$ is surjective, which implies
the desired conclusion
when the central extension~\eqref{eq:central extension} is trivial
(i.e.\ when $\widetilde G = \Q/\Z \times G$ and $\rho=\mathrm{pr}_2$)
since in this case the splittings of~\eqref{eq:pulled extension}
are parametrised by $H^1(H,\Q/\Z)$
and the automorphisms of~\eqref{eq:central extension} by $H^1(G,\Q/\Z)$.
Let us now fix an unramified triple $(\widetilde G,\rho,r)$ such that~\eqref{eq:central extension} is nontrivial.
As  $H^2(D_n,\Q/\Z)=0$
when~$n$ is odd
and $H^2(D_n,\Q/\Z)=\Z/2\Z$ when~$n$ is even
(see \cite[Theorem~5.2, Theorem~5.3]{handel}),
the integer~$n$ must be even and~\eqref{eq:central extension}
is the unique nontrivial central extension of~$D_n$ by~$\Q/\Z$ up to isomorphism.
Presenting~$D_m$ as $\langle \sigma,\tau;\mkern2mu\sigma^2=\tau^{m}=1, \sigma\tau\sigma=\tau^{-1}\rangle$
for any~$m$, setting
\begin{align}
\widetilde D_n = (\Q/\Z \times D_{2n}) \mkern1mu/ \mkern-3mu\left\langle\big( 1/2,\tau^n\big)\right\rangle
\end{align}
and letting
$\pi:\widetilde D_n \to D_n$ be induced by the second projection
(noting that $D_{2n}/\langle \tau^n \rangle=D_n$),
one checks that~$\pi$ does not admit a section, as~$n$ is even; hence the central extension
given by~$\widetilde D_n$ and~$\pi$ must be isomorphic to~\eqref{eq:central extension} and
we may assume that $\widetilde G = \widetilde D_n$ and $\rho=\pi$.
Now the classes $\tilde h_1,\tilde h_2 \in \widetilde D_n$
of $(0,\sigma),(0,\sigma\tau^n) \in \Q/\Z \times D_{2n}$
are conjugate in~$\widetilde D_n$ (indeed $\tau^{-n/2}\sigma\tau^{n/2}=\sigma\tau^{n}$ in $D_{2n}$)
and belong to $\pi^{-1}(H)$,
but differ by the class of $(1/2,1)$ in~$\widetilde D_n$, so that $r(\tilde h_1)\neq r(\tilde h_2)$,
a contradiction.
\end{example}

\section{The variety \texorpdfstring{$W$}{W} and Conjecture \texorpdfstring{$\cFplus$}{F₊}}
\label{sec:W}

We recall
in \textsection\textsection\ref{subsubsec:parametersdef}--\ref{subsubsec:geometry}
the
definition and geometry of the quasi-affine variety~$W$
which was introduced in
 \cite[\textsection 9.2.2]{hwfibration}
and partial compactifications of
which had previously appeared
in~\cite[\textsection3.3]{skorodescent}, \cite[p.~391]{ctskodescent}, and
in~\cite[\textsection4.4]{skobook}.
We discuss
in \textsection\textsection\ref{subsec:brauergroup}--\ref{subsubsec:localintpoint}
 some of its basic properties, showing, in particular, that its
Brauer group consists only of constant classes.
When the ground field is a number field, the adelic points of~$W$ are the subject
of two key conjectures studied in this article: Conjecture~$\cF$ and Conjecture~$\cFplus$.
We state them in~\textsection\ref{sec:statementFplus}.
Conjecture~$\cF$ is another name for \cite[Conjecture~9.1]{hwfibration}, while the formulation
of Conjecture~$\cFplus$ is new.
As we shall see later in~\textsection\ref{sec:comparing},
these two conjectures are equivalent when their parameters are allowed to vary.

\subsection{Parameters and definition}
\label{subsubsec:parametersdef}

Let~$k$ be a field of characteristic~$0$.
We denote by~$\sP$ the collection of all triples $\pi=(M, (L_m)_{m \in M}, (b_m)_{m \in M})$
consisting of a finite closed subset $M \subset \A^1_k$
together with the data, for each $m \in M$, of a finite field
extension $L_m$ of the residue field~$k(m)$ of~$m$ and of an element $b_m \in k(m)^*$.
Given $\pi \in \sP$, we
let $a_m \in k(m)$ denote the value at~$m$
of the regular function~$t$ on $\A^1_k=\Spec(k[t])$
and let~$F_m$ denote the singular locus of the variety
$R_{L_m/k}(\A^1_{L_m}) \setminus R_{L_m/k}(\mathbf{G}_{\mathrm{m},{L_m}})$.
Following \cite[\textsection 9.2.2]{hwfibration}, for $\pi \in \sP$, we consider
the morphism
\begin{align}
\label{eq:morphismdefW}
\left(\A^2_k \setminus \{(0,0)\}\right) \times \prod_{m \in M} \left(R_{L_m/k}(\A^1_{L_m}) \setminus F_m\right) \to \prod_{m\in M} R_{k(m)/k}(\A^1_{k(m)})
\end{align}
defined by $(\lambda,\mu,(z_m)_{m\in M}) \mapsto (b_m(\lambda-a_m\mu)-N_{L_m/k(m)}(z_m))_{m \in M}$,
where
$\lambda,\mu$
are the coordinates of~$\A^2_k$
and~$z_m$ stands for a point of $R_{L_m/k}(\A^1_{L_m})$.

\begin{defn}
\label{def:w}
The variety~$W$ over~$k$
associated with $\pi \in \sP$ is the fibre, above
the
rational point~$0$,
 of the morphism~\eqref{eq:morphismdefW}.
\end{defn}

\begin{rmk}
\label{rk:wwpiso}
Let $\pi=(M, (L_m)_{m \in M}, (b_m)_{m \in M}) \in \sP$.
Let $M' = \{m\in M\mkern1mu;\mkern2mu L_m \neq k(m)\}$.
Let $\pi'=(M', (L_m)_{m \in M'}, (b_m)_{m \in M'})$.
The varieties~$W$ and~$W'$ associated with~$\pi$ and~$\pi'$ are naturally isomorphic.
\end{rmk}

\subsection{Geometry}
\label{subsubsec:geometry}

For ease of reference, we collect in the next proposition some elementary facts about the geometry of~$W$
that already appear in the proof of \cite[Proposition~9.9]{hwfibration}.
We set $U=\P^1_k \setminus M$
and let $p:W \to \P^1_k$ denote the map $(\lambda,\mu,(z_m)_{m\in M}) \mapsto [\lambda:\mu]$.

\begin{prop}
\label{prop:fibresofp}
Let $\pi\in \sP$.
The morphism~$p$ is smooth (hence~$W$ is a smooth variety).  Its fibres over~$U$ are geometrically integral.
For each $m \in M$, the fibre $W_m=p^{-1}(m)$ is integral and the algebraic closure of~$k(m)$
in its function field is~$L_m$.
\end{prop}

\begin{rmk}
\label{rmk:wwp}
The variety~$W'$ defined in the same way as~$W$ except that one does not remove~$F_m$
from $R_{L_m/k}(\A^1_{L_m})$ in~\eqref{eq:morphismdefW} was first considered in
\cite[\textsection3.3]{skorodescent}.  It is a smooth variety (\emph{loc.\ cit.}, Lemma~3.3.1)
and it contains~$W$ as the complement of a codimension~$2$ closed subset.
More precisely, the variety~$W$ is the locus where
the faithfully flat morphism $p':W' \to \P^1_k$ that extends~$p$
is smooth.
\end{rmk}

\subsection{Brauer group}
\label{subsec:brauergroup}

The vertical Brauer group of~$p$ is reduced to constant classes,
as was shown in
\cite[Proof of Proposition~9.9]{hwfibration}.  More generally, so is the algebraic Brauer group of~$W$,
according to \cite[\textsection1]{ctcalcul} (which formally deals with~$W'$, in the notation of Remark~\ref{rmk:wwp}, but 
 by purity for the Brauer group \cite[Corollaire~6.2]{grothbr3}
applies to~$W$ too, since the smooth
varieties~$W'$ and~$W$ differ
by a codimension~$2$ closed subset)
or to \cite[Lemma~5.1]{caoweixu}.
The following proposition generalises this assertion to the full Brauer group.

\begin{prop}
\label{prop:brauerw}
Let $\pi \in \sP$.  The pull-back map $\Br(k) \to \Br(W)$ is onto.
\end{prop}

\begin{proof}
Only until the end of this proof, let us modify the definition of~$\sP$ by allowing~$L_m$
to be an arbitrary nonzero finite étale $k(m)$\nobreakdash-algebra, not necessarily a field.
We shall prove that the proposition holds in this slightly more general situation, with the
same definition for~$W$, by induction on the dimension of~$W$.
If $\dim(W)<3$, then $L_m=k(m)$ for all $m \in M$, hence $W = \A^2_k \setminus \{(0,0)\}$,
so that $\Br(k) = \Br(W)$ (see \cite[Theorem~3.7.1, Theorem~6.1.1]{ctskobook}).
Let us now assume that $\dim(W)\geq 3$ and that the conclusion of the proposition holds
for lower dimensions of~$W$ (letting both~$k$ and~$\pi$ vary).

In order to prove that the pull-back map $\Br(k) \to \Br(W)$ is onto, we may and will assume
that~$k$ is algebraically closed, since the algebraic Brauer
group of~$W$ is reduced to constant classes by \cite[Proposition~1.1, Proposition~1.2 (ii)--(iii)]{ctcalcul}.

Setting $d_m=\dim_k L_m$
and letting $\lambda,\mu,(z_{m,j})_{m \in M, 1 \leq j \leq d_m}$ denote
the coordinates of $\A^2_k \times \prod_{m\in M} \A^{d_m}_k$,
the variety~$W$ is then isomorphic to the subvariety of this affine space defined by the following equations:
$\prod_{j=1}^{d_m} z_{m,j} = b_m(\lambda-a_m\mu)$
for each $m \in M$;
at most one of the coordinates $z_{m,1},\dots,z_{m,d_m}$ vanishes for each $m \in M$;
and $(\lambda,\mu)\neq (0,0)$.

Let us fix $m \in M$ such that $d_m>1$.
The generic fibre~$V$ of the projection $W \to \A^1_k$ to the coordinate $z_{m,d_m}$
is a variety of the form~$W$ associated with a parameter in~$\sP$
over the function field of~$\A^1_k$. By the induction hypothesis, we deduce
that $\Br(k(\A^1_k))$ surjects onto $\Br(V)$.
Hence $\Br(V)=0$, in view of Tsen's theorem.
Now, as the schemes~$V$ and~$W$ are irreducible and regular and share the same generic point,
the group~$\Br(W)$ injects into~$\Br(V)$.
We conclude that $\Br(W)=0$; the proof is complete.
\end{proof}

\begin{rmk}
When~$k$ is a number field, it was proved in
\cite[Corollary~9.10]{hwfibration} that if, for every $\pi \in \sP$ and every
finite place~$v_0$ of~$k$, the variety~$W$ satisfies strong approximation
off~$v_0$, then Conjecture~9.1 of \emph{op.\ cit.}\ is true.
This strong approximation property holds in several cases
(see \cite[Theorem~9.11]{hwfibration} and \cite[Theorem~1.2]{browningschindler})
and it is natural to hope that it may hold in general.
Proposition~\ref{prop:brauerw} supports this hope, as it
shows that the Brauer--Manin obstruction to strong approximation on~$W$ always vanishes.
\end{rmk}

\subsection{Canonical ramified Brauer classes}
\label{subsec:canonical ramified}

Given $\pi \in \sP$,
the classes
\begin{align}
\Cores_{L_m/k}(z_m,\chi) \in \Br(k(W))
\end{align}
for $m \in M$ and $\chi \in H^1(L_m,\Q/\Z)$,
where~$z_m$ denotes
 the regular function
on $W \otimes_k L_m$ obtained by pulling back
the tautological
regular function on $(R_{L_m/k}(\A^1_{L_m})) \otimes_k L_m$,
will play a special rôle in the article.

\begin{prop}
\label{prop:residueofcores}
Let $\pi\in \sP$.  Let $m \in M$ and $\chi \in H^1(L_m,\Q/\Z)$.
\begin{enumerate}
\item The class $\Cores_{L_m/k}(z_m,\chi)$ belongs to $\Br(p^{-1}(\P^1_k\setminus \{m\}))$.
\item Its residue at the generic point of $W_m=p^{-1}(m)$
is equal to the image of~$\chi$ by the restriction map
$H^1(L_m,\Q/\Z) \to H^1(k(W_m),\Q/\Z)$
(see Proposition~\ref{prop:fibresofp}).
\end{enumerate}
\end{prop}

\begin{proof}
Assertion~(1) follows from the remark that the regular function~$z_m$ on $W \otimes_k L_m$
is invertible
on $p^{-1}(\P^1_k \setminus \{m\}) \otimes_k L_m$.
Assertion~(2)
results from \cite[Proposition~1.1.2 and Proposition~1.1.3]{ctsd94}.
\end{proof}

\subsection{Local integral points}
\label{subsubsec:localintpoint}

We now assume that~$k$ is a number field.
Given a finite set~$S$ of places of~$k$, we denote by $\sOint_S$ the ring of $S$\nobreakdash-integers
of~$k$ and by $p:\sW\to \P^1_{\sOint_S}$ the integral model of $p:W\to \P^1_k$
obtained by replacing,
in the definition of~$W$,
all occurrences of the fields~$k$, $k(m)$, $L_m$, respectively, by their rings of $S$\nobreakdash-integers
$\sOint_S$, $\sOint_{k(m),S}$, $\sOint_{L_m,S}$,
and~$F_m$ by the singular locus~$\sF_m$ of the scheme
$R_{\sOint_{L_m,S}/\sOint_S}(\A^1_{\sOint_{L_m,S}}) \setminus R_{\sOint_{L_m,S}/\sOint_S}(\mathbf{G}_{\mathrm{m},{\sOint_{L_m,S}}})$.  We will refer to~$\sW$
as the \emph{standard integral model} of~$W$.

For $m \in \P^1_k$, we denote by~$\mtilde$ the Zariski closure of~$m$
in~$\P^1_{\sOint_S}$,
and we set $\Mtilde = \bigcup_{m \in M} \mtilde$.

The next proposition records a useful interpretation for the integral local points of~$\sW$
at the places of~$k$ outside of~$S$, when~$S$ is large enough.

\begin{prop}
\label{prop:localintegralpoint}
Assume that~$k$ is a number field. Let $S\subset \Omega$ be a finite subset containing the
archimedean places, the finite places that ramify in~$L_m$ for some $m \in M$
and the finite places above which at least one of the~$b_m$ for $m \in M$ fails to be a unit.
Assume, finally, that~$S$ is large enough that
$\Mtilde \cup\widetilde\infty$
is \'etale over~$\sOint_S$.
For any $v \in \Omega \setminus S$, the set $\sW(\sOint_v)$ can be identified with the set of families
$(\lambda_v,\mu_v,(z_{m,u})_{m\in M,u|v})$
consisting of two elements $\lambda_v, \mu_v \in \sOint_v$ at least one of which is a unit
and, for each $m \in M$ and each place~$u$ of~$L_m$ dividing~$v$,
of an element $z_{m,u} \in \sOint_{(L_m)_u}$, such that for all $m \in M$, the following properties
are satisfied:
\begin{itemize}
\item letting the product run over the places~$u$ of~$L_m$ dividing~$v$ and letting~$w$ denote the trace of~$u$
on~$k(m)$, the equality
\begin{align*}
\prod N_{(L_m)_u/k(m)_w}(z_{m,u})=b_m(\lambda_v-a_m\mu_v)
\end{align*}
holds in $k(m)_w$;
\item there exists at most one place~$u$ of~$L_m$ dividing~$v$ such that $z_{m,u}$ is not a unit;
\item if such a place~$u$ of~$L_m$ exists, then it has degree~$1$ over~$v$.
\end{itemize}
\end{prop}

\begin{proof}
We need only verify that for any $\sOint_v$\nobreakdash-point
$(z_{m,u})_{u|v}$ of $R_{\sOint_{L_m,S}/\sOint_S}(\A^1_{\sOint_{L_m,S}})$,
the reduction modulo~$v$ of $(z_{m,u})_{u|v}$
lies in $\sF_m(\Fv{v})$ if and only if $z_{m,u}$ fails to be a unit at~$u$
either for at least two distinct~$u$ or for at least
one~$u$ of degree~$>1$ over~$v$.  This local
assertion can be checked after replacing~$k_v$ with any
unramified extension of~$k_v$, in particular it is enough to check it when~$v$ splits completely in~$L_m$,
in which case it is clear.
\end{proof}

\subsection{Statement of Conjecture \texorpdfstring{$\cFplus$}{F₊}}
\label{sec:statementFplus}

We take up the notation introduced in
\textsection\textsection\ref{subsubsec:parametersdef}--\ref{subsubsec:geometry}
and assume, in addition, that~$k$ is a number field:
with any $\pi \in \sP$ are associated a variety~$W$ and a smooth morphism $p:W\to \P^1_k$,
with split fibres above $U=\P^1_k \setminus M$.
For any $c \in \P^1_k$, we set $W_c=p^{-1}(c)$.
The following conjecture on the arithmetic of~$W$ was put forward in
\cite[\textsection9]{hwfibration}.
To be precise,
it is equivalent to \emph{op.\ cit.}, Conjecture~9.1,
according to
\emph{op.\ cit.}, Proposition~9.9.

\begin{conjF}
Let $\pi \in\sP$.
The subset $\displaystyle\bigcup_{c \in U(k)} W_c(\A_k)$ is dense in $W(\A_k)$.
\end{conjF}

The goal of \textsection\ref{sec:statementFplus}
is to propose a more general formulation, which we call Conjecture~$\cFplus$
and view as an improved substitute for Conjecture~$\cF$.  Before stating it, we need
to introduce some additional notation related to the parameters on which
Conjecture~$\cFplus$ depends.

We let $\sP_+$ be the collection
of all
quadruples $\pi_+=(M, (L_m)_{m \in M}, (b_m)_{m \in M}, (K_m)_{m\in M})$
consisting of a triple $\pi=(M, (L_m)_{m \in M}, (b_m)_{m \in M})$ belonging to~$\sP$
together with the data,
for each $m \in M$, of a finite abelian field extension~$K_m$ of~$L_m$.

Given $\pi_+ \in \sP_+$, we set
\begin{align}
C_m = \Ker\big(H^1(L_m,\Q/\Z) \to H^1(K_m,\Q/\Z)\big) = \Hom(\Gal(K_m/L_m),\Q/\Z)
\end{align}
and denote by $B \subseteq \Br(p^{-1}(U))$ the subgroup generated by the
classes
$\Cores_{L_m/k}(z_m,\chi)$
 for $m \in M$ and $\chi \in C_m$
(see Proposition~\ref{prop:residueofcores}~(1)).
This is a finite subgroup.
Finally, for $c \in U(k)$, we denote by $W_c(\A_k)^B$ the subset of $W_c(\A_k)$
consisting of those adelic points that are orthogonal, with respect to the
Brauer--Manin pairing, to the image of~$B$
by the restriction map
$\Br(p^{-1}(U)) \to \Br(W_c)$.

\begin{conjFplus}
Let $\pi_+\in\sP_+$.
The subset $\displaystyle\bigcup_{c \in U(k)} W_c(\A_k)^B$ is dense in $W(\A_k)$.
\end{conjFplus}

Conjecture~$\cF$ can be seen as the special case of Conjecture~$\cFplus$
for those $\pi_+ \in \sP_+$ that satisfy $K_m=L_m$ for all $m \in M$.
In \textsection\textsection\ref{sec:fibration}--\ref{sec:newapplications},
we shall see that Conjecture~$\cFplus$ should be expected to be just as true
as Conjecture~$\cF$, and that the flexibility provided by
allowing nontrivial abelian extensions $K_m/L_m$ leads at the same time
to a better theory
and to more general results about rational points in fibrations
into rationally connected varieties over the projective line.

\section{Fibration theorem for rational points}
\label{sec:fibration}

In this section, we show that the fibration method for proving the density
of rational points in the Brauer--Manin set works for fibrations into
proper smooth rationally connected varieties, if Conjecture~$\cFplus$ holds
for certain parameters~$\pi_+$ associated with the fibration.

\subsection{Main theorem}

Theorem~\ref{th:fibration} should be compared with
\cite[Theorem~9.17]{hwfibration}, whose statement, based on
Conjecture~$\cF$, is very similar but contains an unpleasant
technical assumption
 absent from Theorem~\ref{th:fibration}
(namely, the surjectivity of \cite[(9.9)]{hwfibration};
this assumption is satisfied when~$k$ is totally imaginary or~$M$ contains
a rational point, by \cite[Remark~9.18~(ii)]{hwfibration}).
The proof of Theorem~\ref{th:fibration} refines ideas that were elaborated
over a long
series of works, among which
\cite{harariduke,hararifleches,ctsd94,ctsksd98,hwfibration}.

\begin{thm}
\label{th:fibration}
Let~$X$ be a smooth, irreducible variety over a number field~$k$, endowed with
a morphism $f:X\to \P^1_k$ with geometrically irreducible generic fibre.
Let $U \subset \P^1_k$ be an open subset over which the fibres of~$f$ are split,
with $\infty \in U$.  Let $A \subset \Br(f^{-1}(U))$ be a finite subgroup.
Let $H \subset \P^1_k$ be a Hilbert subset.
Let $M = \P^1_k \setminus U$.

For each $m \in M$, suppose given
an irreducible component~$Y_m$ of multiplicity~$1$ of
the fibre $X_m=f^{-1}(m)$
and a finite abelian extension $E_m/k(Y_m)$ such that
the residue of any element of~$A$ at the generic point of~$Y_m$
belongs to the kernel of the restriction map
$H^1(k(Y_m),\Q/\Z)\to H^1(E_m,\Q/\Z)$.
Let~$L^0_m$ (resp.\ $K^0_m$) denote the algebraic closure of~$k(m)$ in~$k(Y_m)$ (resp.\ in~$E_m$).
For each $m \in M$, suppose given a finite extension~$L_m$ of~$L^0_m$
and a finite abelian extension~$K_m$ of~$L_m$
in which $K^0_m$ can be embedded $L^0_m$\nobreakdash-linearly.

Assume that
for all choices of $(b_m)_{m \in M} \in \prod_{m \in M}k(m)^*$,
 Conjecture~$\cFplus$ holds for
the parameter $\pi_+=(M, (L_m)_{m \in M}, (b_m)_{m \in M}, (K_m)_{m\in M})$.
Then the subset
\begin{align}
\bigcup_{c \in U(k)\cap H} X_c(\A_k)^A
\end{align}
is dense in
$X(\A_k)^{(A+f^*_\eta\Br(\eta)) \cap \Br(X)}$.
\end{thm}

For the reader's ease, we summarise the various fields involved in the above statement with a commutative diagram:
\begin{align*}
\xymatrix@R=1ex@C=1em{
K_m && E_m \\
& K_m^0  \ar@{.}[ul]  \ar@{-}[ur] \\
L_m \ar@{-}[uu] && k(Y_m) \ar@{-}[uu] \\
& L_m^0 \ar@{-}[uu] \ar@{-}[ur] \ar@{-}[ul] \\
\\
& k(m)  \ar@{-}[uu]
}
\end{align*}

\begin{proof}
The arguments used to deduce
\cite[Theorem~9.22]{hwfibration}
 from
\cite[Theorem~9.17]{hwfibration}
also prove, in the exact same way,
that
Theorem~\ref{th:fibration} in the special case where $H=\P^1_k$ implies
Theorem~\ref{th:fibration} in general.
Hence we may, and will, assume that $H=\P^1_k$.

Let~$\Omega$ denote the set of places of~$k$.
Let $(x_v)_{v \in \Omega} \in X(\A_k)^{(A+f^*_\eta\Br(\eta)) \cap \Br(X)}$
be the adelic point that we shall approximate.
Our goal is to produce $c \in U(k)$ and $(x_v''')_{v \in \Omega} \in X_c(\A_k)^A$
arbitrarily close to $(x_v)_{v \in \Omega}$ in $X(\A_k)$.
We break up the proof into five steps.

\begin{step}
\label{step:1}
Choice of the parameter~$\pi_+$.
\end{step}

In Step~\ref{step:3}, we shall apply Conjecture~$\cFplus$
to the parameter
\begin{align}
\label{eq:piplus}
\pi_+=(M, (L_m)_{m \in M}, (b_m)_{m \in M}, (K_m)_{m\in M})
\end{align}
for a certain choice
of a family $(b_m)_{m \in M} \in \prod_{m \in M} k(m)^*$.
Step~\ref{step:1} is devoted to specifying this choice, using the next lemma.

For its statement, we take up the notation of~\textsection\ref{subsubsec:parametersdef}
and~\textsection\ref{sec:statementFplus},
where a variety~$W$,
a morphism
$p:W\to \P^1_k$
and various groups
$(C_m)_{m\in M}$ and $B \subseteq \Br(p^{-1}(U))$
were associated with the parameter~$\pi_+$ defined by~\eqref{eq:piplus}.
In addition, we let $U^0=U \setminus \{\infty\}$, $X^0=f^{-1}(U^0)$
and $W^0=p^{-1}(U^0)$.
We stress that in the statement of Lemma~\ref{lem:choiceofbm}, the
variety~$W^0$, the morphism~$p$ and the group~$B$ depend on $(b_m)_{m \in M}$.
We also note that Lemma~\ref{lem:choiceofbm} does not depend on Conjecture~$\cFplus$.

\begin{lem}
\label{lem:choiceofbm}
There exist an adelic point
 $(x'_v)_{v \in \Omega} \in X^0(\A_k)^A$
 arbitrarily close to $(x_v)_{v \in \Omega}$ in~$X(\A_k)$
and
 a family $(b_m)_{m \in M} \in \prod_{m \in M} k(m)^*$ such that
if we let~$W$, $p$ and~$B$ be defined in terms of $(b_m)_{m \in M}$ as explained
above,
there exists
an adelic point $(z'_v)_{v \in \Omega}\in W^0(\A_k)^B$
satisfying $p(z'_v)=f(x'_v)$ for all $v \in \Omega$.
One can require, in addition, that the $x'_v$ for $v \in \Omega$ all belong to smooth fibres of~$f$.
\end{lem}

\begin{proof}
Let~$T$ be the torus over~$k$ defined by the exact sequence of tori
\begin{align}
\label{se:deft}
\xymatrix{
1 \ar[r] & T \ar[r] & \displaystyle\Gm \times \prod_{m\in M}R_{K_m/k}\Gm \ar[r] & \displaystyle\prod_{m\in M}
R_{k(m)/k}\Gm \ar[r] & 1\rlap,
}
\end{align}
where the second map is $(\nu,(z_m)_{m\in M})\mapsto (N_{K_m/k(m)}(z_m)/\nu)_{m\in M}$.
According to Shapiro's lemma and Hilbert's Theorem~90,
this exact sequence
induces
a surjection
\begin{align}
\label{eq:surjectionh1kt}
\prod_{m\in M}k(m)^* \twoheadrightarrow H^1(k,T)\rlap.
\end{align}
Given $(b_m)_{m \in M} \in \prod_{m\in M} k(m)^*$,
let $p_+:W_+\to \P^1_k$ denote the variety and the morphism associated
in~\textsection\ref{subsubsec:parametersdef}
with the triple
$(M, (K_m)_{m \in M}, (b_m)_{m\in M})$
and let $W_+^0=p_+^{-1}(U^0)$.
The morphism $W_+^0 \to U^0$ induced by~$p_+$
is a torsor under~$T$.
We denote by $\tau \in H^1_{\et}(U^0,T)$ its isomorphism class.

Starting with an arbitrary choice of $(b_m)_{m \in M}$
(for instance $b_m=1$ for all~$m$),
let~$A^0$ be the subgroup of $\Br(X^0)$ generated by~$A$ and by
the inverse images, by $f^*:\Br(U^0)\to \Br(X^0)$,
of the cup products
of $\tau$ with the elements of the finite group $H^1(k,\widehat T)$,
where $\widehat T$ denotes the character group of~$T$.
As $(x_v)_{v \in \Omega} \in X(\A_k)^{A^0\cap \Br(X)}$,
the version of Harari's formal lemma that can be found in \cite[Th\'eor\`eme~1.4]{ctbudapest}
provides $(x'_v)_{v \in \Omega} \in X^0(\A_k)^{A^0}$
arbitrarily close to $(x_v)_{v \in \Omega}$ in~$X(\A_k)$.
As~$A^0$ is finite, we may assume, after modifying the~$x'_v$ using the implicit function theorem,
that they all belong to smooth fibres of~$f$.

Let us apply open descent theory
to the projection $W_+ \times_{\P^1_k} X^0 \to X^0$, which is a torsor under~$T$.
According to \cite[Theorem~8.4, Proposition~8.12]{haskoopendescent},
the adelic point
 $(x'_v)_{v \in \Omega}$ can be lifted to an adelic point of some
twist of this torsor.
Now, twisting the class~$\tau$ by a class in $H^1(k,T)$
amounts, in view of the surjectivity of~\eqref{eq:surjectionh1kt}, to modifying
the choice of $(b_m)_{m\in M}$.  Thus, after modifying our choice
 of $(b_m)_{m\in M}$, we may assume that $(x_v')_{v \in \Omega}$ comes
from
$(W_+ \times_{\P^1_k} X^0)(\A_k)$. In particular, there exists
 $(z_{+,v}')_{v \in \Omega} \in W_+^0(\A_k)$ such that
$p_+(z_{+,v}')=f(x_v')$ for all $v\in \Omega$.

Let
$q:W_+ \to W$ be  defined
by $q(\lambda,\mu,(z_m)_{m\in M})=(\lambda,\mu,(N_{K_m/L_m}(z_m))_{m\in M})$
and set $z_v' = q(z_{+,v}')$,
so that $p(z_v')=f(x_v')$ for all $v\in \Omega$.
The projection formula implies that~$B$ is contained in the kernel
of the pull-back map $q^*:\Br(p^{-1}(U)) \to \Br(p_+^{-1}(U))$
and hence that $\beta(z_v')=0$ for any $\beta \in B$ and any $v \in \Omega$.
It follows that $(z_v')_{v \in \Omega} \in W^0(\A_k)^B$.
\end{proof}

We now apply Lemma~\ref{lem:choiceofbm}
and leave the resulting
$(x'_v)_{v \in \Omega}$, $(b_m)_{m \in M}$, $\pi_+$,
$W$, $p$, $B$ and $(z_v')_{v \in \Omega}$ fixed
for the remainder of the proof of Theorem~\ref{th:fibration}.

\begin{step}
\label{step:2}
Choice of the finite set of bad places~$S$.
\end{step}

Let~$S$ be a finite set of places of~$k$, containing the archimedean places
and the places at which
we want
to approximate $(x_v)_{v \in \Omega}$,
large enough that $f:X \to \P^1_k$ extends
to a flat morphism $f:\sX \to \P^1_{\sOint_S}$, where~$\sOint_S$ denotes the ring of
$S$\nobreakdash-integers of~$k$.

For $m \in \P^1_k$, we denote by~$\mtilde$ the Zariski closure of~$m$
in~$\P^1_{\sOint_S}$ and by~$\sX_m$ and~$\sY_m$ the Zariski closures
of~$X_m$ and~$Y_m$ in~$\sX$, all endowed with the reduced scheme structures.
When $m\in M$, we denote by $\mtilde' \to \mtilde$ the normalisation of~$\mtilde$
in the finite extension $L^0_m/k(m)$.
We set $G_m=\Gal(E_m/k(Y_m))$
and $H_m=\Gal(E_m/k(Y_m)K_m^0) \subseteq G_m$
and
 fix,
for each $m \in M$,  a dense open subset
 $\sY_m^0 \subseteq \sY_m$,
small enough that
 the normalisation $\sE_m \to \sY_m^0$
of~$\sY_m^0$ in the finite extension $E_m/k(Y_m)$ is a finite and \'etale
morphism
and that the restriction of~$f$ to $\sY_m^0$ factors through a
smooth morphism
$f_m:\sY_m^0 \to \mtilde'$.

Set $\Mtilde = \bigcup_{m \in M} \mtilde$,
$\sU=\P^1_{\sOint_S} \setminus \Mtilde$
and $\sU^0=\P^1_{\sOint_S} \setminus (\Mtilde \cup \widetilde\infty)$.
Let $p:\sW\to \P^1_{\sOint_S}$ denote the standard integral model of $p:W\to \P^1_k$
(see~\textsection\ref{subsubsec:localintpoint}).
Let $\sW^0 = p^{-1}(\sU^0)$.

After enlarging~$S$, we may assume
that~$\Mtilde \cup\widetilde\infty$
is \'etale over~$\sOint_S$,
that the morphisms $\mtilde' \to \mtilde$ for $m \in M$ are \'etale,
that the $b_m$ for $m \in M$ are units above the places of~$\Omega \setminus S$,
that the extensions $K_m/k$ are unramified above the places of
 $\Omega \setminus S$,
that $z'_v \in \sW^0(\sOint_v)$ for all $v \in \Omega\setminus S$,
that $A \subseteq \Br(f^{-1}(\sU))$,
that $\sum_{v \in S} \inv_v \alpha(x'_v)=0$ for all $\alpha \in A$,
that $\sum_{v \in S} \inv_v \beta(z'_v)=0$ for all $\beta \in B$,
that~$S$ contains the finite places which divide the order of~$A$,
and,
by the Lang--Weil--Nisnevich bounds~\cite{langweil} \cite{nisnevic}
and by a geometric version of Chebotarev's density theorem~\cite[Lemma~1.2]{ekedahl},
that the following hold:
\begin{itemize}
\item every closed fibre of~$f$ above~$\sU$ contains a smooth rational point;
\item for every $m \in M$, every closed fibre of
$f_m:\sY_m^0 \to \mtilde'$
contains a rational point;
\item for every $m \in M$ and every place~$u$ of~$L_m^0$
 which splits in~$K_m^0$ and does not lie above a place of~$S$,
any element of~$H_m$ can
be realised as the Frobenius automorphism
of the irreducible abelian \'etale cover $\sE_m\to\sY_m^0$
at some rational point
of the fibre of~$f_m$
 above the closed point of~$\mtilde'$ corresponding to~$u$.
\end{itemize}

\begin{step}
\label{step:3}
Application of Conjecture~$\cFplus$.
\end{step}

Let us fix a collection $(v_m)_{m \in M}$ of places of~$k$ that are pairwise distinct
and do not belong to~$S$, such that $v_m$ splits completely in~$K_m$ for all $m \in M$.
For each $m \in M$, let us fix a place $w_m$ of~$k(m)$ dividing~$v_m$,
an element
$t_{v_m} \in k_{v_m}$ such that $t_{v_m}-a_m$ is a uniformiser at~$w_m$,
and a point
$z''_{v_m} \in \sW(\sOint_{v_m})$ such that $p(z''_{v_m}) = t_{v_m}$,
where we view $t_{v_m}$ inside $k_{v_m} = \A^1(k_{v_m}) \subset \P^1(k_{v_m})$
(see Proposition~\ref{prop:localintegralpoint} for the existence of $z_{v_m}''$).
For every $v \in \Omega \setminus \{v_m\mkern1mu;\mkern2mu m \in M\}$,
we set $z''_v=z'_v$.
We have thus defined an adelic point $(z''_v)_{v \in \Omega} \in W(\A_k)$
such that $z''_v \in \sW(\sOint_v)$ for all $v \in \Omega\setminus S$.

We now apply Conjecture~$\cFplus$ to~$\pi_+$ and $(z''_v)_{v \in \Omega}$
and deduce that there exist $c \in U(k)$ and $(z_v''')_{v \in \Omega} \in W_c(\A_k)^B$ arbitrarily
close to $(z_v'')_{v\in\Omega}$ in $W(\A_k)$,  in particular such that
$z_v''' \in \sW(\sOint_v)$ for all $v\in \Omega\setminus S$
and such that $\sum_{v \in S} \inv_v \beta(z_v''')=0$ for all $\beta \in B$.

For each $v \in \Omega \setminus S$, let $(\lambda'''_v, \mu'''_v, (z_{m,u}''')_{m\in M,u|v})$
be the data corresponding to $z_v'''$ in the notation of Proposition~\ref{prop:localintegralpoint}
and let $w$ denote the closed point $\ctilde \cap \P^1_{\Fv{v}}$
of~$\P^1_{\sOint_S}$.

For $m \in M$, we set $\Omega_m=\{v \in \Omega \setminus S\mkern1mu;\mkern2mu w \in \mtilde\}$;
equivalently, this is the set of places $v\in \Omega\setminus S$
such that $\lambda'''_v-a_m\mu'''_v$ has positive valuation
at some place of~$k(m)$ dividing~$v$.
The sets $\Omega_m$ for $m \in M$ are finite and pairwise disjoint.
When $v \in \Omega_m$, we shall identify~$w$ with a place of~$k(m)$ of degree~$1$ over~$v$.
For each $m \in M$,
we may assume, by choosing~$z_{v_m}'''$ close enough to~$z_{v_m}''$,
that $v_m \in \Omega_m$ and that $\lambda_{v_m}'''-a_m\mu_{v_m}'''$ is a uniformiser at~$w$.

For later use, we note that according to Proposition~\ref{prop:localintegralpoint}, for each $m \in M$ and each $v \in \Omega_m$,
there exists a unique place~$u$ of~$L_m$ dividing~$v$ such that~$z_{m,u}'''$ is not a unit.
Moreover, this place divides both~$w$ and~$v$ and it has degree~$1$ over them.

\begin{step}
\label{step:4}
Construction of $(x_v''')_{v \in \Omega} \in X_c(\A_k)$
assuming given $(\sigma_m)_{m \in M} \in \prod_{m \in M} H_m$.
\end{step}

For $v \in S$,
we can ensure that~$c$
is arbitrarily close to $f(x_v')$,
by choosing~$z_v'''$ close enough to $z_v''$.
On the other hand, by the implicit function theorem,
the map $X^0(k_v) \to \P^1(k_v)$ induced by~$f$
admits a local $v$\nobreakdash-adic analytic section, around~$c$, passing through~$x_v'$.
Thus, by choosing~$z_v'''$ close enough to $z_v''$, we may assume that for
every $v \in S$, there exists $x_v''' \in X_c(k_v)$ arbitrarily close to~$x_v'$ in~$X(k_v)$.
We fix $x_v'''$ in this way for every $v \in S$.
By ensuring that~$x_v'''$ is close enough to $x_v'$ for $v \in S$, we may assume
that
\begin{align}
\label{eq:suminvs}
\sum_{v \in S} \inv_v \alpha(x_v''')=0
\end{align}
for all $\alpha \in A$.

For $v \in \Omega \setminus (S \cup \bigcup_{m \in M} \Omega_m)$,
noting that $w \in \sU$,
we fix a smooth rational point of the fibre of $f:\sX \to \P^1_{\sOint_S}$ above~$w$
and use Hensel's lemma to lift it
to a $k_v$\nobreakdash-point $x_v'''$ of~$X_c$.

For any $m \in M$ and any $v \in \Omega_m \setminus \{v_m\}$,
let us consider
 the trace on~$L_m^0$ of the unique place~$u$ of~$L_m$ dividing~$v$
such that $z_{m,u}'''$ is not a unit.
It defines a closed point of~$\mtilde'$.
We fix a rational point~$\xi_v$
of the fibre of
$f_m:\sY_m^0 \to \mtilde'$ above this
closed point and, viewing it as a smooth rational point of the fibre
of $f:\sX\to \P^1_{\sOint_S}$
above~$w$, we lift it, using Hensel's lemma,
to a $k_v$\nobreakdash-point $x_v'''$ of~$X_c$.

Let us assume that we are given a family $(\sigma_m)_{m \in M} \in \prod_{m \in M} H_m$
(to be specified at the end of Step~\ref{step:5}).
Then, for $m \in M$, we construct $x_v'''$ for $v=v_m$ in the exact same way
as we constructed $x_v'''$ for $v \in \Omega_m \setminus \{v_m\}$,
except that we require, in addition,
that
the Frobenius automorphism
of the irreducible abelian \'etale cover $\sE_m\to\sY_m^0$
at~$\xi_v$ be equal to~$\sigma_m$.

At this stage, we have now constructed an adelic point
$(x_v''')_{v \in \Omega}$ of~$X_c(\A_k)$,
depending on the choice
of $(\sigma_m)_{m \in M} \in \prod_{m \in M} H_m$.
To conclude the proof of the theorem, it only remains to show
that the family
$(\sigma_m)_{m \in M}$ can be prescribed in such a way
that $(x_v''')_{v \in \Omega}$
automatically belongs to $X_c(\A_k)^A$.

\begin{step}
\label{step:5}
Evaluation of the Brauer--Manin obstruction.
\end{step}

For $m \in M$ and $v \in \Omega_m$, let $n_v$ denote the local intersection multiplicity
of~$\ctilde$ and~$\mtilde$ at~$w$ inside~$\P^1_{\sOint_S}$ (i.e.\ the length of the
local ring of $\ctilde \cap \mtilde$ at~$w$)
and let $\Frob_{\xi_v} \in G_m$ denote the Frobenius
automorphism of the irreducible abelian \'etale cover $\sE_m\to\sY_m^0$
at~$\xi_v$.

\begin{lem}
\label{lem:isinhm}
For each $m \in M$, the element
\begin{align*}
\sum_{v \in \Omega_m \setminus \{v_m\}} n_v \Frob_{\xi_v}
\end{align*}
of~$G_m$ belongs to the subgroup~$H_m$.
\end{lem}

\begin{proof}
Let us fix $m \in M$.  We recall that $(z_v''')_{v \in \Omega} \in W_c(\A_k)^B$
and that $\sum_{v \in S} \inv_v \beta(z_v''')=0$ for all $\beta\in B$,
so that
$\sum_{v \in \Omega\setminus S} \inv_v \beta(z_v''')=0$
for all $\beta \in B$.
Let us apply this equality to the class $\beta=\Cores_{L_m/k}(z_m,\chi)$ for $\chi \in C_m$.
For $v \in \Omega \setminus (S \cup \Omega_m)$,
we have $\inv_v \beta(z_v''')=0$
since~$\chi$ is unramified above~$v$ and~$z_{m,u}'''$ is a unit at all places~$u$ of~$L_m$ above~$v$
(see Proposition~\ref{prop:localintegralpoint}).
For $v \in \Omega_m$, we have
$\inv_v \beta(z_v''')=\inv_u (z_{m,u}''',\chi)$ where~$u$ is the unique place of~$L_m$ dividing~$v$
such that~$z_{m,u}'''$ fails to be a unit;
moreover,
the normalised valuation of $z_{m,u}'''$ is equal to~$n_v$, so that
$\inv_u (z_{m,u}''',\chi)=n_v\chi(\Frob_u)$, where we view~$\chi$ as a homomorphism $\Gal(K_m/L_m)\to \Q/\Z$
and $\Frob_u \in \Gal(K_m/L_m)$ denotes the Frobenius automorphism at~$u$.
Finally, we recall that $\Frob_u=0$ if $v=v_m$, since~$v_m$ splits completely in~$K_m$.
All in all, we conclude that
\begin{align}
\sum_{v \in \Omega_m \setminus \{v_m\}} n_v \chi(\Frob_u)=0
\end{align}
for all $\chi \in C_m=\Hom(\Gal(K_m/L_m),\Q/\Z)$, hence
$\sum_{v \in \Omega_m \setminus \{v_m\}} n_v \Frob_u=0$ in $\Gal(K_m/L_m)$
by Pontrjagin duality.
Applying the natural map $\Gal(K_m/L_m) \to \Gal(K_m^0/L_m^0)$ to this  equality now
yields the statement of the lemma,
as the image of $\Frob_u$ by this map
coincides with the image of $\Frob_{\xi_v}$ by the quotient map $G_m \to G_m/H_m=\Gal(K_m^0/L_m^0)$.
\end{proof}

For $m \in M$ and $\alpha \in A$, let $\partial_{\alpha,m} \in H^1(G_m,\Q/\Z) \subset H^1(k(Y_m),\Q/\Z)$
 denote the residue
of~$\alpha$ at the generic point of~$Y_m$.
Viewing $\partial_{\alpha,m}$ as a homomorphism $G_m \to \Q/\Z$,
 we have
\begin{align}
\inv_v \alpha(x'''_v) = n_v\partial_{\alpha,m}(\Frob_{\xi_v})
\end{align}
for any $m \in M$ and any $v \in \Omega_m$, since $\alpha \in \Br(f^{-1}(\sU))$
(see \cite[Corollaire~2.4.3]{harariduke}).  Moreover,
we have
$\inv_v \alpha(x_v''')=0$ for all $v \in \Omega \setminus (S \cup \bigcup_{m \in M} \Omega_m)$,
as
$\alpha \in \Br(f^{-1}(\sU))$ and $x_v'''$ is an $\sOint_v$\nobreakdash-point of~$f^{-1}(\sU)$
while $\Br(\sOint_v)=0$.
In view of these remarks and of~\eqref{eq:suminvs},
we deduce that in order for $(x'''_v)_{v \in \Omega}$ to be orthogonal to~$A$ with respect to
the Brauer--Manin pairing,
it suffices that the equality
\begin{align}
\label{eq:sumfrobvanishes}
\sum_{v \in \Omega_m} n_v \Frob_{\xi_v}=0
\end{align}
hold in~$G_m$ for all $m \in M$.
Now, as $n_{v_m}=1$ and as $\Frob_{\xi_{v_m}}=\sigma_m$ for every $m \in M$,
we can force~\eqref{eq:sumfrobvanishes} to hold by
choosing
$\sigma_m= - \sum_{v \in \Omega_m \setminus \{v_m\}} n_v \Frob_{\xi_v}$
for all $m \in M$
in Step~\ref{step:4},
thanks to Lemma~\ref{lem:isinhm}.
This concludes the proof
of Theorem~\ref{th:fibration}.
\end{proof}

\subsection{Stability of Conjecture~\texorpdfstring{$\cFplus$}{F₊}}

We first use Theorem~\ref{th:fibration} to show,
in Corollary~\ref{cor:fplusimpliesfplus} below,
that when the~$b_m$ are allowed to vary,
Conjecture~$\cFplus$ is stable under the operation of replacing the fields~$L_m$
and~$K_m$ by subfields.

\begin{cor}
\label{cor:fplusimpliesfplus}
Let~$k$ be a number field
and $\pi_+^0 \in \sP_+$ be a parameter for Conjecture~$\cFplus$.
Write
$\pi_+^0=(M, (L_m^0)_{m \in M}, (b_m^0)_{m \in M}, (K_m^0)_{m\in M})$.
For every $m \in M$, suppose given a finite extension~$L_m$ of~$L^0_m$ and a finite abelian
extension~$K_m$ of~$L_m$ in which~$K^0_m$ can be embedded $L^0_m$\nobreakdash-linearly.
Assume that
for all choices of $(b_m)_{m \in M} \in \prod_{m \in M}k(m)^*$,
 Conjecture~$\cFplus$ holds for
 $\pi_+=(M, (L_m)_{m \in M}, (b_m)_{m \in M}, (K_m)_{m\in M}) \in \sP_+$.
Then Conjecture~$\cFplus$ holds
for $\pi^0_+$.
\end{cor}

\begin{proof}
Set $U = \P^1_k \setminus M$.
Let $p:W \to \P^1_k$ and $B \subset \Br(p^{-1}(U))$
denote the morphism and the subgroup associated
in~\textsection\ref{subsubsec:parametersdef}
and~\textsection\ref{sec:statementFplus}
with the parameter~$\pi_+^0$.
Taking Proposition~\ref{prop:fibresofp}
and
Proposition~\ref{prop:residueofcores}
into account,
we may apply
 Theorem~\ref{th:fibration} to
 $X=W$, $f=p$, $A=B$, $H=\P^1_k$,
$Y_m=W_m$ and
 $E_m=k(W_m) \otimes_{L_m^0} K_m^0$.
In view of
Proposition~\ref{prop:brauerw},
the desired conclusion follows.
\end{proof}

\subsection{Specialisation of the Brauer group}

To obtain concrete corollaries for the fibration method, we shall apply Theorem~\ref{th:fibration}
in~\textsection\ref{subsec:maincorollary} in conjunction with
the following specialisation result for the Brauer group, which goes back to the work of
Harari \cite{harariduke, hararifleches}.
The version we state here simultaneously generalises
 \cite[Proposition~4.1]{hwfibration} (in which~$f$ was assumed to be proper)
and \cite[Théorème~2.7]{cthararifamilles} (in which the generic fibre of~$f$ was assumed
to be a homogeneous space of a connected, semi-simple, simply connected linear algebraic group, with connected
and reductive geometric
stabilisers).

\begin{prop}
\label{prop:specialisation}
Let~$C$ be a smooth, irreducible curve over a number field~$k$.
Let~$X$ be a smooth, separated, irreducible variety over~$k$, endowed with a morphism
$f:X \to C$ whose geometric generic fibre~$X_\etabar$ is irreducible.
Assume that $H^1_\et(X_\etabar,\Q/\Z)=0$ and that $\Br(X_{\etabar})$ is finite.
Let $C^0 \subseteq C$ be a dense open subset, let $X^0=f^{-1}(C^0)$ and let
$B \subseteq \Br(X^0)$ be a subgroup.
If the natural map $B \to \Br(X_\eta)/f_\eta^*\mkern.5mu\Br(\eta)$ is surjective,
there exists
a Hilbert subset $H \subseteq C^0$ such that the natural map $B \to \Br(X_h)/f_h^*\mkern.5mu\Br(h)$ is surjective
for all $h \in H$.
\end{prop}

\begin{proof}
We may assume, after shrinking~$C$,  that $C^0=C$.
By the next lemma,
we may assume, after further shrinking~$C$, that
the étale sheaf $R^2f_*\Q/\Z(1)$ is a direct limit of locally constant
sheaves with finite stalks
and
that
 $H^1_\et(X_{\bar h},\Q/\Z(1))=0$ for any geometric point~$\bar h$ of~$C$.
Indeed,
our hypothesis that $H^1_\et(X_\etabar,\Q/\Z)=0$
implies that $H^1_\et(X_\etabar,\mmu_n)=0$ for all $n>0$,
and
the sheaf $R^1f_*\mmu_n$  must vanish if it is locally constant
while its stalk at~$\etabar$ vanishes.
From this point on,
the proof given in
 \cite[Proposition~4.1]{hwfibration}
for the special case in which~$f$ is proper works
verbatim (ignoring its first sentence).
\end{proof}

\begin{lem}
For any $q\geq 0$, there exists a dense open subset $U\subseteq C$
such that for every $n>0$,
the restriction of the étale sheaf
$R^qf_*\mmu_n$ to~$U$ is
 locally constant and its stalk
at any geometric point~$\bar h$ of~$U$
is naturally isomorphic to $H^q_\et(X_{\bar h},\mmu_n)$.
\end{lem}

\begin{proof}
After shrinking~$C$, we may assume,
thanks to Nagata~\cite{nagataimbedding,delignenagata,conradnagata}
and to Hironaka~\cite[\textsection3, Main theorem~I, \textsection5, Corollary~3]{hironaka},
that there exist an open immersion $j:X \to X'$ and a smooth and proper morphism $g:X' \to C$
such that $f = g \circ j$ and that $X' \setminus X$ is a divisor with simple normal crossings
on~$X'$
relatively to~$C$ (in the sense of \cite[Exp.~XIII, \textsection2.1]{sga1}).
The morphism~$f$
is then cohomologically proper with respect to~$\mmu_n$ for every $n>0$ (see
\cite[Appendice, \textsection1.3.1, \textsection1.3.3]{delignefinitude}),
which already ensures the second part of the assertion.
On the other hand, it is a locally acyclic morphism, since it is smooth (see \cite[Theorem~4.15]{milneet}).
Hence the étale sheaf
$R^qf_*\mmu_n$ is
 locally constant for every $n>0$
(see
\cite[Appendice, \textsection2.4]{delignefinitude} and
\cite[Proposition~2.11]{artinfaisceauxconstructibles};
or see the proof of \cite[Corollary~4.2]{milneet}).
\end{proof}

\subsection{Main corollary}
\label{subsec:maincorollary}

It is through Corollary~\ref{cor:concretecorollary} below,
which depends on Proposition~\ref{prop:specialisation},
that we shall apply Theorem~\ref{th:fibration}
in~\textsection\ref{sec:newapplications}.

\begin{cor}
\label{cor:concretecorollary}
Let~$X$ be a smooth, separated, irreducible variety over a number field~$k$
and $f:X \to \P^1_k$ be a morphism with irreducible geometric generic fibre~$X_\etabar$.
Assume
that
\begin{enumerate}[label={\upshape(\roman*)}]
\setlength\itemsep{.5ex}
\item
the group $H^1_\et(X_\etabar,\Q/\Z)$ vanishes and the group $\Br(X_{\etabar})$ is finite,
\item
 every fibre of~$f$ contains an irreducible component of multiplicity~$1$,
\item
the geometric fibre $X_{\bar\infty}=f^{-1}(\infty) \otimes_k \bar k$ is smooth and irreducible, and the group $H^1_\et(X_{\bar\infty},\Q/\Z)$ vanishes.
\end{enumerate}
For each $m \in \A^1_k$,
choose an irreducible component of multiplicity~$1$ of $f^{-1}(m)$
and let~$L_m$ denote the algebraic closure
of~$k(m)$ in its function field.
Finally, assume that
for all finite subsets $M \subset \A^1_k$
and for all choices of $(b_m)_{m \in M}$
and of $(K_m)_{m\in M}$,
Conjecture~$\cFplus$ holds for
the parameter $\pi_+=(M, (L_m)_{m \in M}, (b_m)_{m \in M}, (K_m)_{m\in M})$.
Then,
for any Hilbert subset $H \subset \P^1_k$,
 the subset
\begin{equation}
\bigcup_{c \in U(k)\cap H} X_c(\A_k)^{\Br(X_c)}
\end{equation}
is dense in $X(\A_k)^{\Br(X)}$.
In particular, if $X_c(k)$ is dense in $X_c(\A_k)^{\Br(X_c)}$
for all rational points~$c$ of a Hilbert subset of~$\P^1_k$,
then $X(k)$ is dense in $X(\A_k)^{\Br(X)}$.
\end{cor}

\begin{proof}
By the Kummer exact sequence, the vanishing of $H^1_\et(X_{\etabar},\Q/\Z)$,
which amounts to the vanishing of
 $H^1_\et(X_{\etabar},\Z/n\Z)$ for all $n>0$,
implies that the group $\Pic(X_{\etabar})$ is torsion-free and that
the group of invertible functions on~$X_\etabar$
is divisible
(see \cite[Proposition~4.11]{milneet}).
As the group of invertible functions on~$X_\etabar$ modulo
the subgroup
of constant invertible functions
is finitely generated (see \cite[Lemma on p.~28]{rosenlicht}),
we deduce
that every invertible function on~$X_\etabar$
is constant.
From these  facts and from the finiteness
of $\Br(X_{\etabar})$,
it follows, by the Hochschild--Serre spectral sequence,
that the quotient $\Br(X_\eta)/f_\eta^*\Br(\eta)$ is finite,
where $f_\eta:X_\eta\to \eta$ denotes the generic fibre of~$f$.

For $\beta \in \Br(X_\eta)$,
our hypothesis~(iii) implies
 that the residue of~$\beta$ at the generic point of $X_\infty=f^{-1}(\infty)$ can be written
as $f_\eta^* \chi$ for some $\chi \in H^1(k,\Q/\Z)$.
Let $\delta=(t,\chi) \in \Br(\eta)$.  By \cite[Proposition~1.1.1]{ctsd94},
the residue of $f_\eta^*\delta$ at the generic point of~$X_\infty$ is equal to $-f_\eta^*\chi$.
We have thus shown that
for any $\beta \in \Br(X_\eta)$, there exists $\delta \in \Br(\eta)$
such that $\beta + f_\eta^*\delta$ is unramified 
along $f^{-1}(\infty)$.
As
 the quotient $\Br(X_\eta)/f_\eta^*\Br(\eta)$ is finite,
we can therefore
choose a
 finite subgroup $A \subset \Br(X_\eta)$
 that surjects
onto $\Br(X_\eta)/f_\eta^*\Br(\eta)$
and whose elements are unramified along~$f^{-1}(\infty)$.

Let $U \subset \P^1_k$ be a dense open subset containing~$\infty$
such that $A \subset \Br(f^{-1}(U))$.
According to Proposition~\ref{prop:specialisation},
there exists a Hilbert subset
 $H_0 \subseteq U$
such that the natural map $A \to \Br(X_c)/f_c^*\Br(k)$ is surjective
for all $c \in H_0 \cap \P^1(k)$.
To conclude the proof of Corollary~\ref{cor:concretecorollary},
we now apply Theorem~\ref{th:fibration} to the Hilbert subset $H \cap H_0$.
\end{proof}

\begin{rmks}
\label{rk:onconcretecorollary}
(i)
When assumption~(i) of Corollary~\ref{cor:concretecorollary} is satisfied,
assumption~(iii)
can always be made to hold by a change of coordinates on~$\P^1_k$.

(ii)
If~$f$ is proper and~$X_\etabar$ is rationally connected,
 assumptions~(i) and~(ii) of
Corollary~\ref{cor:concretecorollary}
are satisfied
(see \cite[Corollary~4.18(b)]{debarrehigherdim}, \cite[Lemma~1.3~(i)]{ctskogoodreduction}, \cite[Lemma~8.6]{hwfibration}, \cite[Theorem~1.1]{ghs}).
\end{rmks}

\section{Comparing Conjectures \texorpdfstring{$\cF$}{F} and \texorpdfstring{$\cFplus$}{F₊}}
\label{sec:comparing}

This section is devoted to a detailed study of the relationship between Conjecture~$\cF$ and
Conjecture~$\cFplus$.
In order to facilitate their comparison,
we introduce an intermediate statement, Conjecture~$\cFconst$.
We prove that when the parameters are allowed to vary,
the three conjecture are equivalent,
and that in certain special circumstances, Conjecture~$\cF$ coincides with Conjecture~$\cFconst$, while
in certain other special circumstances, Conjecture~$\cFconst$ coincides with Conjecture~$\cFplus$.
One advantage of considering Conjecture~$\cFconst$ is that
under some abelianness assumptions, it is
implied by Schinzel's hypothesis~$(\mathrm{HH}_1)$,
as we will see in~\S\ref{sec:knowncases}. This is what will eventually allow us to deduce
Theorem~\ref{thm:schinzel-intro}.

\subsection{Introduction}

A number field~$k$ is fixed until 
the end of~\textsection\ref{sec:from-const-to-plus}.
Let us first formulate Conjecture~$\cFconst$.
Given
$\pi_+=(M, (L_m)_{m \in M}, (b_m)_{m \in M}, (K_m)_{m\in M}) \in \sP_+$, we
define
\begin{align*}
C_{m,{\const}} = C_m \cap \Imm\mkern-2mu\Big(H^1(k(m),\Q/\Z) \to H^1(L_m,\Q/\Z)\Big)
\end{align*}
for all $m \in M$ and let $B_{\const} \subseteq B$ be
the subgroup
generated by the
classes
$\Cores_{L_m/k}(z_m,\chi)$
for $m \in M$ and $\chi \in C_{m,\const}$.

\begin{conjFconst}
Let $\pi_+\in\sP_+$.
The subset $\displaystyle\bigcup_{c \in U(k)} W_c(\A_k)^{B_\const}$ is dense in $W(\A_k)$.
\end{conjFconst}

We note right away, in the next proposition,
that when the parameters are allowed to vary,
the conjectures we have introduced are all equivalent.
Thus, the existing evidence for
Conjecture~$\cF$ (see \cite[\textsection9.2]{hwfibration},
\cite{browningschindler}) lends support to
Conjecture~$\cFplus$ as well.

\begin{prop}
\label{prop:allequivalent}
The following statements are equivalent:
\begin{enumerate}
\item Conjecture~$\cF$ holds for all $\pi \in \sP$;
\item Conjecture~$\cFconst$ holds for all $\pi_+ \in \sP_+$;
\item Conjecture~$\cFplus$ holds for all $\pi_+ \in \sP_+$.
\end{enumerate}
\end{prop}

\begin{proof}
The implications (3)$\mkern2mu\Rightarrow\mkern2mu$(2)$\mkern2mu\Rightarrow\mkern2mu$(1) being obvious, we need only
prove that (1)$\mkern2mu\Rightarrow\mkern2mu$(3).
Let us fix $\pi_+=(M, (L_m)_{m \in M}, (b_m)_{m \in M}, (K_m)_{m\in M}) \in \sP_+$.
Among the conditions
\begin{enumerate}[label={\upshape(\roman*)}]
\item
Conjecture~$\cF$ for $(M, (K_m)_{m \in M}, (b'_m)_{m \in M})$
for all choices of $(b'_m)_{m\in M}$,
\item
Conjecture~$\cFplus$ for
 $(M, (K_m)_{m \in M}, (b'_m)_{m \in M},(K_m)_{m \in M})$
for all choices of $(b'_m)_{m\in M}$,
\item
Conjecture~$\cFplus$ for~$\pi_+$,
\end{enumerate}
we have (i)$\mkern2mu\Leftrightarrow\mkern2mu$(ii) by definition, and (ii)$\mkern2mu\Rightarrow\mkern2mu$(iii)
by Corollary~\ref{cor:fplusimpliesfplus}; hence (1)$\mkern2mu\Rightarrow\mkern2mu$(3).
\end{proof}

From here onwards, we shall systematically consider the three conjectures with fixed parameters.
Whenever a quadruple $\pi_+ \in \sP_+$ is given,
it will be understood that the notation
$M$, $(L_m)_{m \in M}$, $(b_m)_{m \in M}$ and $(K_m)_{m\in M}$ refers
to its components, unless specified otherwise,
and
that
$\pi$ denotes the underlying triple $(M, (L_m)_{m \in M}, (b_m)_{m \in M}) \in \sP$.

It is immediate that Conjecture~$\cFplus$ for~$\pi_+$ implies Conjecture~$\cFconst$
for~$\pi_+$ and that the latter in turn implies Conjecture~$\cF$ for~$\pi$.
We shall now attempt to reverse these trivial implications.

\subsection{From Conjecture \texorpdfstring{$\cF$}{F} to Conjecture \texorpdfstring{$\cFconst$}{Fconst}}

\begin{prop}
\label{prop:fromFtoFconst}
Let $\pi_+ \in \sP_+$.
The inclusion
\begin{align}
B_{\const} \subset p^*\Br(U)
\end{align}
of subgroups of $\Br(p^{-1}(U))$
holds if and only if
the two groups
\begin{align}
\label{eq:grouplm}
\bigoplus_{m \in M} \Ker\big(H^1(k(m),\Q/\Z) \to H^1(L_m,\Q/\Z)\big)
\end{align}
and
\begin{align}
\label{eq:groupkm}
\bigoplus_{m \in M} \Ker\big(H^1(k(m),\Q/\Z) \to H^1(K_m,\Q/\Z)\big)
\end{align}
have the same image by
the ``sum of corestrictions'' map
\begin{align}
\label{eq:sumcorestrictions}
\bigoplus_{m \in M} H^1(k(m),\Q/\Z) \to H^1(k,\Q/\Z)\rlap.
\end{align}
When these conditions are satisfied,
Conjecture~$\cF$ for~$\pi$ implies
Conjecture~$\cFconst$ for~$\pi_+$.
\end{prop}

\begin{proof}
If $B_{\const} \subset p^*\Br(U)$,
then
$W_c(\A_k) = W_c(\A_k)^{B_{\const}}$ for all $c \in U(k)$.
Hence in this case
Conjecture~$\cF$ for~$\pi$ implies
Conjecture~$\cFconst$ for~$\pi_+$.
It only remains to check the first assertion of Proposition~\ref{prop:fromFtoFconst}.

Assume first that
$B_{\const} \subset p^*\Br(U)$
and
let $(\chi_m)_{m \in M}$ belong to~\eqref{eq:groupkm}.
Let $\gamma \in \Br(U)$
be such that $\sum_{m \in M} \Cores_{L_m/k}(z_m,\chi_m)=p^*\gamma$.
By the Faddeev exact sequence,
the family $(\partial_m\gamma)_{m \in M}$
given by the residues of~$\gamma$
belongs to the kernel of~\eqref{eq:sumcorestrictions}
(see \cite[\textsection1.2]{ctsd94}).
On the other hand,
by
Proposition~\ref{prop:fibresofp},
Proposition~\ref{prop:residueofcores}
and
\cite[Proposition~1.1.1]{ctsd94},
computing residues at the generic point of~$W_m$ for $m \in M$
shows that $(\partial_m\gamma-\chi_m)_{m \in M}$
belongs to~\eqref{eq:grouplm}.
We have thus proved that~\eqref{eq:grouplm} and~\eqref{eq:groupkm}
have the same image by~\eqref{eq:sumcorestrictions}.

Conversely, assuming that this last condition holds,
let us fix $\beta \in B_{\const}$ and write it as
$\beta=\sum_{m \in M} \Cores_{L_m/k}(z_m,\chi_m)$
with $(\chi_m)_{m \in M}$ in~\eqref{eq:groupkm}.
Our assumption provides $(\delta_m)_{m\in M}$
in the kernel of~\eqref{eq:sumcorestrictions}
such that $(\delta_m-\chi_m)_{m \in M}$ belongs to~\eqref{eq:grouplm}.
By the Faddeev exact sequence again, there exists $\gamma \in \Br(U)$
such that $\partial_m\gamma=\delta_m$ for all $m \in M$.
By
Proposition~\ref{prop:fibresofp}
and
Proposition~\ref{prop:residueofcores},
the class
 $\beta-p^*\gamma \in \Br(p^{-1}(U))$
belongs to the subgroup~$\Br(W)$.
By Proposition~\ref{prop:brauerw}, we conclude that $\beta \in p^*\Br(U)$.
\end{proof}

\begin{cor}
\label{cor:comparisonFFconst}
Suppose given a finite closed subset $M \subset \A^1_k$,
a point $m_0 \in M$
and, for each $m \in M \setminus \{m_0\}$, a finite
extension $L_m/k(m)$ and a finite abelian extension $K_m/L_m$.
Assume that $k(m_0)=k$ or that~$k$ is totally imaginary.
Then there exists a finite abelian extension $L_0/k(m_0)$ with the following property:
for any $(b_m)_{m\in M}\in \prod_{m\in M}k(m)^*$
and any finite extension $L_{m_0}/L_0$,
if we set $K_{m_0}=L_{m_0}$,
Conjecture~$\cF$ for
$\pi=(M, (L_m)_{m \in M}, (b_m)_{m \in M})$
 implies
Conjecture~$\cFconst$ for
 $\pi_+=(M, (L_m)_{m \in M}, (b_m)_{m \in M}, (K_m)_{m\in M})$.
\end{cor}

\begin{proof}
The corestriction map $H^1(k(m_0),\Q/\Z) \to H^1(k,\Q/\Z)$
is surjective
since $k(m_0)=k$ or~$k$ is totally imaginary
(see~\cite[Note, p.~327]{gras},
\cite[Remark~9.18~(ii)]{hwfibration}).
Hence we can choose a finite subgroup $C \subset H^1(k(m_0),\Q/\Z)$
whose image by this map coincides with the image,
by the map~\eqref{eq:sumcorestrictions}, of the finite group
\begin{align}
\bigoplus_{m \in M \setminus \{m_0\}} \Ker\big(H^1(k(m),\Q/\Z) \to H^1(K_m,\Q/\Z)\big)\rlap{.}
\end{align}
Let $L_0/k(m_0)$
be a finite abelian extension
 such that the image
of~$C$ in $H^1(L_0,\Q/\Z)$ vanishes.
The condition of Proposition~\ref{prop:fromFtoFconst} is now satisfied
for any~$L_{m_0}$ and~$K_{m_0}$ as in the statement of the corollary.
\end{proof}

\begin{rmks}
\label{rks:surj99}
(i)
More generally, one can verify that in the situation
of \cite[Theorem~9.17]{hwfibration},
if we set
$K_m=L_m$ for  $m \in M'$
and $L_m=k(m)$ for $m \in M''$
and if we are given
finite abelian extensions $K_m/k(m)$ for $m \in M''$,
the assumption that the map~(9.9) of \emph{loc.\ cit.}\ is onto
for all $m \in M''$
(an assumption that is satisfied when $M'$ contains a rational point
or~$k$ is totally imaginary, by \cite[Remark~9.18~(ii)]{hwfibration})
implies that the condition of Proposition~\ref{prop:fromFtoFconst}
holds, with $M=M' \cup M''$,
as soon as the fields~$L_m$ for $m \in M'$ are large enough (in the sense
that they contain certain subfields, as in Corollary~\ref{cor:comparisonFFconst}).

(ii)
Conjecture~$\cFconst$ for~$\pi_+$ is the same as Conjecture~$\cFplus$ for~$\pi_+$
when for all $m \in M$, at least one of the extensions $K_m/L_m$ and $L_m/k(m)$ is trivial.

(iii)
In view of
Remark~\ref{rk:wwpiso} and of Remarks~\ref{rks:surj99}~(i)--(ii),
the statement of \cite[Theorem~9.17]{hwfibration}
can be recovered by combining
Theorem~\ref{th:fibration} with
Proposition~\ref{prop:fromFtoFconst}.
\end{rmks}

\subsection{From Conjecture \texorpdfstring{$\cFconst$}{Fconst} to Conjecture \texorpdfstring{$\cFplus$}{F₊}}\label{sec:from-const-to-plus}

The next theorem, which gives an equivalent formulation for Conjecture~$\cFplus$,
will allow us to pass from Conjecture~$\cFconst$ to Conjecture~$\cFplus$
in more general situations than that of Remark~\ref{rks:surj99}~(ii).

Let $\pi_+ \in \sP_+$.
For $m\in M$,
we denote by $C_{m,\nr} \subseteq C_m$  the subgroup
consisting of those
$\chi \in C_m$
such that
the class
\begin{align}
\Cores_{L_m/k(m)}(z,\chi) \in \Br(R^1_{L_m/k(m)}\Gm)
\end{align}
belongs to the subgroup
 $\Br_\nr(R^1_{L_m/k(m)}\Gm)$,
where~$z$ stands for
the tautological invertible
function on $R^1_{L_m/k(m)}\Gm \otimes_{k(m)} L_m$.
We denote by $B_{\nr} \subseteq B$
 the subgroup
 generated by the
classes
$\Cores_{L_m/k}(z_m,\chi) \in \Br(p^{-1}(U))$
for $m \in M$ and
 $\chi \in C_{m,\nr}$.

\begin{thm}
\label{th:conjfnrconjfplus}
Let $\pi_+\in \sP_+$.  The following conditions are equivalent:
\begin{enumerate}
\item
 the subset
$\bigcup_{c \in U(k)} W_c(\A_k)^{B_\nr}$ is dense in $W(\A_k)$;
\item
Conjecture~$\cFplus$ holds for~$\pi_+$.
\end{enumerate}
\end{thm}

\begin{proof}
The implication (2)$\mkern2mu\Rightarrow\mkern2mu$(1) being trivial,
we assume~(1) and prove~(2).
To this end, we fix an adelic point
 $(z_v)_{v \in \Omega} \in W(\A_k)$ and a finite subset $S\subset \Omega$
large enough that it satisfies the assumptions of Proposition~\ref{prop:localintegralpoint}.
We take up the notation of~\textsection\ref{subsubsec:localintpoint}
for
the standard integral model $p:\sW\to \P^1_{\sOint_S}$  of $p:W\to \P^1_k$,
 for the Zariski closure~$\mtilde$ of~$m \in \P^1_k$ in~$\P^1_{\sOint_S}$,
and for
 $\Mtilde = \bigcup_{m \in M} \mtilde$.
We choose~$S$
large enough
that
 $z_v \in \sW(\sOint_v)$ for all $v \in \Omega \setminus S$.
Our goal is then
to produce $c \in U(k)$ and $(z_v''')_{v \in \Omega} \in W_c(\A_k)^B$ with~$z_v'''$ arbitrarily close
to~$z_v$ for $v \in S$ and $z_v''' \in \sW(\sOint_v)$
for $v \in \Omega\setminus S$.

For each $m\in M$,  fix a
Galois closure $J_m/k$ of $K_m/k$.
We denote by $\Theta$ the set of triples $(m,\sigma,\tau)$,
where $m \in M$ and $(\sigma,\tau) \in \Gal(J_m/L_m) \times \Gal(J_m/k(m))$
are such that $\tau\sigma\tau^{-1} \in \Gal(J_m/L_m)$.
For $m \in M$,
we denote by $\Theta_m \subseteq \Theta$ the subset of triples whose first component is~$m$
and by $\phi_m:\Theta_m \to \Gal(K_m/L_m)$
 the map
that sends $(m,\sigma,\tau)$ to the restriction, to~$K_m$,
of the automorphism
$\sigma\tau\sigma^{-1}\tau^{-1}$ of~$J_m$.

For $\theta=(m,\sigma,\tau) \in \Theta$,
Chebotarev's density theorem ensures the existence of
infinitely many places of~$J_m$
that do not lie over a place of~$S$
and
whose Frobenius automorphism in $\Gal(J_m/k)$ is
equal to~$\sigma$.
For each $\theta$, we choose such a place~$r_\theta$ of~$J_m$ and let~$v_\theta$
denote its trace on~$k$.
We choose the~$r_\theta$ for $\theta \in \Theta$ in such a way that the~$v_\theta$
are pairwise distinct, and let $T \subset \Omega \setminus S$ denote
the subset consisting of the places $v_\theta$ for $\theta \in \Theta$.

Given $\theta=(m,\sigma,\tau) \in \Theta$,
we denote
by~$w_\theta$ the trace of~$r_\theta$ on~$k(m)$
and
by~$u_{\theta,1}$ and~$u_{\theta,2}$
the traces of~$r_\theta$ and of~$\tau(r_\theta)$ on~$L_m$,
respectively.
We note that~$u_{\theta,1}$ and~$u_{\theta,2}$ both divide~$w_\theta$
and that the Frobenius automorphisms of~$r_\theta$ and of~$\tau(r_\theta)$ in $\Gal(J_m/k)$
 are equal to~$\sigma$ and to~$\tau\sigma\tau^{-1}$, respectively.
In particular, letting $\Frob_{u_{\theta,i}} \in \Gal(K_m/L_m)$ denote the Frobenius automorphism
of the abelian extension $K_m/L_m$ at~$u_{\theta,i}$,
we have
\begin{align}
\label{eq:frobu1u2phim}
\Frob_{u_{\theta,1}} - \Frob_{u_{\theta,2}} = \phi_m(\theta)
\end{align}
in the abelian group $\Gal(K_m/L_m)$.

\begin{lem}
\label{lem:existszvi}
Let $\theta=(m,\sigma,\tau)\in\Theta$.
For ease of notation,
set $v=v_\theta$ and $w=w_\theta$.
\begin{enumerate}
\item There exists
$t_{v} \in k_{v}$ such that $t_{v}-a_m \in k(m)_w$ is a uniformiser.
\item
For any
$t_{v} \in k_{v}$  such that
 $t_{v}-a_m \in k(m)_w$ is a uniformiser
and any $i \in \{1,2\}$,
there exists
$z_{v,i} \in \sW(\sOint_v)$
such that $p(z_{v,i})=t_v$
(viewing $t_v$ in $k_v=\A^1(k_v) \subset \P^1(k_v)$)
and such that
for all $m' \in M$ and all $\chi \in C_{m'}$,
the following equality holds:
\begin{align}
\label{eq:invcoreschi}
\inv_v \big(\mkern-2mu\Cores_{L_{m'}/k}(z_{m'},\chi)\big)(z_{v,i})=
\begin{cases}
\chi\big(\Frob_{u_{\theta,i}}\big) & \text{ if } m'=m\rlap{,}\\
0  & \text{ if } m'\neq m\rlap{.}
\end{cases}
\end{align}
\end{enumerate}
\end{lem}

In the right-hand side of~\eqref{eq:invcoreschi}, we view~$\chi$ as a
homomorphism $\Gal(K_m/L_m)\to\Q/\Z$.

\begin{proof}
As the Frobenius automorphisms of~$r_\theta$ and of~$\tau(r_\theta)$ in $\Gal(J_m/k)$
 lie in $\Gal(J_m/L_m)$,
the three places~$u_{\theta,1}$, $u_{\theta,2}$ and $w$ have degree~$1$ over~$v$.
Assertion~(1) follows.

Let now $t_v$ be as in~(2),
in particular $t_v \in \sOint_v$.
As $\Mtilde \otimes_{\sOint_S} \sOint_v$ is regular, we may interpret its closed points as pairs $(m',w')$
where $m' \in M$ and~$w'$ is a place of~$k(m')$ dividing~$v$.
For any $(m',w') \in \Mtilde\otimes_{\sOint_S}\sOint_v$,
the element $t_v-a_{m'}$ has positive $w'$\nobreakdash-adic valuation
if and only if
the Zariski closure of $t_v \in \P^1(k_v)$ in~$\P^1_{\sOint_v}$ meets $\Mtilde \otimes_{\sOint_S} \sOint_v$
at $(m',w')$.
As $\Mtilde \cup \widetilde\infty$ is \'etale over~$\sOint_S$
and as $t_v-a_m$ is a uniformiser at~$w$,
it follows that $t_v-a_{m'}$ is a unit at~$w'$ for all $(m',w') \in \Mtilde \otimes_{\sOint_S}\sOint_v$
distinct from $(m,w)$.

Fix $i \in \{1,2\}$
and set $\lambda_v=t_v$ and $\mu_v=1$.
By the conclusions of the last two paragraphs,
there exists
  $(z_{m',u})_{m' \in M,u|v} \in \prod_{m' \in M,u|v} \sOint_{(L_{m'})_u}$,
where~$u$ runs over the places of~$L_{m'}$ dividing~$v$,
that satisfies the conditions of Proposition~\ref{prop:localintegralpoint}
as well as the following property: $z_{m,u_i}$ is a uniformiser while
$z_{m',u}$ is a unit for every $m' \in M$ and every place~$u$ of~$L_{m'}$
dividing~$v$ such that $(m',u) \neq (m,u_i)$.
By Proposition~\ref{prop:localintegralpoint},
the family
$(\lambda_v,\mu_v,(z_{m',u})_{m'\in M,u|v})$
then gives rise to a  point $z_{v,i} \in \sW(\sOint_v)$
such that $p(z_{v,i})=t_v$.

For any $m' \in M$ and $\chi \in C_{m'}$,
the left-hand side of~\eqref{eq:invcoreschi} is equal to
$\sum_{u|v}\inv_u(z_{m',u},\chi)$,
where the sum runs over the places~$u$ of~$L_{m'}$
dividing~$v$
(see \cite[Theorem~8.9]{harbook}).
As~$\chi$ is unramified above~$v$ and as $z_{m',u}$ is a unit if $m'\neq m$ or if $u \neq u_i$,
and is a uniformiser otherwise, the validity of~\eqref{eq:invcoreschi} follows
(see \cite[Corollary~9.6]{harbook}).
\end{proof}

For $v \in \Omega \setminus T$,
we set $z_v' = z_v$.
For each $v \in T$,
let us apply the two parts of Lemma~\ref{lem:existszvi}, with $i=1$,
and denote by $z'_{v} \in \sW(\sOint_{v})$ the resulting local integral point.

Applying
 our assumption~(1)
to the adelic point $(z'_v)_{v \in \Omega} \in W(\A_k)$ thus constructed
 yields $c \in U(k)$ and $(z''_v)_{v \in\Omega} \in W_c(\A_k)^{B_{\nr}}$,
with $(z''_v)_{v \in \Omega}$
 arbitrarily close
to~$(z'_v)_{v \in \Omega}$.  We may, in particular, assume
that $c\neq \infty$,
that~$z''_v$ is arbitrarily close to~$z_v$ for $v \in S$,
that $z''_v \in \sW(\sOint_v)$ for all $v \in \Omega\setminus S$
and that $\beta(z''_v)=\beta(z'_v)$ for all $v \in T$
and all $\beta \in B$
(so that $\inv_v \beta(z''_v)$ is described by~\eqref{eq:invcoreschi}).

\begin{lem}
There exists a subset $\Theta^\sharp \subseteq \Theta$
such that
\begin{align}
\label{eq:suminvntchi}
\sum_{v \in \Omega} \inv_v \big(\mkern-2mu\Cores_{L_m/k}(z_m,\chi)\big)(z''_v)=
\sum_{\theta \in \Theta_m^\sharp}  \chi\big(\phi_m(\theta)\big)
\end{align}
for all $m \in M$
and all $\chi \in C_m$, where $\Theta_m^\sharp=\Theta^\sharp\cap\Theta_m$.
\end{lem}

\begin{proof}
For $m\in M$, let $D_m \subseteq \Gal(K_m/L_m)$ denote the subgroup generated
by the image of~$\phi_m$.
Under the perfect duality of finite abelian groups
\begin{align}
C_m \times \Gal(K_m/L_m) \to \Q/\Z\rlap{,}
\end{align}
the subgroups $C_{m,\nr}$ and~$D_m$ are exact orthogonal complements,
according to Corollary~\ref{cor:chinr}.
On the other hand,
the homomorphism
$C_m \to \Q/\Z$ that sends $\chi \in C_m$ to the left-hand side of~\eqref{eq:suminvntchi}
vanishes on~$C_{m,\nr}$,
since $(z''_v)_{v \in \Omega} \in W_c(\A_k)^{B_{\nr}}$.
It follows that there exists
 a family of integers $(n_\theta)_{\theta \in \Theta}$
such that
for all $m \in M$
and all $\chi \in C_m$,
the left-hand side of~\eqref{eq:suminvntchi} is equal
to
$\sum_{\theta \in \Theta_m} n_{\theta} \chi\big(\phi_m(\theta)\big)$.
Now one observes that for any $m \in M$ and any integer~$n$,
the image of~$\phi_m$ is
 stable
under multiplication by~$n$ in the abelian
group $\Gal(K_m/L_m)$.  Hence we can choose the $n_{\theta}$ in $\{0,1\}$,
and the lemma is proved.
\end{proof}

Let $T^\sharp \subseteq T$
denote
the subset consisting of the places $v_\theta$ for $\theta \in \Theta^\sharp$.
For each $v \in T^\sharp$,
apply
Lemma~\ref{lem:existszvi}~(2) with $i=2$ and with $t_v=c \in \A^1(k)$,
and let $z_v''' \in \sW(\sOint_v) \cap W_c(k_v)$ denote
the resulting local integral point.
Set $z_v'''=z_v''$ for
all $v \in \Omega \setminus T^\sharp$.
We have thus constructed an adelic point $(z_v''')_{v \in \Omega} \in W_c(\A_k)$
with $z_v''' \in \sW(\sOint_v)$ for $v \in \Omega \setminus S$
and
with $z_v'''$ arbitrarily close to~$z_v$ for $v \in S$.

We finally check that $(z_v''')_{v \in \Omega} \in W_c(\A_k)^B$.
It is enough to see that
for any $m' \in M$ and any $\chi \in C_{m'}$,
if we set
$\beta=\Cores_{L_{m'}/k}(z_{m'},\chi)$,
then $\sum_{v \in \Omega} \inv_v \beta(z_v''')$ vanishes.
Now, as $z_v'''=z_v''$ for all
 $v \in \Omega \setminus T^\sharp$,
 this sum can be rewritten as
\begin{align}
\label{eq:sumrewritten}
\sum_{v \in \Omega} \inv_v \beta(z''_v)
+
\sum_{\theta \in \Theta^\sharp}\Big( \inv_{v_\theta} \beta(z'''_{v_\theta})-\inv_{v_\theta}\beta(z''_{v_\theta}) \Big)\rlap{.}
\end{align}
The first term of~\eqref{eq:sumrewritten}
is equal to
$\sum_{\theta \in \Theta_{m'}^\sharp}  \chi\big(\phi_{m'}(\theta)\big)$,
according to~\eqref{eq:suminvntchi}.
On the other hand,
 the  definition of~$z_v'$ and of~$z_v'''$
shows that
for any $m \in M$ and any $\theta \in \Theta_m^\sharp$,
the term
$\inv_{v_\theta} \beta(z'''_{v_\theta})-\inv_{v_\theta}\beta(z''_{v_\theta})$
vanishes  if $m\neq m'$ and is equal to $\chi(\Frob_{u_{\theta,2}})-\chi(\Frob_{u_{\theta,1}})$ otherwise
(see~\eqref{eq:invcoreschi}).
In the latter case we have
 $\chi(\Frob_{u_{\theta,2}})-\chi(\Frob_{u_{\theta,1}})=-\chi(\phi_{m'}(\theta))$
(see~\eqref{eq:frobu1u2phim}).
Hence~\eqref{eq:sumrewritten} indeed vanishes.
\end{proof}

Theorem~\ref{th:conjfnrconjfplus} is especially useful
when $B_{\nr}=B_{\const}$.
The following corollary records situations in which this equality holds for purely algebraic reasons.

\begin{cor}
\label{cor:comparisonFconstFplus}
Let $\pi_+\in\sP_+$.
Assume that for each $m \in M$, at least one of the following conditions is satisfied:
\begin{enumerate}
\item\label{it:unramified-is-constant}
the torus $T=R^1_{L_m/k(m)}\Gm$ over~$k(m)$
satisfies $\Br_{\nr}(T)=\Br_0(T)$;
\item\label{it:extension-is-trivial}
the extension $K_m/L_m$ is trivial.
\end{enumerate}
Then Conjecture~$\cFconst$ for~$\pi_+$ implies Conjecture~$\cFplus$ for~$\pi_+$.
\end{cor}

\begin{proof}
Under either assumption, one has $C_{m,\nr}=C_{m,\const}$, according to Proposition~\ref{prop:chinr}~(1)
in case~\ref{it:unramified-is-constant}, and because
 $C_m=0$
in case~\ref{it:extension-is-trivial}.
Hence $B_{\nr}=B_{\const}$.
\end{proof}

Various conditions which ensure
that $\Br_{\nr}(T)=\Br_0(T)$ are listed in
Proposition~\ref{prop:criterion-for-brauer}.

\section{Known cases of Conjecture \texorpdfstring{$\cFplus$}{F₊}}
\label{sec:knowncases}

The concrete cases in which Conjecture~$\cF$ is currently known to hold are those
listed in \cite[\textsection9.2]{hwfibration} and in \cite{browningschindler}.
In many of these cases, the underlying arguments can be enhanced to prove
Conjecture~$\cFplus$ as well, as we verify in this section.

\subsection{From strong approximation}

The following proposition strengthens \cite[Corollary~9.10]{hwfibration},
where the analogous conclusion was obtained for Conjecture~$\cF$.

\begin{prop}
\label{prop:strongapproximpliesfp}
Let $\pi_+ \in \sP_+$.
If the variety~$W$ satisfies strong approximation off~$v_0$ for every finite place~$v_0$ of~$k$,
then Conjecture~$\cFplus$ holds for~$\pi_+$.
\end{prop}

\begin{proof}
Let $(z_v)_{v \in \Omega} \in W(\A_k)$.
Let $S \subset \Omega$ be a finite subset of places.  Let $p:\sW \to \P^1_{\sOint_S}$
be
the standard integral model of $p:W \to \P^1_k$
(see \textsection\ref{subsubsec:localintpoint}).
We choose~$S$ large enough that it satisfies
the assumptions of Proposition~\ref{prop:localintegralpoint}
and that $z_v \in \sW(\sOint_v)$ for all $v \in \Omega\setminus S$.
To prove the theorem, it suffices to produce $c \in U(k)$ and $(z'_v)_{v \in \Omega} \in W_c(\A_k)^B$
such that~$z'_v$ lies arbitrarily close to~$z_v$ for $v \in S$ and that $z'_v \in \sW(\sOint_v)$
for $v \in \Omega\setminus S$.

Fix a place $v_0 \in \Omega \setminus S$  that splits completely in~$K_m$ for all $m \in M$.
By assumption, there exists $z \in \sW\big(\sOint_{S \cup \{v_0\}}\big)$ lying arbitrarily
close to~$z_v$ for $v \in S$.    Set $c=p(z)$.
As~$v_0$ splits completely in~$L_m$ for all $m \in M$,
there exists $z'_{v_0} \in \sW(\sOint_{v_0})$
such that $p(z'_{v_0})=c$
(see Proposition~\ref{prop:localintegralpoint}).
Set $z'_v=z$ for $v \in \Omega \setminus \{v_0\}$.
For any $\beta \in B$, one has
 $\sum_{v \in \Omega} \inv_v \beta(z)=0$
(by the global reciprocity law)
and $\inv_{v_0} \beta(z'_{v_0})=\inv_{v_0} \beta(z)=0$
(since~$v_0$ splits completely in~$K_m$ for all $m \in M$),
hence $\sum_{v \in \Omega} \inv_v \beta(z'_v)=0$, as desired.
\end{proof}

In the next corollary,
case~(iv) is a
delicate
theorem of Browning and Schindler \cite{browningschindler} resting on methods
of analytic number theory.

\begin{cor}
\label{cor:fplusstrongapprox}
Let $\pi_+ \in \sP_+$.  Let $M' = \{m\in M\mkern1mu;\mkern2mu L_m \neq k(m)\}$.
Conjecture~$\cFplus$ holds for~$\pi_+$
under any of the following sets of assumptions:
\begin{enumerate}[label={\upshape(\roman*)}]
\item $\sum_{m\in M'} [k(m):k] \leq 2$;
\item $\sum_{m \in M'} [k(m):k]=3$ and $[L_m:k(m)]=2$ for all $m\in M'$;
\item $\sum_{m \in M'} [k(m):k]=3$ and there exists $m_0 \in M'$ such that $k(m_0)=k$ and such that
$[L_m:k(m)]=2$ for all $m\in M' \setminus \{m_0\}$;
\item $\sum_{m \in M'} [k(m):k]=3$, the set~$M'$ has cardinality~$2$, and $k=\Q$.
\end{enumerate}
\end{cor}

\begin{proof}
We claim that in all cases, the variety~$W$ satisfies strong approximation off~$v_0$
for any finite place~$v_0$ of~$k$.
To prove this,
we may assume that $M=M'$, by Remark~\ref{rk:wwpiso}.
In case~(iv), the claim is then
 \cite[Corollary~2.1]{browningschindler}.
In the other cases, one observes,
as in the proof of \cite[Theorem~9.11]{hwfibration},
that the variety~$W$ is
isomorphic, in case~(i), to the complement of a codimension~$2$ closed subset in an affine
space, in case~(ii), to the punctured affine cone over the complement of a codimension~$2$
closed subset in a smooth projective quadric of dimension~$4$,
and in case~(iii), to a variety as in the statement of Lemma~\ref{lem:puncturedaffineconegeneralised}
below.
Such a variety satisfies
strong approximation off~$v_0$
for any finite place~$v_0$ of~$k$,
in case~(i) by \cite[Proposition~3.6]{caoxustrongtoric}
or \cite[Lemma~1.1]{weistrongtoric} (see also \cite[Lemma~1.8]{hwfibration}),
in cases~(ii) and~(iii) by
Lemma~\ref{lem:puncturedaffineconegeneralised}
and
Lemma~\ref{lem:puncturedaffinecone}.
\end{proof}

\begin{lem}
\label{lem:puncturedaffineconegeneralised}
Let~$k$ be a field of characteristic~$0$,
with algebraic closure~$\bar k$.
Let~$L$ be a nonzero \'etale algebra over~$k$,
of dimension~$m$.
Let $q \in k[x_1,\dots,x_n]$ be a non-degenerate quadratic form.
Denote by~$X$ the smooth closed subvariety
of $R_{L/k}(\A^1_L) \times \left(\A^n_k \setminus \{(0,\dots,0)\}\right)$
 defined by the equation
\begin{align}
N_{L/k}(z) = q(x_1,\dots,x_n)\rlap,
\end{align}
where~$z$ stands for a point of $R_{L/k}(\A^1_L)$ and $x_1,\dots,x_n$ are the coordinates of~$\A^n_k$.
Let $F \subset X$ be a closed subset of
codimension~$\geq 2$
and set $U=X \setminus F$.
\begin{enumerate}[label={\upshape(\roman*)}]
\setlength\itemsep{.5ex}
\item If $n \geq 1$, then~$U$ is geometrically integral.
\item If $n \geq 2$, then $\Br(U_{\bar k})=0$ and $\bar k[U]^*=\bar k^*$.
\item
If $n \geq 3$, then $\Pic(U_{\bar k})=0$.
If $n=2$, then $\Pic(U_{\bar k})\simeq \Z^{m-1}$ (ignoring Galois actions).
\item If $n\geq 3$, the pull-back map $\Br(k)\to \Br(U)$ is onto.
\item
If $n=2$ and both~$L$ and~$q$ are split, then
$\Gal(\bar k/k)$ acts trivially on $\Pic(U_{\bar k})$
and the pull-back map $\Br(k)\to \Br(U)$ is onto.
\end{enumerate}
Assume that~$k$ is a number field and let~$v_0$ be a place of~$k$.
\begin{enumerate}[label={\upshape(\roman*)}]
\setcounter{enumi}{5}
\item If $n=2$, the subset $U(k)$ is dense off~$v_0$ in $U(\A_k)^{\Br(U)}$
(see \cite[Definition~2.9]{wittenbergslc}).
\item If $n\geq 3$, the variety $U$ satisfies strong approximation off~$v_0$.
\end{enumerate}
\end{lem}

\begin{proof}
Let us first prove (ii)--(v).
As the restriction maps $H^q(X_{\bar k},\Gm)\to H^q(U_{\bar k},\Gm)$
and $H^q(X,\Gm)\to H^q(U,\Gm)$
are isomorphisms
for $q \leq 2$
(see \cite[Theorem~3.7.2~(i)]{ctskobook}),
we may assume that $U=X$.
We now argue by induction on~$m$, with $n\geq 2$ fixed but letting $k$, $L$, $q$ vary,
to prove (ii)--(v) under this assumption.
If $m=1$, then
$X=\A^n_k \setminus \{(0,\dots,0)\}$ and (ii)--(v) are true
(see \cite[Theorem~3.7.1, Theorem~6.1.1]{ctskobook}).
Let us fix $m>1$ and assume that (ii)--(v) hold for lower values of~$m$.
The Hochschild--Serre spectral sequence
shows that~(ii) and~(iii) imply~(iv); hence, in order to prove (ii)--(v), we
may assume, after extending the scalars,
that both~$q$ and~$L$ are split over~$k$.
The variety~$X$ is then isomorphic to the
subvariety of $\A^m_k \times \big(\A^n_k \setminus \{(0,\dots,0)\}\big)$,
with coordinates $z_1,\dots,z_m,x_1,\dots,x_n$,
defined by $\prod_{i=1}^m z_i = q(x_1,\dots,x_n)$.
Let $f:X \to \A^1_k$ be the projection to the coordinate~$z_m$.
The generic fibre~$X_\eta$ of~$f$
is a variety of the form appearing in
Lemma~\ref{lem:puncturedaffineconegeneralised}, associated with a split algebra of rank~$m-1$
and a split quadratic form of rank~$n$, over the function field~$K$ of~$\A^1_k$.
Letting~$\bar K$ be an algebraic closure of~$K$,
the induction hypothesis therefore guarantees that
$\bar K[X_\eta]^*=\bar K^*$ (hence $\bar k[X]^*=\bar k^*$, as~$f$ is surjective),
that $\Br(X_\eta \otimes_K \bar K)=0$,
that
$\Pic(X_\eta \otimes_K \bar K)=0$ if $n \geq 3$,
and that $\Pic(X_\eta \otimes_K \bar K)\simeq \Z^{m-2}$,
with trivial action of $\Gal(\bar K/K)$, if $n=2$.
As
 $\Br(K \otimes_k \bar k)=0$ (Tsen's theorem)
and  $H^1(K\otimes_k\bar k,\Pic(X_\eta \otimes_K \bar K))\simeq H^1(K\otimes_k\bar k,\Z^{m-2})=0$
if $n=2$,
one deduces, by the Hochschild--Serre spectral sequence,
that
\begin{align}
\label{eq:picxetabark}
\Pic(X_{\eta} \otimes_k \bar k)\simeq\begin{cases}
0 & \text{ if } n\geq 3\text{,} \\
\Z^{m-2} & \text{ if } n=2\text{,}
\end{cases}
\end{align}
with trivial action of $\Gal(\bar k/k)$,
and
that
$\Br(X_\eta \otimes_k \bar k)=0$.
It follows that $\Br(X_{\bar k})=0$,
since $\Br(X_{\bar k}) \subseteq \Br(X_\eta \otimes_k \bar k)$.
As the last part of~(v) results,
again by the Hochschild--Serre spectral sequence,
 from the rest of (ii)--(v),
we need only check that
\begin{align}
\label{eq:picxbarkneedtoprove}
\Pic(X_{\bar k})\simeq\begin{cases}
0 & \text{ if } n\geq 3\text{,} \\
\Z^{m-1} & \text{ if } n=2\text{,}
\end{cases}
\end{align}
with trivial action
of $\Gal(\bar k/k)$.
As~$X_{\bar k}$ is smooth and as $\Pic(\A^1_{\bar k})=0$, we have an exact sequence of $\Gal(\bar k/k)$\nobreakdash-modules
\begin{align}
\label{eq:npicxbarkpicxeta}
N \to \Pic(X_{\bar k}) \to \Pic(X_{\eta} \otimes_k \bar k) \to 0\rlap,
\end{align}
where~$N$ is the quotient of the group of divisors on~$X_{\bar k}$ supported on the fibres of~$f$ by
the subgroup $f^*\Div(\A^1_{\bar k})$.  
When $n \geq 3$, the fibres of~$f$ are geometrically integral, hence $N=0$,
which proves~\eqref{eq:picxbarkneedtoprove} in this case, in view of~\eqref{eq:picxetabark}.
Assume now that $n=2$. The fibres of~$f$ are then geometrically integral except $f^{-1}(0)$,
which is
the disjoint union of two geometrically integral subvarieties, 
say~$E$ and~$E'$.
Thus $N=\Z$, generated by the class of~$E_{\bar k}$, with trivial
action of~$\Gal(\bar k/k)$.
As $\bar K[X_\eta]^*=\bar K^*$, the class of~$E_{\bar k}$ in~$\Pic(X_{\bar k})$ has infinite order:
the first arrow of~\eqref{eq:npicxbarkpicxeta} is injective.
As any extension of $\Gal(\bar k/k)$\nobreakdash-modules of $\Z^{m-2}$
by~$\Z$ is equivariantly split, this completes the proof of~\eqref{eq:picxbarkneedtoprove}, and hence of (ii)--(v).

Let us prove~(i).
As~$U$ is smooth and nonempty, if it were not geometrically integral, it would not be geometrically connected,
which would contradict the last part of~(ii) if $n\geq 2$.
When $n=1$, the variety~$X_{\bar k}$ is isomorphic to the subvariety
of $\A^{m+1}_{\bar k}$, with coordinates $z_1,\dots,z_m,x$, defined by
$\prod_{i=1}^m z_i=x^2\neq 0$, which is indeed connected.

To prove~(vi), we note that if $n=2$, then~$X$ is a toric variety.
Indeed, after a change of variables, we may assume that~$q$ is diagonal, i.e.\ $q(x_1,x_2)=b(x_1^2-cx_2^2)$
for some $b,c\in k^*$;
the equations
$N_{L/k}(z)=x_1^2-cx_2^2 \neq 0$ then define a torus which acts on~$X$
with a dense open orbit.
As the variety~$X$ is toric and satisfies $\bar k[X]^*=\bar k^*$,
the main theorem of \cite{weistrongtoric} ensures the validity of~(vi).

We shall now prove that the assertion of~(vi) in fact holds for all $n \geq 2$,
by induction on~$n$.
In view of~(iv), this will establish~(vii).

Let $n \geq 3$ be such that the assertion of~(vi)
holds for smaller values of~$n$.
Choose a codimension~$2$ linear subspace
$D \subset \A^n_k$ containing $(0,\dots,0)$
such that $X \cap \big(R_{L/k}(\A^1_L) \times D\big)$ is smooth
and that
$F \cap \big(R_{L/k}(\A^1_L) \times D\big)$ has
codimension~$\geq 1$ in~$F$.
Write $\Lambda$ for the projective line parametrising
hyperplanes in~$\A^n_k$ containing~$D$.
Let $g:X' \to X$ be the blow-up of~$X$ along
$X \cap \big(R_{L/k}(\A^1_L) \times D\big)$
and $f:X' \to \Lambda$ the morphism obtained by composing~$g$, the second projection
$X \to \A^n_k$ and the rational map $\A^n_k \dashrightarrow \Lambda$ given by projection from~$D$.
The fibres of~$f$ are the varieties
$X \cap \big(R_{L/k}(\A^1_L) \times H\big)$
where $H$ ranges over the hyperplanes of~$\A^n_k$ containing~$D$.
As the restriction of~$q$ to any hyperplane of~$\A^n_k$ is a quadratic form of rank~$\geq n-2$,
and as $n \geq 3$, we deduce that the fibres of~$f$ are geometrically integral.
Let  $F'=g^{-1}(F)$
and $U'=X'\setminus F'$.
As~$F'$ has codimension~$\geq 2$ in~$X'$, the geometric generic fibre~$U'_{\etabar}$
of $f|_{U'}:U' \to \Lambda$
is a variety of the form appearing in
Lemma~\ref{lem:puncturedaffineconegeneralised}
(with~$n$ replaced by~$n-1$).
In particular,
it has no non-constant
invertible function, by~(ii), and the abelian group
$\Pic(U'_{\etabar})$ is torsion-free, by~(iii),
so that
$H^1_\et(U'_{\etabar},\Q/\Z)=0$;
and~(ii) ensures that $\Br(U'_{\etabar})=0$.
We can therefore apply Corollary~\ref{cor:concretecorollary} to~$f|_{U'}$ (recalling that $\Lambda \simeq \P^1_k$).
The parameters~$\pi_+$ which appear in the statement of Corollary~\ref{cor:concretecorollary}
satisfy $L_m=k(m)$ for all $m \in M$, so that Conjecture~$\cFplus$ holds for them,
by Corollary~\ref{cor:fplusstrongapprox}~(i).
We conclude that any point of $U'(\A_k)^{\Br(U')}$ can be approximated arbitrarily well
by a point of $U'_c(\A_k)^{\Br(U'_c)}$ for a rational point~$c$ of an arbitrary dense open
subset of~$\Lambda$.  By the induction hypothesis, this point
of $U'_c(\A_k)^{\Br(U'_c)}$
can in turn be approximated,
for the adelic topology off~$v_0$, by a rational point of~$U'_c$.  This shows that $U'(k)$ is
dense off~$v_0$ in $U'(\A_k)^{\Br(U')}$.
On the other hand, as the map $U'\to U$ induced by~$g$ is a blow-up with smooth centre,
it induces a surjection $U'(\A_k)^{\Br(U')} \twoheadrightarrow U(\A_k)^{\Br(U)}$.
Hence
 $U(k)$ is
dense off~$v_0$ in $U(\A_k)^{\Br(U)}$.
\end{proof}

\begin{lem}
\label{lem:puncturedaffinecone}
Let~$X$ be the punctured affine cone over an $n$\nobreakdash-dimensional smooth projective quadric,
over a number field~$k$.
Let~$v_0$ be a place of~$k$.
Let $U=X \setminus F$, where $F \subset X$ is a closed subset of
codimension~$\geq 2$.
If $n=2$, the subset $U(k)$ is dense off~$v_0$ in $U(\A_k)^{\Br(U)}$.
If $n\geq 3$, the variety $U$ satisfies strong approximation off~$v_0$.
\end{lem}

\begin{rmk}
\label{rem:puncturedaffinecone}
In the case $n=4$,
Lemma~\ref{lem:puncturedaffinecone}
was asserted and used in the proof of
 \cite[Theorem~9.11]{hwfibration}.
However,
a gap in the justification given there
(as well as a very simple fix for the proof of that theorem,
by working around Lemma~\ref{lem:puncturedaffinecone})
was kindly pointed out to us by Yang Cao.
Thus, Lemma~\ref{lem:puncturedaffinecone}
is new
 and requires a proof,
even for $n=4$. The proof we give also fills
the aforementioned gap.
\end{rmk}

\begin{proof}[Proof of Lemma~\ref{lem:puncturedaffinecone}]
Choose, for the underlying quadric, a diagonal equation
of the form $\sum_{i=1}^{n+2} a_ix_i^2=0$ with $a_{n+2}=-1$.
Set
$q(x_1,\dots,x_n)=\sum_{i=1}^n a_ix_i^2$ and $L=k[t]/(t^2-a_{n+1})$.
The variety~$X^0$ associated with~$L$ and~$q$ by Lemma~\ref{lem:puncturedaffineconegeneralised}
is an open subvariety of~$X$ whose complement has codimension~$\geq 2$ in~$X$.
Hence the conclusions of Lemma~\ref{lem:puncturedaffineconegeneralised}
for $U^0 = X^0 \cap U$ imply the same conclusions for~$U$.
\end{proof}

\subsection{From Schinzel's hypothesis}
\label{subsec:from schinzel}

The next theorem strengthens \cite[Theorem~9.6]{hwfibration},
where the same conclusion was obtained for Conjecture~$\cF$
rather than Conjecture~$\cFconst$.
In its statement, almost abelian extensions are meant in
 the sense of \cite[Definition~9.4]{hwfibration}.
Let us recall that
abelian extensions as well as cubic extensions are almost abelian
(see the end of~\textsection\ref{sec:intro} for the definition).
We refer the reader to \cite[\textsection9.2.1]{hwfibration} for the statement
of Schinzel's hypothesis in its homogeneous form~$(\mathrm{HH}_1)$,
which follows from the more common Schinzel's hypothesis~$(\mathrm{H})$
and which takes,
as input, a finite collection of homogeneous polynomials in $k[\lambda,\mu]$.
Given a closed point $m \in \A^1_k$, we set $P_m(\lambda,\mu)=N_{k(m)/k}(\lambda-a_m\mu)$.

\begin{thm}
\label{th:hhfconst}
Let $\pi_+ \in \sP_+$.
Assume that for each $m\in M$, the extension $L_m/k(m)$ is almost abelian.
If Schinzel's hypothesis~$(\mathrm{HH}_1)$
holds for the set of homogeneous polynomials $(P_m(\lambda,\mu))_{m \in M'}$,
where $M' = \{m\in M\mkern1mu;\mkern2mu L_m \neq k(m)\}$,
then Conjecture~$\cFconst$ holds for~$\pi_+$.
\end{thm}

\begin{proof}
Let $U = \P^1_k \setminus M$
and let $q:W \to \A^2_k \setminus \{(0,0)\}$ and $r:\A^2_k \setminus \{(0,0)\}\to \P^1_k$ denote the natural projections,
so that $p=r \circ q$.
Under our assumptions,
Conjecture~$\cF$ holds for~$\pi$, by Remark~\ref{rk:wwpiso} and \cite[Theorem~9.6, Proposition~9.9]{hwfibration}.
The proof given in \emph{loc.\ cit.}\ is in fact a proof of the more precise result that
the subset of $W(\A_k)$ consisting of the adelic points that lie in a fibre of~$q$
over a rational point of $\A^2_k \setminus \{(0,0)\}$
is dense in $W(\A_k)$.
On the other hand,
for any $m \in M$ and any $\chi \in H^1(k(m),\Q/\Z)$,
we have
\begin{align*}
\Cores_{L_m/k}(z_m,\chi)=\Cores_{k(m)/k}(N_{L_m/k(m)}(z_m),\chi)=\Cores_{k(m)/k}(b_m(\lambda-a_m\mu),\chi)\rlap{,}
\end{align*}
hence $B_{\const} \subseteq q^*\Br(r^{-1}(U))$
and any adelic point of~$p^{-1}(U)$ that lies in a fibre of~$q$
over a rational point of $\A^2_k \setminus \{(0,0)\}$
is therefore
automatically orthogonal to $B_{\const}$ with respect to the Brauer--Manin pairing.
\end{proof}

\begin{cor}
\label{cor:hh1fplus}
Let $\pi_+ \in \sP_+$.
Assume that for each $m\in M$, at least one of the following conditions is satisfied:
\begin{enumerate}
\item the extension $L_m/k(m)$ is cyclic, or it is almost abelian but not abelian;
\item the extension $L_m/k(m)$ is abelian and the extension $K_m/L_m$ is trivial.
\end{enumerate}
Let $M' = \{m\in M\mkern1mu;\mkern2mu L_m \neq k(m)\}$
and assume that Schinzel's hypothesis~$(\mathrm{HH}_1)$
holds for the set of homogeneous polynomials $(P_m(\lambda,\mu))_{m \in M'}$.
Then Conjecture~$\cFplus$ holds for~$\pi_+$.
\end{cor}

\begin{proof}
By Theorem~\ref{th:hhfconst},
Conjecture~$\cFconst$ holds for~$\pi_+$.
By Corollary~\ref{cor:comparisonFconstFplus},
Proposition~\ref{prop:criterion-for-brauer} and
Remark~\ref{rk:brnrtrivial}, Conjecture~$\cFconst$ for~$\pi_+$
implies Conjecture~$\cFplus$ for~$\pi_+$.
\end{proof}

The work of
Heath-Brown and Moroz~\cite{hbm}
on primes represented by binary cubic forms
implies the validity of
Schinzel's hypothesis~$(\mathrm{HH}_1)$
for a single polynomial of degree~$3$ with coefficients in~$\Q$
(see \cite[Remark~9.7]{hwfibration}).
We thus obtain the following corollary (to be compared with
Corollary~\ref{cor:fplusstrongapprox}~(iv)):

\begin{cor}
\label{cor:heathbrownmoroz}
Let $\pi_+ \in \sP_+$.
Assume that $k=\Q$, that there is a unique $m \in M$ such that $L_m \neq k(m)$,
 that this~$m$ is such that $[k(m):k]=3$,
and that for this~$m$, the extension $L_m/k(m)$ is  cyclic or is almost abelian but non-abelian.
Then Conjecture~$\cFplus$ holds for~$\pi_+$.
\end{cor}

\subsection{From additive combinatorics}
\label{subsec:additivecomb}

Under the assumption that $k=\Q$,
Conjecture~$\cF$ is known to hold
for triples $\pi \in \sP$ such that~$M$ only consists of rational points.
This was proved by Matthiesen \cite{matthiesen}
(see \cite[Theorem~9.14]{hwfibration}),
following her work with Browning \cite{browningmatthiesen} and
using the methods of additive combinatorics developed by Green, Tao and Ziegler
\cite{gt0, gt1, gt2, gtz3}.
The next theorem strengthens this result, by replacing Conjecture~$\cF$
with Conjecture~$\cFplus$ while allowing non-rational points in~$M$ with trivial
extension $L_m/k(m)$ (and arbitrary finite abelian extensions $K_m/L_m$).

\begin{thm}
\label{thm:lilian}
Let $\pi_+\in\sP_+$.
Assume that $k=\Q$ and that for each $m \in M$,
at least one of the extensions $L_m/k(m)$ and $k(m)/k$ is trivial.
Then Conjecture~$\cFplus$ holds for~$\pi_+$.
\end{thm}

\begin{proof}
Let
$\pi_+^0=(M, (L_m^0)_{m \in M}, (b_m^0)_{m \in M}, (K_m^0)_{m\in M}) \in \sP_+$
satisfy the assumptions of the theorem. We shall prove Conjecture~$\cFplus$ for~$\pi_+^0$.

Set
 $M' = \{m\in M\mkern1mu;\mkern2mu L_m^0 \neq k(m)\}$.
If $M'=\emptyset$, then
Conjecture~$\cFplus$ holds for~$\pi^0_+$ by Corollary~\ref{cor:fplusstrongapprox}~(i).
Otherwise, we choose $m_0 \in M'$ and note that $k(m_0)=k$ by assumption.
For $m \in M \setminus M'$,
let us set
 $K_m=K_m^0$ and $L_m=L_m^0=k(m)$.
For $m \in M' \setminus \{m_0\}$,
let us set $K_m=L_m=K_m^0$.
Let~$L_0/k$ denote the field extension given by
Corollary~\ref{cor:comparisonFFconst} applied to the extensions~$K_m/L_m/k(m)$
for $m \in M \setminus \{m_0\}$.
Finally, let us choose a finite extension~$L_{m_0}$ of~$K^0_{m_0}$ in which~$L_0$ can be embedded
$k$\nobreakdash-linearly, and set $K_{m_0}=L_{m_0}$.

According to Corollary~\ref{cor:fplusimpliesfplus}, we will be done if we prove Conjecture~$\cFplus$
for the parameter $\pi_+=(M, (L_m)_{m \in M}, (b_m)_{m \in M}, (K_m)_{m\in M})$
for any choice of $(b_m)_{m\in M}\in \prod_{m\in M}k(m)^*$.
Let us fix $(b_m)_{m\in M}$.  By Matthiesen's theorem \cite[Theorem~9.14, Proposition~9.9]{hwfibration}
and Remark~\ref{rk:wwpiso},
Conjecture~$\cF$ holds for~$\pi_+$.
By the definition of~$L_0$
(see Corollary~\ref{cor:comparisonFFconst}),
it follows
that Conjecture~$\cFconst$ holds for$~\pi_+$.
By Remark~\ref{rks:surj99}~(ii),
we conclude that Conjecture~$\cFplus$ holds for~$\pi_+$, as desired.
\end{proof}

\section{Applications}
\label{sec:newapplications}

As was the case for Conjecture~$\cF$ in~\cite{hwfibration}, our motivation
for Conjecture~$\cFplus$ ultimately lies in the following question (which
slightly refines Question~\ref{q:fibration-intro} by incorporating a
Hilbert subset into its statement):

\begin{question}
\label{q:fibration}
Let~$X$ be a smooth, proper, irreducible variety over a number field~$k$.
Let $f:X \to \P^1_k$ be a dominant morphism whose geometric generic fibre
is rationally connected.
Assume that $X_c(k)$ is dense in $X_c(\A_k)^{\Br(X_c)}$
for all rational points~$c$ of a Hilbert subset of~$\P^1_k$.
Does it follow that $X(k)$ is dense in $X(\A_k)^{\Br(X)}$?
\end{question}

Question~\ref{q:fibration} admits an affirmative answer if Conjecture~$\cF$ (or
Conjecture~$\cFplus$) holds true,
by \cite[Corollary~9.25]{hwfibration} (or Corollary~\ref{cor:concretecorollary} and Remarks~\ref{rk:onconcretecorollary} (i)--(ii)), and unconditionally, in various
special cases listed in \cite[\textsection9.4]{hwfibration}
and \cite{browningschindler}, the most notable one being when $k=\Q$ and the non-split fibres
of~$f$ lie over rational points of~$\P^1_k$ (using Matthiesen's theorem, see \cite[Theorem~9.28]{hwfibration}).

\subsection{Statements}

We now turn to the new affirmative answers to Question~\ref{q:fibration}
that can be obtained by combining the main results of \textsection\textsection\ref{sec:fibration}--\ref{sec:knowncases}.  In the statements below, we fix~$X$ and~$f$ as in Question~\ref{q:fibration},
and let $M \subset \P^1_k$ denote the locus of non-split fibres of~$f$.
For each $m \in M$, we choose an irreducible homogeneous polynomial $P_m(\lambda,\mu)$ that
vanishes on~$m$, where $\lambda, \mu$ denote homogeneous coordinates of~$\P^1_k$.
Following a terminology
introduced by Skorobogatov \cite{skorodescent}, the \emph{rank} of~$f$, denoted $\rank(f)$,
is the degree of~$M$ over~$k$ (viewing~$M$ as a reduced closed subscheme of~$\P^1_k$).
Finally, for $m \in M$, we say that a finite extension $L_m$ of~$k(m)$ \emph{splits the fibre~$X_m$}
if the variety $X_m \otimes_{k(m)} L_m$ over~$L_m$ is split.

\begin{thm}
\label{thm:schinzelcyclic}
Assume that the following two conditions are satisfied:
\begin{enumerate}
\item Schinzel's hypothesis~$(\mathrm{HH}_1)$
holds for the homogeneous polynomials $(P_m(\lambda,\mu))_{m \in M}$.
\item For each $m \in M$, there exists an extension of~$k(m)$ that splits the fibre $X_m$
and that is
either cyclic or almost abelian but non-abelian (e.g.\ a cubic extension).
\end{enumerate}
Then Question~\ref{q:fibration} admits an affirmative answer.
\end{thm}

\begin{proof}
Combine Corollary~\ref{cor:concretecorollary}, Remarks~\ref{rk:onconcretecorollary}~(i) and~(ii),
the invariance
of~$(\mathrm{HH}_1)$
under changes of coordinates of~$\P^1_k$ (if $X_\infty$ is singular), and
Corollary~\ref{cor:hh1fplus}.
\end{proof}

Theorem~\ref{thm:schinzelcyclic} recovers and generalises a theorem of Smeets \cite[Corollaire~1.5]{smeets},
who dealt with the special case where the generic fibre of~$f$ is a torsor under a torus defined over~$k$ and quasi-split by a cyclic extension of~$k$.
Apart from this case,
Theorem~\ref{thm:schinzelcyclic} was previously known only under the assumption
that the smooth fibres of~$f$ satisfy weak approximation
(in which case~$X_m$ can be allowed to be split by an abelian extension;
see \cite[Corollary~9.27]{hwfibration},
which expanded on \cite{ctsksd98} and on
\cite[Theorem~4.6]{weioneq}).

\begin{thm}
\label{th:applicationrank2}
Question~\ref{q:fibration} admits an affirmative answer if $\rank(f)\leq 2$.
\end{thm}

\begin{proof}
Combine Corollary~\ref{cor:concretecorollary}, Remarks~\ref{rk:onconcretecorollary}~(i) and~(ii),
and Corollary~\ref{cor:fplusstrongapprox}~(i).
\end{proof}

Theorem~\ref{th:applicationrank2} was previously known only under the assumption
that~$k$ is totally imaginary or that~$M$ consists of rational points of~$\P^1_k$
(see \cite[Theorem~9.31]{hwfibration}).

\begin{thm}
\label{thm:applicationrank3}
Assume that $\rank(f)=3$ and that at least one of the following holds:
\begin{enumerate}[label={\upshape(\roman*)}]
\item $k=\Q$ and~$f$ has at least two non-split fibres;
\item\label{it:cyclic} $k=\Q$, the morphism~$f$ has a unique non-split fibre, say over $m$,
and~$X_m$ is split
by an extension of $k(m)$ that is
either cyclic or almost abelian but non-abelian;
\item\label{it:quadratic} for every $m \in M$, the fibre~$X_m$ is split by a quadratic extension of~$k(m)$;
\item\label{it:quadraticbis} there exists $m_0 \in M$ such that $k(m_0)=k$ and such that for every $m \in M \setminus\{m_0\}$, the fibre~$X_m$ is split by a quadratic extension of~$k(m)$.
\end{enumerate}
Then Question~\ref{q:fibration} admits an affirmative answer.
\end{thm}

\begin{proof}
Combine Corollary~\ref{cor:concretecorollary} and Remarks~\ref{rk:onconcretecorollary}
with
Corollary~\ref{cor:fplusstrongapprox}~(iv) (which builds on the work of Browning and Schindler)
in case~(i) if~$f$ has two non-split fibres,
with Theorem~\ref{thm:lilian}
(which builds on the work of Matthiesen)
 in case~(i)
 if~$f$ has three non-split fibres,
with Corollary~\ref{cor:heathbrownmoroz}
(which builds on the work of Heath-Brown and Moroz)
 in case~(ii),
 with Corollary~\ref{cor:fplusstrongapprox}~(ii)  in case~(iii)
and with Corollary~\ref{cor:fplusstrongapprox}~(iii)  in case~(iv).
\end{proof}

The above theorem collects everything that can be proved to this day about
Question~\ref{q:fibration} when~$f$ has rank~$3$.
Theorem~\ref{thm:applicationrank3} in case~(i) is \cite[Theorem~1.1]{browningschindler}
and is only stated for the record.
Cases~(ii), (iii) and~(iv), on the other hand, are entirely new.
Among them, the only previously known particular case
was a theorem
of
Colliot-Thélène and Skorobogatov \cite[Theorem~B]{ctskodescent},
who
 had established
Theorem~\ref{thm:applicationrank3} in case~(iii) under the
assumption that the smooth fibres of~$f$ satisfy weak approximation.

\subsection{Examples}
\label{subsec:examples}

We now describe some concrete examples of varieties for which one can prove
the density of rational points in the Brauer--Manin set by applying
Theorem~\ref{thm:applicationrank3} to rank~$3$ fibrations
meeting the requirements~\ref{it:cyclic}, \ref{it:quadratic} or \ref{it:quadraticbis} of its statement.

\subsubsection{An example for Theorem~\ref{thm:applicationrank3}~\ref{it:cyclic}}
\label{subsubsec:example 7.1(ii)}

Consider a number field~$k$
and a
nonzero étale algebra $L = \prod_i L_i$
over~$k$, where the~$L_i$ are number fields.
The arithmetic of smooth and proper models~$X$ of the affine
closed subvariety of $R_{L/k}(\A^1_L) \times \A^1_k$ defined by the equation
\begin{align*}
N_{L/k}(z) = p(t)\rlap,
\end{align*}
where~$z$ and~$t$ are coordinates in $R_{L/k}(\A^1_L)$ and in~$\A^1_k$
respectively,
and where $p \in k[t]$ is a polynomial in one variable,
has been extensively studied in the literature
(see e.g.\ 
\cite{heathbrownskorobogatov,cthasko,varillyviraychatelet,browningheathbrown,swarbrickjones,weioneq,derenthalsmeetswei,browningmatthiesen,irving,shute}).
One can always choose $X$ so that the projection $(z,t) \mapsto t$ extends to a
morphism $f: X \to \P^1_k$
each of whose smooth fibre is a compactification of a
torsor under the norm torus
defined by $N_{L/k}(z)=1$.
Then the generic fibre of~$f$ is
rationally connected
 and $X_c(k)$ is dense in $X_c(\A_k)^{\Br(X_c)}$
for all $c \in \P^1(k)$ such that~$X_c$ is smooth,
by a theorem of Colliot-Thélène and Sansuc
 (see~\cite[Theorem~6.3.1]{skobook}).
Applying Theorem~\ref{thm:applicationrank3}~\ref{it:cyclic} thus yields the
following:

\begin{cor}\label{cor:norms}
Let $b \in \Q^*$.
Let $e \geq 1$ be an integer.
Let $q \in \Q[t]$ be an irreducible cubic polynomial.
Set $E=\Q[t]/(q(t))$.
Let $L=\prod L_i$ be a nonzero \'etale $\Q$\nobreakdash-algebra,
where the~$L_i$ are number fields.
Let $X$ be a smooth and proper model over~$\Q$ of
the closed subvariety of $R_{L/\Q}(\A^1_L) \times \A^1_\Q$ defined by the equation
\begin{align*}
N_{L/\Q}(z) = bq(t)^e\rlap,
\end{align*}
where~$z$ denotes a point of $R_{L/\Q}(\A^1_L)$ and~$t$
is the coordinate of $\A^1_\Q$.
Suppose that the following conditions hold:
\begin{enumerate}
\item\label{it:infinity}
The gcd of the degrees $[L_i:\Q]$ divides $3e$.
\item
Writing the \'etale
$E$\nobreakdash-algebra $L \otimes_\Q E$ as a product of fields,
at least one of the factors is an extension of~$E$ that is
either cyclic or almost abelian but non-abelian.
\end{enumerate}
Then the subset $X(k)$ is dense in $X(\A_k)^{\Br(X)}$.
\end{cor}

The r\^ole of condition~\ref{it:infinity} in the above corollary is to ensure that when~$X$ is chosen
in such a way that the projection $(z,t)\mapsto t$ extends to a morphism
 $f: X \to \P^1_\Q$ (which we can assume, as the conclusion of the corollary
is a birational
invariant \cite[Remark~2.4~(iv)]{wittenbergslc}),
 the fibre $f^{-1}(\infty)$ is split.
 The following homogeneous variant of this example permits one to dispense with this condition:

\begin{cor}\label{cor:norms-2}
Let $b\in\Q^*$.
Let $e\geq 1$ be an integer.
Let $q \in \Q[\lambda,\mu]$ be an irreducible homogeneous cubic polynomial.
Let $L=\prod L_i$ be a nonzero \'etale $\Q$\nobreakdash-algebra,
where the~$L_i$ are number fields.
Let~$Y$ be a smooth and proper model over~$\Q$ of the closed
subvariety of $R_{L/\Q}(\A^1_L) \times (\A^2_\Q \setminus \{(0,0)\})$ defined by the equation
\begin{align*}
N_{L/\Q}(z) = bq(\lambda,\mu)^e\rlap,
\end{align*}
where $z$ denotes a point of $R_{L/\Q}(\A^1_L)$ and $\lambda,\mu$ are
the coordinates of
 $\A^2_\Q \setminus \{(0,0)\})$.
Setting $E=\Q[t]/(q(t,1))$
and writing the \'etale
$E$\nobreakdash-algebra $L \otimes_\Q E$ as a product of fields,
suppose that
at least one of the factors is an extension of~$E$ that is
either cyclic or almost abelian but non-abelian.
Then the subset $Y(k)$ is dense in $Y(\A_k)^{\Br(X)}$.
\end{cor}

\begin{rmk}
\label{rmk:oncor2}
When the gcd of the degrees $[L_i:\Q]$ divides $3e$, the variety $Y$ considered in Corollary~\ref{cor:norms-2} is birationally equivalent to $X \times \Gm$, where $X$ is the variety associated in Corollary~\ref{cor:norms}
with the polynomial $q(t,1)$.
As a result, Corollary~\ref{cor:norms} is in fact equivalent to a special case
of Corollary~\ref{cor:norms-2}.
We have nevertheless opted for stating Corollary~\ref{cor:norms} separately
in view of the
considerable attention
that the variety~$X$ has received
in the literature
(see the references at the beginning of~\textsection\ref{subsubsec:example 7.1(ii)}).
\end{rmk}

\subsubsection{An example for Theorem~\ref{thm:applicationrank3}~\ref{it:quadraticbis}}

We keep the set-up of~\textsection\ref{subsubsec:example 7.1(ii)} and
assume that~$L$ is a quartic extension of~$k$ and that $p\in k[t]$ is an
irreducible quadratic polynomial having its roots in~$L$.  In this case, using
techniques from analytic number theory, Browning and
Heath-Brown~\cite[Theorem~1]{browningheathbrown} established the Hasse
principle and weak approximation for~$X$ when~$k=\Q$.
Derenthal, Smeets and the second-named author then provided a second proof based on the descent method,
which led, in \cite[Theorem~1]{derenthalsmeetswei}, to the validity of the
same result over an arbitrary number field~$k$.
These authors also verified the equality $\Br(X)=\Br_0(X)$
 in the case under consideration
(see \cite[Theorem~4]{derenthalsmeetswei}).
The theorem of Browning and Heath-Brown was subsequently understood to fit into the framework of the
fibration method: indeed it can be seen as an application of
Browning and Schindler's
Theorem~\ref{thm:applicationrank3}~(i), in view of the equality $\Br(X)=\Br_0(X)$.
Until the present work, however, its generalisation
 to arbitrary number fields
had remained outside of the scope of the fibration method.
We remedy this
gap
 with Theorem~\ref{thm:applicationrank3}~(iv).
Applied to $f:X\to \P^1_k$ (with $m_0=\infty$),
the latter immediately yields
 the density of~$X(k)$ in $X(\A_k)^{\Br(X)}$,
thus recovering
 \cite[Theorem~1]{derenthalsmeetswei}
since
 $\Br(X)=\Br_0(X)$.

\subsubsection{An example for Theorem~\ref{thm:applicationrank3}~\ref{it:quadratic}}

Let us start with a field~$k$ of characteristic~$0$
and a reduced closed subscheme $M \subset \P^1_k$ of degree~$3$ over~$k$.
Let $U = \P^1_k \setminus M$.
Let~$C$ be a smooth projective curve over~$k$
(which we do not assume to be connected or geometrically connected) and $\pi:C \to \P^1_k$ be a finite
morphism satisfying the following condition:
\begin{enumerate}
\item[($*$)] the morphism~$\pi$ is \'etale over~$U$ and for every $m \in M$, the gcd of the ramification
indices of~$\pi$ at the points of~$\pi^{-1}(m)$ is equal to~$2$.
\end{enumerate}
For $b \in k^*$, we consider the closed subvariety~$X^0$ of
$R_{C/\P^1_k}(\A^1_k \times C)$ defined by
\begin{align*}
N_{C/\P^1_k}(z)=b
\end{align*}
and let $f^0:X^0\to \P^1_k$ denote the projection.

\begin{cor}
\label{cor:norms-3}
For any $M$, $C$, $\pi$, $b$ as above, and any smooth and proper model~$X$ of~$X^0$,
if~$k$ is a number field, the subset $X(k)$ is dense in $X(\A_k)^{\Br(X)}$.
\end{cor}

\begin{proof}
The variety~$X^0$ is smooth (see Lemma~\ref{lem:fibres of f0} below).
As the conclusion of the corollary is a birational invariant, we may therefore assume that~$X$
contains~$X^0$ as a dense open subset and that~$f^0$ extends to a morphism $f:X\to \P^1_k$.
The corollary then results from
Theorem~\ref{thm:applicationrank3}~\ref{it:quadratic} in view of the following description of the
fibres of~$f^0$.
\end{proof}

\begin{lem}
\label{lem:fibres of f0}
The morphism~$f^0$ is smooth.
Its fibres over~$U$ are geometrically integral.
For $m \in M$, the fibre $(f^0)^{-1}(m)$ is split by the extension $k(m)\big(\sqrt b\big)/k(m)$.
\end{lem}

\begin{proof}
As the morphism~$f^0$ is obtained by base change from the norm map
\begin{align}
\label{eq:norm weil restr}
N_{C/\P^1_k}: R_{C/\P^1_k}(\Gm \times C) \to \Gm \times \P^1_k
\end{align}
it suffices to prove that the latter is smooth,
and to describe its fibres.

The fibre of~\eqref{eq:norm weil restr} above an arbitrary point $(b,m)$ of $\Gm \times \P^1_k$
is the closed subvariety of $R_{\pi^{-1}(m)/m}(\A^1_k \times \pi^{-1}(m))$ defined by the equation
$N_{\pi^{-1}(m)/m}(z)=b$.
Let us choose an isomorphism $\pi^{-1}(m)=\coprod_{i=1}^s \Spec\big(k_i[v]/(v^{e_i})\big)$,
where $k_1,\dots,k_s$ are finite extensions of~$k(m)$
and $e_1,\dots,e_s$ are the ramification indices.
There results an isomorphism
\begin{align}
\label{eq:isoweil restr}
R_{\pi^{-1}(m)/m}(\A^1_k \times \pi^{-1}(m))=\prod_{i=1}^s \prod_{j=0}^{e_i-1}R_{k_i/k(m)}\A^1_{k_i}\rlap.
\end{align}
Letting $z_{i,j}$ stand for a point of the corresponding factor $R_{k_i/k(m)}\A^1_{k_i}$
in the right-hand side of~\eqref{eq:isoweil restr},
the equation
$N_{\pi^{-1}(m)/m}(z)=b$
is rewritten, through this isomorphism,
as
\begin{align}
\label{eq:normkikmeib}
\prod_{i=1}^s N_{k_i/k(m)}(z_{i,0})^{e_i}=b\rlap.
\end{align}
All in all,
the fibre of~\eqref{eq:norm weil restr}
above $(b,m)$ is isomorphic to
$Z \times \A_{k(m)}^{\deg(\pi)-\sum_{i=1}^s [k_i:k(m)]}$,
where~$Z$ denotes the closed subvariety of $\prod_{i=1}^s R_{k_i/k(m)}\Gm$ defined by~\eqref{eq:normkikmeib}.

Let $e$ denote the gcd of the~$e_i$.
It is easy to see that~$Z$
is a torsor under a group of multiplicative type over~$k(m)$ which is an extension of~$\mmu_e$ by a torus,
and that the torsor under~$\mmu_e$ induced by~$Z$ is the closed subvariety of
$\mathbf{G}_{\mathrm{m},{k(m)}}$ defined by $z^e=b$.  Hence~$Z$ is geometrically integral if
$e=1$, and in any case it is split by $k(m)(b^{1/e})/k(m)$.

Thanks to~($*$), we have now proved the second and third assertions of the lemma.
Our description of the fibres of~\eqref{eq:norm weil restr} also shows that they are smooth
and all have the same dimension (namely $\deg(\pi)-1$).
As in addition~\eqref{eq:norm weil restr}
is a finite type morphism between regular schemes (indeed, between smooth $\P^1_k$\nobreakdash-schemes,
see \cite[7.6/5]{blr}), it follows that it is smooth
(flatness being ensured by \cite[Proposition~6.1.5]{ega42}).
\end{proof}

It remains to give examples of covers~$\pi$ satisfying~($*$).

\begin{example}
\label{ex:biquadratic}
Condition~($*$) holds if $\pi:C\to \P^1_k$ is a connected Galois cover
 with branch locus equal to~$M$
and with Galois group $(\Z/2\Z)^2$.
Such covers exist if~$M$ consists of three rational points
(e.g.\ take $k(C)=k(t)\big(\sqrt{t(t+1)},\sqrt{t(t-1)}\big)$ if $M=\{-1,0,1\}$).
\end{example}

\begin{example}
\label{ex:twistedbiquadratic}
More generally, consider the algebraic group~$S$ over~$k$
defined as the kernel
of the norm map $N_{M/k}: R_{M/k}(\Z/2\Z) \to \Z/2\Z$.
This is a twisted form of $(\Z/2\Z)^2$.
Let~$t$ denote the parameter of~$\A^1_k$ and, assuming for simplicity that $M \subset \A^1_k$,
let $a \in H^0(M,\sO_M)$ denote the restriction of~$t$ to~$M$.
As~$M$ has degree~$3$ over~$k$, the exact sequence
\begin{align}
\xymatrix{
0 \ar[r] & S \ar[r] & R_{M/k}(\Z/2\Z) \ar[r] & \Z/2\Z \ar[r] & 0
}
\end{align}
induces an isomorphism
$H^1_{\et}(U,S) \isoto \Ker\big(N_{M/k}:H^1_{\et}(U \times_k M,\Z/2\Z) \to H^1_{\et}(U,\Z/2\Z)\big)$;
in particular, any invertible function on $U \times_k M$ whose norm down to~$U$ is a square defines
a class in $H^1_{\et}(U,S)$.  Applying this to the function $N_{M/k}(t-a)/(t-a)$,
we obtain the isomorphism
class of a torsor $C_U \to U$ under~$S$. Let $\pi:C \to \P^1_k$ be the cover obtained by compactifying
this torsor.  We claim that~($*$) holds and that~$C$ is geometrically connected over~$k$.
To check this, we may freely extend the scalars and thus assume
that~$k$ is algebraically closed and that $M=\{-1,0,1\}$,
in which case we can identify~$S$ with $(\Z/2\Z)^2$
and $\pi$ with the cover considered in Example~\ref{ex:biquadratic}.
\end{example}

\begin{rmks}
\label{rem:galois groups last example}
(i)
In the situation of Example~\ref{ex:twistedbiquadratic},
the Galois group of a Galois closure of the quartic extension
$k(C)/k(t)$, when viewed as a subgroup of~$S_4$, is $(\Z/2\Z)^2$ if~$M$ consists of three rational
points, or $D_4$ if~$M$ consists of a rational point and a quadratic point,
or~$A_4$ if~$M$ consists of a cubic point with cyclic residue field,
or else~$S_4$.

To see this, let~$k'$ be a minimal Galois extension of~$k$ that splits~$M$ completely
and $a_1, a_2, a_3$ be the values of~$t$ at the $k'$\nobreakdash-points of~$M$.
Let $p_i=\prod_{j \neq i}(t-a_j) \in k'[t]$ for every~$i$.
As the set $\{p_1,p_2,p_3\}$
is stable under $\Gal(k'/k)$ and as
 $k'(C)=k'(t)\big(\sqrt{p_1},\sqrt{p_2},\sqrt{p_3}\big)$,
the extension $k'(C)/k(t)$ is Galois.
Viewing $\Gal(k'/k)$ both as the quotient $\Gal(k'(t)/k(t))$ of $G=\Gal(k'(C)/k(t))$
and as its subgroup $\Gal(k'(C)/k(C))$, and noting that $k'(C)/k'(t)$ is biquadratic,
we find that
$G \simeq (\Z/2\Z)^2 \rtimes \Gal(k'/k)$, from which the claim follows easily.

(ii)
In the situation of Lemma~\ref{lem:fibres of f0}, even though~$f$ and~$X$ are not explicit,
it is possible to get a hold on the splitting behaviour of the fibres of~$f$
(rather than~$f^0$)
by considering the points of the generic fibre of~$f^0$ with values in complete discretely
valued fields with pseudo-algebraically closed residue field
(see \cite[Proposition~3.8]{ctdegenerescences}).
In this way, one can check that for every $m \in M$, the fibre~$f^{-1}(m)$ is split
if and only if~$b$ becomes a square in the field~$k'$ of
Remark~\ref{rem:galois groups last example}
(while it follows from the proof of Lemma~\ref{lem:fibres of f0}
that $(f^0)^{-1}(m)$ is split if and only if~$b$ becomes a square in~$k(m)$, a slightly stronger condition
in general).
Thus, in Examples~\ref{ex:biquadratic}
and~\ref{ex:twistedbiquadratic},
all of the fibres of~$f$ over~$M$ are truly non-split if~$b$ does not become a square in~$k'$.
\end{rmks}

\subsubsection{Further comments}

In all of the examples given in~\textsection\ref{subsec:examples},
the smooth fibres of the fibrations we construct are compactifications of torsors under algebraic tori.
Specifically, in Corollary~\ref{cor:norms},
the algebraic torus in question is $R^1_{L/\Q}\Gm$;
in Corollary~\ref{cor:norms-2}, it is the subtorus of $\Gm \times R_{L/\Q}\Gm$
defined by the equation $y^{3e}N_{L/k}(z)=1$;
and in Corollary~\ref{cor:norms-3},
the fibre over $c \in U$ is a compactification of a torsor under
the norm torus $R^1_{\pi^{-1}(c)/c}\Gm$.

Such torsors can have non-constant unramified Brauer classes and generally
fail to satisfy the Hasse principle or weak approximation.
Non-constant unramified Brauer classes in the fibres do exist,
in the case of Corollary~\ref{cor:norms}
(and therefore also of Corollary~\ref{cor:norms-2},
see
Remark~\ref{rmk:oncor2}),
when $E/\Q$ is a cyclic extension
and $L/\Q$ is a Galois extension with Galois group $G=\Z/3\Z \times \Z/3\Z$
that contains~$E$,
since in this case the norm torus~$T$ associated with~$L/\Q$ satisfies
$\Sha^2_{\cyc}(G,\widehat T)\neq 0$ by Remark~\ref{rk:brnrtrivial}~(ii).
(Recall that the injection~\eqref{eq:injectionbr1br0z}
is an isomorphism since~$k$ is a number field.)
In the case of Corollary~\ref{cor:norms-3},
the same phenomenon occurs
when~$M$ consists either of three rational points or of one cubic point with cyclic residue field,
according to Remark~\ref{rem:galois groups last example}~(i),
Remark~\ref{rk:brnrtrivial}~(ii) and Example~\ref{ex:abelian-alternating}.

Because of the presence of non-constant unramified Brauer classes in the fibres,
the various examples we have given are not covered by~\cite[Theorem~B]{ctskodescent}.

\bibliographystyle{amsalpha}
\bibliography{hww}
\end{document}